%% file: mss.tex
\documentclass[reqno,a4paper,oneside]{amsart}

\usepackage{uniform-kan-prelude}
\usepackage[qeds]{formatting-christian}

\newcommand{\m}{{\mathsf{m}}}
\newcommand{\unmarked}{{\mathsf{u}}}

\newcommand{\nd}{{\mathsf{nd}}}
\newcommand{\termaug}{{\mathsf{t}}}

\renewcommand{\fib}{{\mathsf{fib}}}

\newcommand{\M}{\widehat{\Delta}_\m}
\newcommand{\Mfib}{\widehat{\Delta}_{\m,\fib}}
\newcommand{\MS}{(\widehat{\Delta_+})_\m}
\newcommand{\MSfib}{(\widehat{\Delta_+})_{\m,\fib}}
\newcommand{\core}{\mathsf{core}}

\newcommand{\aug}{{\mathsf{aug}}}
\newcommand{\SemiCat}{\mathbf{SemiCat}}

\newcommand{\inv}{\operatorname{inv}}
\newcommand{\anod}{\mathsf{anod}}
\newcommand{\sd}{\operatorname{sd}}
\newcommand{\frames}{\operatorname{N}_{\operatorname{f}}}
\newcommand{\spans}{\operatorname{N}_{\operatorname{s}}}

\newcommand{\ev}{\operatorname{ev}}

\newcommand{\cel}{{\textstyle \int}}
\newcommand{\mcel}{{\textstyle{\int}\vphantom{\int}_{\hspace{-0.2em}\m}}}
\newcommand{\scel}{{\textstyle{\int}\vphantom{\int}_{\hspace{-0.2em}+}}}
\newcommand{\smcel}{{\textstyle{\int}\vphantom{\int}_{\hspace{-0.2em}\m,+}}}

\newcommand{\arrfib}{\twoheadrightarrow}
\newcommand{\arrcof}{\rightarrowtail}
\newcommand{\arrtrivfib}{\mathrel{\mathrlap{\hspace{1pt}\raisebox{6pt}{$\scriptscriptstyle\triv$}}\mathord{\twoheadrightarrow}}}

\newcommand{\arranod}{\mathrel{\mathrlap{\hspace{1pt}\raisebox{6pt}{$\scriptscriptstyle\anod$}}\mathord{\rightarrowtail}}}

\begin{document}

\title[{Constructive homotopy theory of marked semisimplicial sets}]{Constructive homotopy theory \\ of marked semisimplicial sets}

\author{Christian Sattler}

\begin{abstract}
We develop the homotopy theory of semisimplicial sets constructively and without reference to point-set topology to obtain a constructive model for $\omega$-groupoids.
Most of the development is folklore, but for a few results the author is unaware of previously known constructive proofs.
These include the statements that the unit of the free simplicial set adjunction is valued in weak equivalences and that the geometric product and cartesian product of fibrant semisimplicial sets are weakly equivalent.

We then extend the development to marked semisimplicial sets in order to obtain a constructive model for $(\omega, 1)$-categories.
\end{abstract}

\maketitle

\tableofcontents

\section{Introduction}

Semisimplicial sets are a well-known model for spaces~\cite{rourke-sanderson:delta-sets} (there called $\Delta$-sets).
We give a constructive development of their homotopy theory; the development should for most parts be folklore, but we take care to only use constructive and synthetic arguments.
Cornerstones are the facts that semisimplicial sets form a fibration category when restricted to fibrant objects (\cref{semisimplicial-sets-fib-cat}) as well as a cofibration category (\cref{semisimplicial-sets-cofibration-category}), weakly equivalent in several ways to the classical structures on simplicial sets once classical logic is assumed (\cref{summary-equiv}).
A constructive proof that the geometric product is weakly equivalent to the cartesian product on fibrant inputs is recorded in the dependency chain of \cref{cartesian-vs-geometric}.

\medskip

We then extend the development to marked semisimplicial sets, a surprisingly underexplored constructive model for higher categories (meaning $(\omega, 1)$-categories for us).
Marked (semi)simplicial sets are a variation on (semi)simplicial sets where every object comes equipped with a subset of its edges (including all degenerate edges in the simplicial case), called marked, that are to be regarded as invertible in the higher category presented by that object.
Markings of simplicial sets are used in~\cite{lurie:htt}, but there primarily for the purpose of modelling cartesian morphisms in cartesian fibrations.
In the context of presenting invertible cells of a higher category (possibly even $(\omega, \omega)$-category), markings (under the name of thin elements) of simplicial sets are used extensively by Verity~\cite{verity:complicial}.

Simplicial sets supports a nice model of higher categories (meaning $(\omega, 1)$-categories), namely quasicategories~\cite{joyal-quaderns,lurie:htt}, without markings, but this is not the case for semisimplicial sets: since the semisimplicial structure does not include degeneracies, we cannot obtain identity edges as in the quasicategorical setting.
Instead, we have to rely on fillings of one-dimensional horn inclusions to obtain sufficiently many marked edge that can then be used to produce edges that behave like identities.
This idea is explored in~\cite{harpaz:semi-segal} in the context of complete semi-Segal spaces, but spaces there are modelled by simplicial sets (although one would expect a possible replacement by semisimplicial sets), so the homotopy theory of simplicial sets still underlies that development.

This second role played by markings is visible also in another phenomenon.
Recall that quasicategories, fibrant objects in the Joyal model structure, are characterized by lifts against inner horns, whereas fibrations are furthermore required to be isofibration in the higher sense, \ie lift equivalences.
In marked semisimplicial sets, this difference vanishes: the isofibration lifting condition is already necessary for objects to obtain identities as explained above.

Mirroring the unmarked case, we show that marked semisimplicial sets form a fibration category when restricted to fibrant objects (\cref{semisimplicial-sets-fib-cat-marked}) as well as a cofibration category (\cref{semisimplicial-sets-cofibration-category-marked}), weakly equivalent to the Joyal model in simplicial sets once classical logic is assumed (\cref{summary-equiv-marked,summary-equiv-marked-composite}).
We also lift the weak equivalence of the geometric product and the cartesian product on fibrant inputs to the marked case (\cref{cartesian-vs-geometric}).

\medskip

This development of semisimplicial analogues of the homotopy theories of Kan complexes and quasicategories is motivated by the following points.

\begin{itemize}
\item
There are problems with simplicial homotopy theory in a constructive setting.
One usually defines the cofibrations as generated by boundary inclusions $\partial\Delta^n \to \Delta^n$.
Clasically, cofibrations coincide with monomorphisms, implying the highly desirable property that every object is cofibrant.
This is no longer provable constructively: cofibrancy of a simplicial set is equivalent to decidability of degeneracy of its elements, which in general requires using the axiom of excluded middle.
As a consequence, proofs by induction on the skeletal decomposition of a simplicial sets, pervasive in simplicial homotopy theory, no longer work constructively and must be restricted to cofibrant inputs.
Unfortunately, restricting ones attention to only the full subcategory of cofibrant simplicial sets is not a workable solution to these problems: many desirable constuctions cannot be shown to preserve cofibrancy, the core example being exponentials (this problem can already be seen in the full subcategory of sets when the exponent is infinite).
\item
There is an asymmetry in the simplicial model of $\omega$-groupoids with respect to nullary and binary composition.
Whereas binary composition amd associated coherence is modelled up to homotopy using lifting conditions, nullary compositions (\ie identities) and certain coherence conditions involving them are modelled strictly using degeneracies as part of the structure of presheaves over the simplex category.
This asymmetry is rectified in the semisimplicial model, where both cases are modelled up to homotopy.
\item
There are other proposed models of higher categories based on presheaves over a direct category, most prominently opetopic sets~\cite{baez-dolan:opetopes} (which seeks to model $(\omega, \omega)$-categories in general).
So far, the opetopic model, even restricted to the $\omega$-groupoidal and $(\omega, 1)$-categorical case, has not been related to better developed models such as Kan complexes and quasicategories.
As both models use presheaves over a direct category, it is relatively reasy to define change of shape comparison functors: in one direction, we have the semisimplicial subdivision of an opetope (given by the nerve of its category of elements), in the other direction, we have the opetopic subdivision of a simplex (this latter constuction is functorial only semisimplicially, not simplicially).
We expect the cocontinuous extensions of these maps to be inverse to each other up to weak equivalence, but leave this topic to future work.

Another model of $(\omega, \omega)$-categories is weak complicial sets~\cite{verity:complicial}.
The equivalence (using classical logic) of the homotopy theory of marked semisimplicial sets and the homotopy theory of quasicategories should lift analogously to an equivalence between a homotopy theory of stratified semisimplicial sets (having a notion of marked simplex in each positive dimensions) and the homotopy theory of weak complicial sets.
As such, one could hope to use the stratified semisimplicial model as an intermediate step in a conjectured equivalence between the opetopic and the weak complicial worlds.
\item
Semisimplicial sets have been gainfully employed in the theory of fibration categories~\cite{schwede:top-triang-categories}.
The key point there is that the restriction of the simplex category to injections is direct, so one has a notion of fibration category of Reedy fibrant semisimplicial object in a given fibration category.

Recently, Szumi{\l}o~\cite{szumilo:cofibration-categories} has given a direct construction associating to a given fibration category $\mathcal{C}$ a (finitely complete) quasicategory of frames $\frames \mathcal{C}$.
The combinatorial details are quite involved: already a vertex of $\frames \mathcal{C}$ is a Reedy fibrant semisimplicial object whose maps are weak equivalences; in general; the $n$-simplices of $\frames \mathcal{C}$ are generalizations of such semisimplicial objects to $n$-dimensional grids.
Szumi{\l}o observes that the complexity of the construction is due to the need for degeneracies.
If one takes as model for higher categories not quasicategories but fibrant marked semisimplicial sets, then this combinatorial complexity disappears.
The fibrant marked semisimplicial set $\spans \mathcal{C}$ associated to $\mathcal{C}$ has as vertices simply the objects of $\mathcal{C}$, as edges the Reedy fibrant spans whose left leg is a weak equivalence (with an edge being marked if also its right leg is a weak equivalence), and so on.
In fact, Szumi{\l}o's equivalence between fibration categories and finitely complete quasicategories factors via the equivalence of fibrant marked semisimplicial sets and quasicategories of~\cref{summary-equiv-marked} (with finite completeness restrictions on both sides).
This part of the construction also encapsulates the use of classical logic.
The remaining part, the equivalence between fibration categories and finitely complete fibrant marked semisimplicial sets, is combinatorially simpler and constructive (this will be the topic of a companion note).
\end{itemize}

This treatment of (marked) semisimplicial sets was developed and written independently to closely related results obtained in~\cite{henry:semisimplicial} (we thank Thierry Coquand for pointing out this recent preprint).

\section{Preliminaries}

Given a category $\mathcal{C}$, we write $\mathcal{C}^\to$ for its arrow category and $\widehat{\mathcal{C}}$ for the category of presheaves on $\mathcal{C}$.

\subsection{Weak factorization systems}

Let $\mathcal{C}$ be a category.
Given maps $l$ and $r$, we say that $l$ lifts against $r$ from the left (or $r$ lifts against $l$ from the right, the direction will always be clear from the context) if in any commuting square
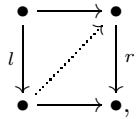
\[
\xymatrix{
  \bullet
  \ar[r]
  \ar[d]_{l}
&
  \bullet
  \ar[d]^{r}
\\
  \bullet
  \ar[r]
  \ar@{.>}[ur]
&
  \bullet
\rlap{,}}
\]
we have a diagonal filler as indicated.
In that case, we write $l \pitchfork r$.
We extend this notation to $L \pitchfork R$ for classes of maps $L$ and $R$ in the evident way.
Given a class of maps $L$, the right lifting closure $\liftr{L}$ denotes the class of maps that lift against $L$.
We have a dual notion of left lifting closure $\liftl{R}$ for a class of maps $R$.

From a constructive perspective, to avoid instances of the axiom of choice, it is important that the notation $L \pitchfork R$ is read as a lifting operation (that takes elements $l \in L$ and $r \in R$ and horizontal arrows forming a lifting problem as in the above diagram and returns the diagonal filler) rather than existence of a lift in any specific situation as in the above diagram.
We will however continue to use intuitive language; with this remark, it should be clear how to interpret statements constructively.

A \emph{weak factorization system} (wfs) $(L, R)$ on a category $\mathcal{C}$ consists of classes $L, R$ of maps in $\mathcal{C}$ such that:
\begin{itemize}
\item every map $f$ in $\mathcal{C}$ factors as $f = r l$ with $l \in L$ and $r \in R$,
\item $L \pitchfork R$,
\item $L$ and $R$ are closed under retracts.
\end{itemize}
One may equivalently replace the second and third conditions by the requiring that $L = \liftl{R}$ and $R = \liftr{L}$.

A weak factorization system $(L, R)$ is \emph{generated} by a set of maps $I$ if $R = \liftr{I}$.
In that case, we call $L$ the \emph{weak saturation} of $I$.
For cocomplete $\mathcal{C}$, the small object argument provides a way to generate a weak factorization system on a given set of maps $I$ provided that the domain of each map in $I$ is $\omega$-compact.
The class $L$ is then given by the retract closure of \emph{relative $I$-cell complexes}, $\omega$-compositions of pushouts of coproducts of maps in $I$.

\subsection{Cofibration categories}

We work with the notion cofibration category as defined in~\cite{szumilo:cofibration-categories}.
A \emph{cofibration category} is a category $\mathcal{E}$ with subcategories of \emph{weak equivalences} and \emph{cofibrations} containing all isomorphisms that satisfy the following properties:
\begin{enumerate}[label=(\arabic*)]
\item \label{cofib-cat:initial}
there is an initial object and every object $A$ is cofibrant, \ie $0 \to A$ is a cofibration,
\item \label{cofib-cat:pushout}
pushouts along cofibrations exist and cofibrations are stable under pushout,
\item \label{cofib-cat:pushout-triv-cofib}
trivial cofibrations, \ie cofibrations that are weak equivalences, are stable under pushout,
\item \label{cofib-cat:factorization}
every map factors as a cofibration followed by a weak equivalence,
\item \label{cofib-cat:2-out-of-6}
weak equivalences are closed under 2-out-of-6.
\end{enumerate}
An \emph{exact functor} between cofibration categories is a functor of underlying categories that preserves weak equivalences, cofibrations, initial objects, and pushouts along cofibrations.
An exact functor is called a \emph{weak equivalence} if it induces an equivalence upon taking homotopy categories.
We refer to~\cite{radulescu:cofibration-categories} for an overview of the theory of cofibration categories.

As shown in~\cite{szumilo:cofibration-categories,kapulkin-szumilo:quasicat-of-cof-cat}, cofibration categories provide a model for finitely cocomplete homotopy theories (which we take as a synonym for higher categories); the higher category presented by a cofibration category in this way coincides the one presented by the underlying category with weak equivalences.

We have an evident dual notion of fibration category, modelling finitely complete higher categories.

\subsection{Leibniz construction}

We will frequently make use of the theory of Leibniz constructions as discussed in~\cite[Section~4]{riehl-verity:reedy}.
This is also known as the Joyal-Tierney calculus, or the calculus of pushout (monoidal) products and pullback (monoidal) exponentials.
Explicitly, the \emph{Leibniz construction} of a functor $F \co \mathcal{C} \times \mathcal{D} \to \mathcal{E}$ with $\mathcal{E}$ finitely cocomplete is the functor $\widehat{F} \co \mathcal{C}^\to \times \mathcal{D}^\to \to \mathcal{E}^\to$ that sends $f \co A \to B$ in $\mathcal{C}$ and $g \co C \to D$ in $\mathcal{D}$ to the induced map from the pushout to the bottom right object in the diagram
\[
\xymatrix{
  F(A, C)
  \ar[r]^-{F(A, g)}
  \ar[d]_{F(f, C)}
&
  F(A, D)
  \ar[d]^{F(f, D)}
\\
  F(B, C)
  \ar[r]^-{F(B, g}
&
  F(B, D)
\rlap{,}}
\]
\ie the map
\[
\widehat{F}(f, g) : F(A, D) +_{F(A, C)} F(B, C) \to F(B, D)
.\]
For a (monoidal) product $\otimes$ on a finitely cocomplete category $\mathcal{C}$, one obtains the associated Leibniz or pushout (monoidal) product $\widehat{\otimes}$.
For a (monoidal) exponential $\exp \co \mathcal{C}^\op \times \mathcal{C} \to \mathcal{C}$ with $\mathcal{C}$ finitely complete, one would apply the above construction to the opposite functor $\exp \co \mathcal{C} \times \mathcal{C}^\op \to \mathcal{C}^\op$ to obtain the associated Leibniz or pullback (monoidal) exponential $\widehat{\exp}$.
The above binary version of the Leibniz construction has $n$-ary analogues for finite $n$, for example the nullary Leibniz construction takes an object $X \in \mathcal{E}$ and returns the arrow $0 \to X$.
For finitely cocomplete $\mathcal{E}$, this provides a way of lifting a (symmetric) monoidal structure on $\mathcal{E}$ to an associated Leibniz (symmetric) monoidal structure on $\mathcal{E}^\to$; if $\mathcal{E}$ is additionally finitely complete, then the (right or left) closure of the original monoidal structure lifts to the Leibniz monoidal structure.

The Leibniz construction enjoys various formal properties such as preservation of argumentwise cocontinuity and preservation of adjunctions.
In particular, the Leibniz construction of an argumentwise cocontinuous functor preserves weak saturation in the sense of weak factorization systems in each argument.
Properties such as these will be made use of frequently in our development.

\subsection{Kan complexes}

We refer to~\cite{goerss-jardine} for a modern development of the homotopy theory of simplicial sets.
We will not assume any knowledge about Kan complexes except in the classical parts of~\cref{subsection:comparison-kan-complexes}.

\subsection{Quasicategories}

Our development of the homotopy theory of marked semisimplicial sets follows quite closely the development of the homotopy theory of quasicategories by Joyal~\cite{joyal-quaderns}.
We refer to~\cite{lurie:htt} for a thorough development of the theory of quasicategories, there called $\infty$-categories, but will not assume any knowledge about quasicategories in our development (also not wishing to bring in any hidden sources of non-constructivity) except in the classical parts of~\cref{subsection:comparison-quasicategories}.

\section{Semisimplicial sets}

\subsection{Basic notions}

\subsubsection{Categories and semicategories}

We write $\Cat$ for the category of categories.
We write $\SemiCat$ for the category of semicategories.
The forgetful functor $\Cat \to \SemiCat$ has both a left and right adjoints, the free and cofree category on a semicategory functor.
The free category on a semicategory $\mathcal{C}$ has the same objects and as morphisms the disjoint union of the morphisms of $\mathcal{C}$ together with fresh identities for every object.
The cofree category on semicategory $\mathcal{C}$ has as objects pairs $(A, u)$ where $A \in \mathcal{C}$ and $u$ is an idempotent on $A$ and as morphisms $(A, u) \to (B, v)$ maps $f \co A \to B$ such that $fu = f = vf$; the identity on $(A, u)$ is given by $u$ and composition is inherited from $\mathcal{C}$.

Let for this paragraph $L$ denote the free category on a semicategory functor.
Given a semicategory $\mathcal{C}$ and an object $X \in \mathcal{C}$, note that the slice $L\mathcal{C}_{/X}$ is freely generated by a semicategory.
This semicategory is called the \emph{categorical slice} $\mathcal{C}_{/X}$ (to distinguish it from the ordinary semicategorical slice); its formation is functorial in $X$.
One can compute that $\mathcal{C}_{/X}$ is the semicategorical slice joined with $\top$.
It can also be defined via the pullback
\[
\xymatrix{
  \mathcal{C}_{/X}
  \ar[r]
  \ar[d]
  \pullbackcorner{dr}
&
  U(L\mathcal{C}_{/LX})
  \ar[d]
\\
  \mathcal{C}
  \ar[r]
&
  UL\mathcal{C}
\rlap{.}}
\]
The \emph{categorical coslice} $\mathcal{C}_{\backslash X}$ is defined dually.

There is a \emph{geometric symmetric monoidal structure} $(\top, \otimes)$ on $\SemiCat$ where $\top$ is the semicategory with one object and no morphisms and $A \otimes B$ is the semicategory with objects pairs of objects of $A$ and $B$ and a morphism $(X, Y) \to (X', Y')$ given by either a pair of morphisms $X \to X'$ and $Y \to Y'$, just a morphism $X' \to X$, or just a morphism $Y' \to Y$, with composition defined in the evident manner.
The free category functors lifts to a symmetric monoidal functor between the geometric symmetric monoidal structure on $\SemiCat$ and the cartesian monoidal structure on $\Cat$.

\subsubsection{Semisimplicial sets}

We write $\Delta$ for the simplex category.
The augmented simplex category is denoted $\Delta_\aug$.
Both are Reedy categories.
Given a Reedy category $\mathcal{R}$, we write $\mathcal{R}_+$ and $\mathcal{R}_-$ for the direct and inverse subcategories; in diagrams, we will write arrows in $\mathcal{R}_+$ as $\arrcof$ and arrows in $\mathcal{R}_-$ as $\arrfib$ (not to be confused with cofibrations or fibrations in a different context).
In particular, we write $\Delta_+$ ($\Delta_{\aug,+}$) for the (augmented) semisimplex category, the restriction of $\Delta$ ($\Delta_\aug)$ to monomorphisms.
Note that our notation differs from another common convention, which instead denotes $\Delta_+$ the augmented simplex category.

\emph{Semisimplical sets} $\widehat{\Delta_+}$ are given by presheaves over semisimplices.
Given a semisimplicial set $X$, we adopt some shorthand notation for low-dimensional elements with specified boundary.
For $a, b \in X_0$, we write $X_1(a, b)$ for the set of edges $f$ from $a$ to $b$, \ie $f d_1 = a$ and $f d_0 = b$.
For $a, b, c \in X_1$ and $f \in X_1(a, b), g \in X_1(b, c), h \in X_1(a, c)$, we write $X_2(f, g, h)$ for the set of triangles $\sigma$ with boundary $f, g, h$, \ie $\sigma d_2 = f$, $\sigma d_0 = g$, and $\sigma d_1 = h$.

The inclusion $\Delta_+ \to \Delta$ gives rise to a triple of adjoint functors: every simplicial set has an underlying semisimplicial set obtained by forgetting degeneracies and every semisimplicial set induces a free and a cofree simplicial set on it.

\medskip

Recall the \emph{geometric symmetric monoidal structure} $(\top, \otimes)$ on semisimplicial sets.
It is induced by the geometric symmetric monoidal structure of semicategories, but can also be described as the cartesian monoidal structure of simplicial sets with results restricted to non-degenerate elements.
Explicitly, we have $\top = y[0]$ and an $n$-simplex of $A \otimes B$ consists of a pair of jointly monic epimorphisms $p \co [n] \to [a]$ and $q \co [n] \to [b]$ in $\Delta_+$ together with $a \in A_a$ and $b \in B_b$; the functorial action of a monomorphism $f \co [n'] \to [n]$ on this data is given by Reedy factorizing $p f = f_0 p'$ and $q f = f_1 q'$ and returning the pair of jointly monic epimorphisms $(p', q')$ together with $(a[f_0], b[f_1])$.
The geometric symmetric monoidal structure is closed; we denote its internal hom by $\hom(-, -)$.
The free simplicial set functor lifts to a symmetric monoidal functor.

\begin{lemma} \label{tensor-as-right-adjoint}
The geometric product functor is parametric right adjoint.
\end{lemma}

\begin{proof}
This holds by construction.
Let
\[
S \co \textstyle\int (1 \otimes 1) \to \widehat{\Delta_+} \times \widehat{\Delta_+}
\]
that sends an element of $1 \otimes 1$, corresponding to a jointly monic pair of epimorphisms $[n] \to [a]$ and $[n] \to [b]$, to $([a], [b])$.
By construction of the geometric product, we have an isomorphism
\[
\widehat{\Delta_+}_{/(1 \otimes 1)}((y[n], f), A \otimes B) \simeq (\widehat{\Delta_+} \times \widehat{\Delta_+})(S([n], f), (A, B))
.\]
natural in $([n], f) \in \int (1 \otimes 1)$ and $A, B \in \widehat{\Delta_+}$.
This shows that $\otimes$ factors as
\[
\xymatrix{
  \widehat{\Delta_+} \times \widehat{\Delta_+}
  \ar[r]
&
  (\widehat{\Delta_+})_{/1 \otimes 1}
  \ar[r]
&
  \widehat{\Delta_+}
}
\]
where the first arrow is right adjoint to the cocontinuous extension of $S$ and the second arrow is the forgetful functor.
\end{proof}

\medskip

Recall the simplicial nerve $N$, arising from the fully faithful inclusion of $\Delta$ into categories $\Cat$.
There is also a semisimplicial nerve $N$, arising from the fully faithful inclusion of $\Delta_+$ into semicategories $\SemiCat$.
Observe that the diagram of inclusions
\[
\xymatrix{
  \Delta_+
  \ar[r]
  \ar[d]
&
  \SemiCat
  \ar[d]
\\
  \Delta
  \ar[r]
&
  \Cat
}
\]
commutes strictly.
Using left Kan extension, we obtain a diagram
\[
\xymatrix{
  \widehat{\Delta_+}
  \ar[r]
  \ar[d]
&
  \SemiCat
  \ar[d]
\\
  \widehat{\Delta}
  \ar[r]
&
  \Cat
}
\]
commuting up to canonical natural isomorphism.
The bottom functor is a monoidal functor between cartesian symmetric monoidal structures, and the top functor lifts to a symmetric monoidal functor between the geometric symmetric monoidal structures.
By adjointness, we obtain a diagram
\[
\xymatrix{
  \Cat
  \ar[r]^{N}
  \ar[d]
&
  \widehat{\Delta}
  \ar[d]
\\
  \SemiCat
  \ar[r]^-{N}
&
  \widehat{\Delta_+}
\rlap{,}}
\]
which commutes up to canonical natural isomorphism; the vertical maps denote forgetful functors.
By adjointness, the bottom functor is lax symmetric monoidal; it is in fact symmetric monoidal.
As a horizontal and vertical transpose of the penultimate diagram, the diagram
\[
\xymatrix{
  \SemiCat
  \ar[r]^-{N}
  \ar[d]
&
  \widehat{\Delta_+}
  \ar[d]
\\
  \Cat
  \ar[r]^-{N}
&
  \widehat{\Delta}
}
\]
commutes up to canonical natural isomorphism.

\subsubsection{Join and slice}

Recall the monoidal structure of $\Delta_\aug$ given by the join of finite total posets.
By Day convolution, it induces a closed monoidal structure on augmented simplicial sets.
Via the forgetful functor induced by $\Delta \to \Delta_\aug$ and its right adjoint, it induces the \emph{join monoidal structure} $(0, \star)$ on simplicial sets.
Although not closed, it preserves connected colimits in each argument.
Given a simplicial set $A$, the functors
\[
- \star A, A \star - \co \widehat{\Delta} \to \widehat{\Delta}_{\backslash A}
\]
admit right adjoints.
Their respective action on an object $p \co A \to X$ is written $X_{/p}$ and $X_{\backslash p}$.

Note that the monoidal structure of $\Delta_\aug$ descends to $\Delta_{\aug,+}$.
Thus, we analogously obtain a join monoidal structure on semisimplicial sets with analogous properties.
It can equivalently be described as induced by the join monoidal structure of free simplicial sets with results restricted to non-degenerate elements.
We use the notation introduced above also in that setting.
By construction, the forgetful functor from simplicial to semisimplicial sets is monoidal with respect to the join.
The induced oplax monoidal structure of the free simplicial set functor is a monoidal structure.

\begin{lemma} \label{join-as-right-adjoint}
The join functor is a parametric right adjoint.
Concretely, it is the composite
\[
\xymatrix{
  (\widehat{\Delta_+})_{/(1 + 1)}
  \ar[r]
&
  (\widehat{\Delta_+})_{/(1 \star 1)}
  \ar[r]
&
  \widehat{\Delta_+}
}
\]
where the first map is pushforward (right adjoint to pullback) along $1 + 1 \to 1 \star 1$ and the second map is the forgetful functor.
\end{lemma}

\begin{proof}
Same as in simplicial sets.

Denote $\iota_0, \iota_1 \co 1 \to 1 \star 1$ the evident inclusions.
Given semisimplicial sets $A$ and $B$, let us check that $A \star B \to 1 \star 1$ is terminal in the category of maps $C \to 1 \star 1$ with maps $h_0 \co \iota_0^* C \to A$ and $h_1 \co \iota_1^* C \to B$.
Given such an object, we have to show there is a unique map $h \co C \to A \star B$ over $1 \star 1$ making the triangles
\begin{align*}
\xymatrix@C-1em{
  \iota_0^* C
  \ar[rr]^-{\iota_0^* h}
  \ar[dr]
&&
  \iota_0^*(A \star B)
  \ar[dl]^{\simeq}
\\&
  A
\rlap{,}}
&&
\xymatrix@C-1em{
  \iota_1^* C
  \ar[rr]^-{\iota_1^* h}
  \ar[dr]
&&
  \iota_1^*(A \star B)
  \ar[dl]^{\simeq}
\\&
  B
}
\end{align*}
commute.
By cocontinuity, it will suffice to look at the representable case $C = \Delta^n$.
The given map $C \to 1 \star 1$ corresponds to a decomposition $[n] \simeq [a] \star [b]$; denote the associated inclusions $f \co [a] \to [n]$ and $g \co [b] \to [n]$.
Then $\iota_0^* C \simeq \Delta^a$ and $\iota_1^* C \simeq \Delta^b$.
By naturality, we are forced to let $h$ be the element of $A \star B$ given by $[n] \simeq [a] \star [b]$ with $h_0 \in A_a$ and $h_1 \in B_b$.
\end{proof}

\begin{remark}
The above proof should be reduced to some abstract nonsense of Day convolution monoidal products.
\end{remark}

\subsection{Weak factorization systems}

Let us recall some terminology in semisimplicial sets.

We write $i^n \co \partial\Delta^n \to \Delta^n$ for the \emph{boundary inclusion} of dimension $n$.
The \emph{cofibrations} are the levelwise decidable monomorphisms.
In diagrams, arrows that denote cofibrations will be written as $\arrcof$.
Cofibrations are generated by boundary inclusions under weak saturation, concretely write as $\omega$-compositions of pushouts of coproducts of boundary inclusions.
Note that every object is cofibrant.

\begin{lemma} \label{cof-tensor}
Cofibrations are closed under the Leibniz geometric symmetric monoidal structure.
\end{lemma}

\begin{proof}
The unit $0 \to \top$ of the Leibniz geometric monoidal structure is clearly a cofibration.
It remains to discuss the binary Leibniz monoidal product.
Here, we could appeal to weak saturation and explicitly exhibit the Leibniz geometric product of two boundary inclusions as a relative cell complex of boundary inclusions in the same manner as in simplicial sets.
We will instead sketch a less generator-focussed approach.

Given maps $f \co A \to B$ and $g \co C \to D$, the square
\[
\xymatrix{
  A \otimes C
  \ar[r]
  \ar[d]
&
  A \otimes D
  \ar[d]
\\
  B \otimes C
  \ar[r]
&
  B \otimes D
}
\]
is a pullback by \cref{tensor-as-right-adjoint}.
If $f$ and $g$ are cofibrations, it follows that $f \otimes g$ coincides with the binary union of $f \otimes D$ and $B \otimes g$ over $B \otimes D$.
Binary union clearly preserves cofibrations, so without loss of generality it remains to show that $f \otimes D$ is a cofibration if $f$ is.
This is clear from the definition of the geometric product.
\end{proof}

\begin{lemma} \label{cof-join}
Cofibrations are closed under the Leibniz join monoidal structure.
\end{lemma}

\begin{proof}
The Leibniz join of a boundary inclusion with a boundary inclusion is a boundary inclusion.
\end{proof}

For $n \geq 1$, we write $h_k^n \co \Lambda_k^n \to \Delta^n$ for the $k$-th \emph{horn inclusion} of dimension $n$.
The horn inclusions generate the weak factorization system (\emph{anodyne}, \emph{fibration}).
In diagrams, arrows that denote anodyne maps will be written as $\arranod$; arrows that denote fibration will be written as $\arrfib$.

\begin{remark} \label{fibrations-arbitrary}
In semisimplicial sets, there is more choice for the notion of fibration than in simplicial sets.
One could demand that fibrations lift against additional ``primitive'' cofibrations that are weak equivalences: an example would be the Dunce hat inclusion of \cref{dunce-hat}.
In simplicial sets, these inclusions would (at least classically) already lie in the weak saturation of horn inclusions, but here the class of fibrations would become more restrictive.
However, as per \cref{cof-weak-equiv-lift-fib-between-fibrant}, the fibrant objects and fibrations between fibrant objects would not change.
As such, one should regard the classes of fibrant objects and fibrations between fibrant objects as more canonical than the class of fibrations, the latter being rather arbitrary for non-fibrant codomain.
\end{remark}

\begin{lemma} \label{anodyne-tensor-cof}
Anodyne maps are closed under Leibniz geometric product with cofibrations.
\end{lemma}

\begin{proof}
The proof by Joyal~\cite[Theorem~H.0.20]{joyal-quaderns} works just as well for semisimplicial sets.
It uses ordered simplicial complexes, which embed fully faithfully into semisimplicial sets.

For completeness, let us include a proof nonetheless.
Unfortunately, it is not very abstract.
We will show that the pushout product of $\Lambda^m_k \to \Delta^m$ and $\partial\Delta^n \to \Delta^n$ is anodyne for $0 \leq k \leq m$ with $m \geq 1$ and $n \geq 0$.

Note that $\Delta^m \otimes \Delta^n$ is the nerve of the poset $[m] \otimes [n]$.
The given pushout product corresponds to the subobject $S \subseteq N([m] \otimes [n])$ of those simplices $\langle f_0, f_1\rangle \co [t] \to [m] \times [n]$ for which $\im(f_0) \cup {k} \neq [m]$ or $\im(f_1) \neq [n]$.
We write $[m] = [m_0] \star [0] \star [m_1]$ where $m_0, m_1 \geq -1$.
Note that $m_0$ and $m_1$ cannot both be $-1$ since $m \geq 1$.

Given monomorphisms
\begin{align*}
\langle u_0, v_0\rangle &\co [a_0] \to [m_0] \times [n]
,\\
\langle u_1, v_1\rangle &\co [a_1] \to [m_1] \times [n]
\end{align*}
with $a_0, a_1 \geq -1$, let $M_{\langle u_0, v_0\rangle, \langle u_1, v_1\rangle}$ be the subobject of $N([m] \otimes [n])$ of those simplices that are a face of
\begin{equation} \label{anodyne-join-cof:0}
((u_0 \star \const(0) \star u_1), [v_0, v, v_1]) \co [a_0] \star [a] \star [a_1] \to [m] \times [n]
\end{equation}
for some monomorphism $v \co [a] \to [n]$ making the map order-preserving.

Note that $S \cup M_{\langle u_0, v_0\rangle, \langle u_1, v_1\rangle}$ forms a finite diagram $D$ of subobjects of $N([m] \otimes [n])$ above $S$ indexed by categories of elements of nerves:
\begin{align*}
\langle u_0, v_0\rangle &\in \cel N([m_0] \otimes [n])
,\\
\langle u_1, v_1\rangle &\in \cel N([m_1] \otimes [n])
.\end{align*}
Every simplex of $N([m] \otimes [n])$ is contained in one of these subobject, so the colimit of this diagram is the entirety of $N([m] \otimes [n])$.
It will thus suffice to verify that the diagram is Reedy anodyne.

Let $\langle u_0, v_0\rangle$ and $\langle u_1, v_1\rangle$ be as above.
We will verify that the latching object inclusion of $D$ at this index is anodyne.
If $u_0$ and $u_1$ are not both surjective, then $M_{\langle u_0, v_0\rangle, \langle u_1, v_1\rangle} \subseteq S$ and the latching object inclusion is the identity at $S$.
Assume now that they are both surjective.
Let $v \co [a] \to [n]$ be mono with $[v]$ maximal making~\eqref{anodyne-join-cof:0} order-preserving (if such $v$ does not exist, then the latching object inclusion is an identity).
Note that $[a]$ is non-empty, $v$ is distance-preserving, and all elements of $M_{\langle u_0, v_0\rangle, \langle u_1, v_1\rangle}$ are faces of~\eqref{anodyne-join-cof:0}.
Removing a point in $[a_0]$ or $[a_1]$ will result in a face included in the latching object.
Let $E$ be the subposet of $i \in [a]$ such that $v(i) \in \im(v_0) \cup \im(v_1)$.
Note that $E$ can only contain endpoints of $[a]$, \ie $E \subseteq \braces{0, a}$.
Removing a point in $[a]$ not in $E$ will result in a face included in $S$, hence also the latching object.
Removing the points in $E$ yields a face not in the latching object.
This shows that the latching object is a pushout of the generalized horn incluion
\begin{equation} \label{anodyne-join-cof:1}
i^{[a_0]} \hatjoin \Delta^E \hatjoin i^{[a_1]} \co \Lambda_E^{[a_0] \star [a] \star [a_1]} \to \Delta^{[a_0] \star [a] \star [a_1]}
\end{equation}
(where the Leibniz join is either taken in augmented semisimplicial sets or the corresponding argument is removed if $a_0 = -1$ or $a_1 = -1$), which is anodyne by \cref{anodyne-tensor-cof}.

It remains to show that~\eqref{anodyne-join-cof:1} is anodyne.
For this, we only have to observe that $E$ is a proper subset of $[a_0] \star [a] \star [a_1]$ since $m_0 \geq 0$ or $m_1 \geq 0$ and $u_0, u_1$ are surjective, hence $a_0 \geq 0$ or $a_1 \geq 0$.
\end{proof}

\begin{remark}
In contrast to simplicial sets, higher horn inclusions do not write as retracts of Leibniz products of a one dimensional horn inclusions with a boundary inclusion.
In fact, these maps and their Leibniz product are nerves of direct semicategories and hence do not admit any non-trivial retracts.
\end{remark}

\begin{lemma} \label{anodyne-join-cof}
Anodyne maps are closed under Leibniz join with cofibrations.
Anodyne maps are closed under join with semisimplicial sets.
\end{lemma}

\begin{proof}
Both claims are unified in augmented semisimplicial sets where join with a semisimplicial set $A$ is Leibniz join with $0 \to A$ (with $A$ terminally augmented).
The Leibniz join of a horn inclusion with a boundary inclusion is a horn inclusion.
\end{proof}

\begin{lemma} \label{point-join-cofib}
Join with $\Delta^0$ sends cofibrations to anodyne maps.
\end{lemma}

\begin{proof}
In augmented semisimplicial sets, join with the ``0-dimesional horn inclusion'' $0 \to \Delta^0$ sends boundary inclusions to horn inclusions of positive dimension.
\end{proof}

\begin{corollary} \label{yoneda-anodyne}
The Yoneda functor is valued in anodyne maps.
\qed
\end{corollary}

We have corresponding adjoint closure statements for fibrations that we will use implicitly from now on.

A map between fibrant objects that lifts against cofibrations is called a \emph{trivial fibration}, written in diagrams as $\arrtrivfib$.
General maps lifting against cofibrations do not deserve to be called trivial fibration as can be seen in the case of $\Delta^0 + \Delta^0 \to \Delta^0$.

\subsection{Fundamental groupoid}

For this subsection, fix a fibrant semisimplicial set $X$.

\medskip

Let us write $L$ for the \emph{representable loop}, the quotient of $\Delta^1$ that identifies its vertices.
Note that the unique map $p \co \Delta^0 \to L$ is a cofibration.
A loop $l \co L \to X$ is called \emph{constant} on $x \defeq lp$ if it lifts to a map $\Delta^0 \star L \to X$, \ie if $X_{/l}$ contains a vertex.
Note that $X_{/x}$ is trivially fibrant and that $X_{/l} \to X_{/x}$ is a fibration.
Thus, if $l$ is constant then $X_{/l} \to X_{/x}$ is surjective on vertices.

Since $\Delta^0 \to 0 \star \Delta^0 \to \Delta^0 \star L$ is anodyne, the induced map $\hom(\Delta^0 \star L, X) \to X$ is a trivial fibration.
Thus, the constant loops on a given point are essentially unique in the sense of forming the vertices of a trivially fibrant object.
In particular, they can always be produced.

Note that we have anodyne maps as indicated:
\[
\xymatrix@R-0.5cm@C-0.5cm{
&
  L
  \ar@{>->}[dl]
  \ar@{>->}[dr]
\\
  \Delta^0 \star L
  \ar@{>->}[dr]_(0.4){\anod}
&&
  L \star \Delta^0
  \ar@{>->}[dl]^(0.4){\anod}
\\&
  \Delta^0 \star L \star \Delta^0
\rlap{.}}
\]
Given an extension of a loop $l \co L \to X$ to a constant loop $\Delta^0 \star L \to X$, we can thus also extend $l$ to a map $L \star \Delta^0 \to X$.
This shows that the notion of constant loop is invariant under taking opposites.

\medskip

Let us write $P \defeq \Delta^1 +_{\partial\Delta^1} \Delta^1$ for the \emph{representable parallel pair of edges}.
We say that parallel edges $f, g \in X_1(a, b)$ forming a map $[f, g] \co P \to X$ are \emph{related}, written $f \approx g$, if $P \to X$ lifts to a map $\Delta^0 \star P \to X$, \ie if $X_{/[f, g]}$ contains a vertex.
As for constant loops, one sees that $X_{/[f, g]} \to X_{/f}$ and $X_{/[f, g]} \to X_{/g}$ are surjective on vertices if $f$ and $g$ are related and that the notion of relatedness is invariant under taking opposites.

Using the unique map $P \to \Delta^1$ and that $X_{/f}$ is trivially fibrant for $f \in X_1$, one sees that relatedness is reflexive.
Using the unique non-identity automorphism on $P$, one sees that relatedness is symmetric.
Given $f, g, h \in X_1(a, b)$ with $f \approx g$ and $g \approx h$, then $X_{/[f, g]} \to X_{/g}$ and $X_{/[g, h]} \to X_{/g}$ are surjective on vertices.
Since $X_{/g}$ is trivially fibrant, one obtains vertices of $X_{/[f, g]}$ and $X_{/[g, h]}$ that agree on $X_{/g}$, \ie in particular a vertex of $X_{[f, h]}$.
This shows that relatedness is transitive.
Altogether, relatedness is an equivalence relation.

Note that $P$ is the cokernel of the cofibration $\partial\Delta^1 \to \Delta^1$.
Let $Q$ be the cokernel of the cofibration $[0, d_0] \co \Delta^0 + \Delta^1 \to \Delta^2$.
The map from $\partial\Delta^1 \to \Delta^1$ to $\Delta^0 + \Delta^1 \to \Delta^2$ given on codomains by $d_2$ is Reedy anodyne.
We thus obtain an anodyne map $P \to Q$.
Given triangles $\sigma_0 \in X_2(f_0, g, h_0)$ and $\sigma_1 \in X_2(f_1, g, h_1)$, we form a map $[\sigma_0, \sigma_1] \co Q \to X$ and then have a span
\[
\xymatrix@C-1cm{
&
  X_{/[\sigma_0, \sigma_1]}
  \ar@{->>}[dl]_{\triv}
  \ar@{->>}[dr]
\\
  X_{/[f_0, f_1]}
&&
  X_{/[h_0, h_1]}
\rlap{.}}
\]
It follows that $f_0 \approx f_1$ implies $h_0 \approx h_1$.
We obtain analogous statements for triangles that share the $d_1$- or $d_2$-indexed edge.
Combining these using transitivity, given triangles $\sigma_0 \in X_2(f_0, g_0, h_0)$ and $\sigma_1 \in X_2(f_1, g_1, h_1)$, if two of the pairs $(f_0, f_1), (g_0, g_1), (h_0, h_1)$ are in the relatedness relation, then so is the third.

The \emph{fundamental groupoid} $\tau_1 X$ is defined as follows.
The objects are given by vertices $X_0$ and the morphisms from $a$ to $b$ are given by the quotient $X_1(a, b)/{\approx}$ of edges under relatedness.
Composition is given by filling of $h^2_1$-horns.
This is invariant under choice of representatives and filling by the above.
Composition is associative by fillings of $h^3_1$-horns.
The identity on $a \in X_0$ is given by a constant loop $l \co L \to X$ on $a$.
Right neutrality holds since $X_{/l} \to X_{/x}$ is surjective on vertices; left neutrality holds by duality.
Left and right inverses for a given edge are constructed using $h^2_0$- and $h^2_1$-horn filling.
This makes $\tau_1 X$ into a groupoid.

\subsection{Enrichment in groupoids}

We have the constant functor from sets to semisimplicial sets.
Its left adjoint is the \emph{fundamental set} functor $\tau_0$.
Given a semisimplicial set $A$, one may describe $\tau_0 A$ as the set $A_0$ quotiented under the edge relation.
It thus depends only on the truncation of $A$ to two levels.
Note that if $A$ is fibrant, the edge relation is already an equivalence relation.

Note that $\tau_0$ lifts to a symmetric monoidal functor
\[
\tau_0 \co (\widehat{\Delta_+}, \top, \otimes) \to (\Set, 1, \times)
.\]
We thus can regard semisimplicial sets as enriched in sets (not to be confused with the original hom-sets) with the hom-set from $A$ to $B$ given by $\tau_0(\hom(A, B))$.
When restricted to fibrant objects, this will also be called the \emph{homotopy category} of semisimplicial sets (we will see later that this coincides with the localization of semisimplicial sets at a certain class of weak equivalences).

\medskip

By precomposing with forgetful functors, we obtain a nerve
\[
\xymatrix{
  \Gpd
  \ar[r]
&
  \Cat
  \ar[r]
&
  \SemiCat
  \ar[r]
&
  \widehat{\Delta_+}
}
\]
from groupoids.
Its left adjoint is the \emph{fundamental groupoid} functor $\tau_1$.
Given a semisimplicial set $A$, one may describe $\tau_1 A$ as the free groupoid on the graph $A_1 \rightrightarrows A_0$ quotiented under the relation on morphisms given by triangles $A_2$.
It thus depends only on the truncation of $A$ to three levels.
Note that $\tau_0$ is given by composing $\tau_1$ with the connected components functor.

If $A$ is fibrant, we see that $\tau_1 A$ as defined here coincides up to isomorphism with $\tau_1 A$ as defined in the previous subsection.
That construction has the important property that it produces, for some chosen notion of smallness, a locally small groupoid if $X_1(a, b)$ is small for all $a, b \in X_0$.
This is contrary to the generic case, which would require $X_0$ and $X_1$ to be small.

Since the free groupoid functor on categories preserves finite products, we have that $\tau_1$ lifts to a symmetric monoidal functor
\[
\tau_1 \co (\widehat{\Delta_+}, \top, \otimes) \to (\Gpd, 1, \times)
.\]
We thus can regard semisimplicial sets as enriched in groupoids with the hom-groupoid from $A$ to $B$ given by $\tau_1(\hom(A, B))$ (note that $\hom(A, B)$ is fibrant if $B$ is).
When restricted to to fibrant objects, we will also call this the \emph{homotopy $(2, 1)$-category} of semisimplicial sets.
In order to avoid excessive overloading of notation, identities, composition, and inverses in hom-groupoids will be written $\id_2$, $- \circ_2 -$, and $\inv_2$.
The homotopy $(2, 1)$-category plays a role similar to the 2-category of quasicategories of~\cite{riehl-verity:2-cat-of-quasicat} (in the marked case, we will obtain a homotopy 2-category that is even more analogous and could be used for formally developing the theory of semisimplicial quasicategories, although we will not do so here).

\medskip

In the category of groupoids, let us call cofibrations the functors whose action on objects is a decidable monomorphism, fibrations the isofibrations, and weak equivalences the equivalences.
Recall that this forms a model category.

We have that $\tau_1$ maps cofibrations and anodyne maps of semisimplicial sets to cofibrations and trivial cofibrations of groupoids, respectively: by cocontinuity, it suffices to check this for the respective generators; for these, it is clear.

We also have that $\tau_1$ maps fibrations of semisimplicial sets to fibration of groupoids: given a fibration $Y \to X$ of semisimplicial sets, a morphism $a \to b$ in $\tau_1(X)$ and a lift $y$ of $b$ to $\tau_1(Y)$, recall that $a \to b$ is represented by a zig-zag of edges between the vertices $a$ and $b$ of $X$; using $h^1_0$- and $h^1_1$-horn lifting, we can lift the zig-zag to $Y$.

Given a semisimplicial set $X$ and $f, g \in X_1(a, b)$, we call $f$ and $g$ \emph{related}, written $f \approx g$, if $f, g \co \Delta^1 \to X$ become equal after applying $\tau_1$.
If $X$ is fibrant, we see that this notion of relatedness coincides with the one defined in the previous subsection.

\begin{lemma} \label{lifting-relatedness}
Let $p \co Y \to X$ be a fibration in semisimplicial sets.
Given $f \in Y_1(a, b)$ and $\overline{g} \in X_1(pa, pb)$ such that $pf \approx \overline{g}$, there is $g \in Y_1(a, b)$ such that $pg = \overline{g}$ and $f \approx g$.
\end{lemma}

\begin{proof}
After taking a fibrant replacement of $X$, we see that the claim reduces to the case where $X$ is fibrant.
In that case, we way work with the explicit description of $\tau_1 Y$ and $\tau_1 X$ from the previous subsection.
The claim then follows by solving a lifting problem
\[
\xymatrix{
  \Delta^1
  \ar[r]^{f}
  \ar@{>->}[d]_{\anod}
&
  Y
  \ar@{->>}[d]
\\
  \Delta^0 \star P
  \ar[r]
  \ar@{.>}[ur]
&
  X
}
\]
where the bottom map is the witness of $pf \approx pg$ and the left map comes from one of the inclusions $\Delta^1 \to P$ and is observed to be anodyne.
\end{proof}

\subsection{Homotopies}

We have an \emph{interval} $(I, \delta_0, \delta_1)$ (without contraction) given by $I \defeq \Delta^1$ with endpoint inclusions $\delta_0, \delta_1$ given by $h_0^1, h_1^1 \co \Delta^0 \to \Delta^1$, respectively.
This interval together with the geometric monoidal structure induces a notion of \emph{homotopy} $h \co f_0 \sim f_1$ between maps $f_0, f_1 \co X \to Y$ of semisimplicial sets, consisting of a map $h \co \Delta^1 \otimes X \to Y$ that restricts to $f_0$ and $f_1$ on $\delta_0 \otimes X$ and $\delta_1 \otimes X$, respectively.
This is equivalent to an edge between their transposes $\overline{f_0}$ and $\overline{f_1}$ in $\hom(X, Y)$.

A zig-zag of such edges (or just an edge if $Y$ is fibrant) exists precisely if there is a morphism connecting the images of $\overline{f_0}$ and $\overline{f_1}$ in $\tau_1(\hom(X, Y))$ (or an equality $\overline{f_0} = \overline{f_1}$ in $\tau_0(\hom(X, Y))$.
Note that this is different from a natural isomorphism $\tau_1(f_0) \to \tau_1(f_1)$ (or an equality $\tau_0(f_0) = \tau_0(f_1)$).
Thus, the existence of zig-zags of homotopies can be tested in the set enrichment and groupoidal enrichment of $\widehat{\Delta_+}$; when the target is fibrant, this also applies to single homotopies instead of zig-zags.

In the spirit of~\cite{riehl-verity:2-cat-of-quasicat}, when working with homotopies targetting fibrant objets, we will freely exploit the relation to the groupoidal enrichment and regard homotopies as morphisms in hom-groupoids with relatedness of homotopies given by equality in the hom-groupoid, yielding notions of horizontal composition and vertical composition and inversion for homotopies that is closed under relatedness.
Given a morphism in $\hom(B, C)$ represented by homotopy $H \co I \otimes B \to C$, note that the right whiskering with an object of $\hom(A, B)$ given by a map $f \co A \to B$ is represented by $H \circ (I \otimes f)$ and that the left whiskering with an object of $\hom(C, D)$ given by a map $g \co C \to D$ is represented by $g \circ H$.
This approach will enable us make use of classical 2-categorical statements such as the following.

\begin{lemma} \label{graduate-prelemma}
Let $\mathcal{C}$ be a 2-category.
Given $f \co A \to A$ and an isomorphism $H \co f \to \id_A$, we have $fH = Hf$.
\end{lemma}

\begin{proof}
By interchange, we have
\[
fH \circ_2 H = HH = Hf \circ_2 H
.\]
The claim follows by invertibility of $H$.
\end{proof}

We say that $f \co A \to B$ is a \emph{homotopy equivalence} if there is $g \co B \to A$ with homotopies $H \co gf \sim \id_A$ and $K \co fg \sim \id_B$.
Homotopy equivalences become actual equivalences (or isomorphisms) in the groupoidal enrichment (or set enrichment, respectively) of semisimplicial sets (and come from those if $A$ and $B$ are fibrant).
We say this homotopy equivalence is \emph{coherent} if $f \circ H \approx K \circ (I \otimes f)$, \ie it is an adjoint equivalence (one of the equivalent conditions $fH = Kf$ and $gK = Hg$ holds) in the groupoidal enrichment of semisimplicial sets.

\begin{lemma}[Graduate Lemma] \label{graduate-lemma}
For any homotopy equivalence $(f, g, H, K)$ as above, there is a replacement $K'$ of $K$ such that $(f, g, H, K')$ forms a coherent homotopy equivalence.
\end{lemma}

\begin{proof}
We work in the groupoidal enrichment of semisimplicial sets restricted to fibrant objects.
We define a replacement morphism $K' \co fg \to \id_V$ for $K$ as follows:
\[
K' \defeq \inv_2(Kfg) \circ_2 fHg \circ_2 K
.\]
We now have
\begin{align*}
K' f
&=
\inv_2(Kfgf) \circ_2 fHgf \circ_2 Kf
\\&=
\inv_2(Kfgf) \circ_2 fgfH \circ_2 Kf \qquad \text{(by \cref{graduate-prelemma})}
\\&=
\inv_2(Kfgf) \circ_2 Kfgf \circ_2 fH
\\&=
fH
.\qedhere\end{align*}
\end{proof}

\medskip

Note that the usual notion of strong codeformation retract is not available for lack of degeneracies (though one could parametrize it over constant loops).
However, the notion of strong homotopy equivalence still makes sense.
Recall that $f \co A \to B$ is a \emph{strong homotopy equivalence} if there is a $g \co B \to A$ with homotopies $H \co gf \sim \id_A$ and $K \co fg \sim \id_B$ such that $f \circ H = K \circ (I \otimes f)$.
Note that this is a stricter requirement than merely $f \circ H \approx K \circ (I \otimes f)$ as for a coherent homotopy equivalence.
In contrast to coherent homotopy equivalences, note that $f \circ H = K \circ (I \otimes f)$ is not equivalent to the dual condition $g \circ K = H \circ (I \otimes g)$.

The formal importance of strong homotopy equivalences derives from the following observation.
A map $f$ is a strong homotopy equivalence exactly if $\theta \hatotimes f$ admits a retraction or equivalently
$\hathom(\theta, f)$ admits a section where $\theta$ is the square
\[
\xymatrix{
  0
  \ar[r]
  \ar[d]
&
  \top
  \ar[d]^{\delta_0}
\\
  \top
  \ar[r]^-{\delta_1}
&
  I
}
\]
seen in horizontal direction as a morphism in the arrow category.

\subsection{Homotopy theory}
\label{homotopy-theory}

\begin{lemma} \label{h-equiv-coherent-to-strong}
Let $(f, g, H, K)$ form a coherent homotopy equivalence between objects $U$ and $V$ with $V$ fibrant.
\begin{enumerate}
\item
If $f$ is a fibration, there is $H' \approx H$ such that $(f, g, H', K)$ forms a strong homotopy equivalence,
\item
If $f$ is a cofibration, there is $K' \approx K$ such that $(f, g, H, K')$ forms a strong homotopy equivalence.
\end{enumerate}
\end{lemma}

\begin{proof}
Claim~(i) is \cref{lifting-relatedness} applied to the fibration $\hom(U, f) \co \hom(U, U) \to \hom(U, V)$.
Claim~(ii) is \cref{lifting-relatedness} applied to the fibration $\hom(f, V) \co \hom(V, V) \to \hom(U, V)$.
\end{proof}

\begin{corollary} \label{anodyne-criteria}
The following are equivalent for a cofibration $m$ between fibrant objects:
\begin{enumerate}
\item $m$ is a homotopy equivalence,
\item $m$ is a strong homotopy equivalence,
\item $m$ is anodyne.
\end{enumerate}
\end{corollary}

\begin{proof}
Combining \cref{graduate-lemma} and \cref{h-equiv-coherent-to-strong}, we see that~(i) and~(ii) are equivalent.
If $m$ is a strong homotopy equivalence, then $m$ is a retract of the anodyne map $\delta_0 \hatotimes m$ and thus itself anodyne.

Let $m \co A \to B$ now be anodyne.
Lifting
\[
\xymatrix{
  A
  \ar[r]^{\id}
  \ar[d]_{m}
&
  A
  \ar[d]
\\
  B
  \ar[r]
  \ar@{.>}[ur]^{r}
&
  1
\rlap{,}}
\]
we obtain a retraction $r \co B \to A$.
Let $H \co I \otimes A \to A$ be some constant loop on $\overline{\id_A} \in \hom(A, A)_0$.
Lifting
\[
\xymatrix@C+1.5cm{
  B +_A I \otimes A +_A B
  \ar[r]^-{[mr, mH, \id_B]}
  \ar[d]_{i^1 \hatotimes m}
&
  B
  \ar[d]
\\
  I \otimes B
  \ar[r]
  \ar@{.>}[ur]^{K}
&
  1
\rlap{,}}
\]
we obtain a homotopy $K \co mr \sim \id_Y$ such that $m \circ H = K \circ (I \otimes m)$.
Thus, $(i, r, H, K)$ forms a strong homotopy equivalence.
\end{proof}

\begin{corollary} \label{triv-fib-criteria}
The following are equivalent for a fibration $p$ between fibrant objects:
\begin{enumerate}
\item $p$ is a homotopy equivalence,
\item $p$ is a strong homotopy equivalence,
\item $p$ is a trivial fibration.
\end{enumerate}
\end{corollary}

\begin{proof}
This proceeds dually to \cref{anodyne-criteria}.
\end{proof}

Since equivalences in a $(2, 1)$-category satisfy 2-out-of-6, so do homotopy equivalences between fibrant objects.
Thus, \cref{anodyne-criteria,triv-fib-criteria} show that the full subcategory $(\widehat{\Delta_+})_\fib$ of semisimplicial sets on fibrant objects form a model structure (cofibrations, homotopy equivalences, fibrations).
However, since this notion comes without any guaranteed limits or colimits, it is not particular useful.
The following proposition records some useful properties of finite limits in this context.

\begin{theorem} \label{semisimplicial-sets-fib-cat}
Fibrant semisimplicial sets form a fibration category.
\end{theorem}

\begin{proof}
The embedding of fibrant semisimplicial sets into semisimplicial sets reflects limits.
The terminal object in semisimplicial sets is fibrant, hence also terminal in fibrant semisimplicial sets.
Fibrations in semisimplicial sets are closed under pullback as they are defined by a lifting property, hence pullbacks in fibrant semisimplicial sets along fibrations exist and are also fibrations there.
The rest of the properties of a fibration category follow from the model structure.
\end{proof}

Let us now extend the notion of weak equivalence to maps between arbitrary objects.
We say that a map $A \to B$ is a \emph{weak equivalence} if there is a commuting diagram
\[
\xymatrix{
  A  
  \ar[r]^{\anod}
  \ar[d]
&
  X
  \ar[d]
\\
  B
  \ar[r]^{\anod}
&
  Y
}
\]
with anodyne maps and fibration as indicated such that $X$ and $Y$ are fibrant and $X \to Y$ is a homotopy equivalence.
We will see below in \cref{def-weak-equiv-invariant} that $X \to Y$ being a homotopy equivalence does not depend on the choice of fibrant replacements $X$ and $Y$ and lift $X \to Y$ of $A \to B$.
In particular, a map between fibrant objects is a weak equivalence exactly if homotopy equivalence.
As such, we can also define the notion of weak equivalence with respect to a fixed notion of fibrant replacement and a fixed chosen lift $X \to Y$ so that the definition does not involve quantification over (potentially large) semisimplicial sets.

Note that isomorphisms are weak equivalences because the identity on a fibrant semisimplicial sets is a homotopy equivalence.

\begin{lemma} \label{anod-equal-homotopic}
Consider a diagram
\[
\xymatrix{
  A
  \ar[r]_{\anod}^{j}
&
  X
  \ar@<+0.5em>[r]^{f}
  \ar@<-0.5em>[r]_{g}
&
  Y
}
\]
with $X$ and $Y$ fibrant, $j$ anodyne, and $jf = jg$.
Then $f$ and $g$ are homotopic.
\end{lemma}

\begin{proof}
Since $Y$ is fibrant, we have a reflexivity homotopy $H \co fj \sim gj$.
We then solve a lifting problem
\[
\xymatrix@C+1cm{
  X +_A I \otimes A +_A X
  \ar[r]^-{[f, H, g]}
  \ar[d]_{i^1 \hatotimes j}
&
  Y
  \ar[d]
\\
  I \otimes X
  \ar[r]
  \ar@{.>}[ur]
&
  1
\rlap{,}}
\]
noting that the left map is anodyne.
\end{proof}

\begin{corollary} \label{span-anod-h-equiv}
Given a span
\[
\xymatrix@C-0.5cm{
&
  A
  \ar[dl]_{\anod}
  \ar[dr]^{\anod}
\\
  X_1
&&
  X_2
}
\]
with $X_1$ and $X_2$ fibrant, there is a map $X_1 \to X_2$ under $A$ that is a homotopy equivalence.
\end{corollary}

\begin{proof}
We can produce maps $X_1 \to X_2$ and $X_2 \to X_1$ using lifting problems of anodyne maps against fibrant objects.
These are homotopy inverses by \cref{anod-equal-homotopic}.
\end{proof}

\begin{lemma} \label{def-weak-equiv-invariant}
Given a commuting diagram
\[
\xymatrix{
  X_1
  \ar[d]
&
  A  
  \ar[r]^-{\anod}
  \ar[l]_-{\anod}
  \ar[d]
&
  X_2
  \ar[d]
\\
  Y_1
&
  B
  \ar[r]^-{\anod}
  \ar[l]_-{\anod}
&
  Y_2
}
\]
with anodyne maps and fibration as indicated such that $X_1, X_2, Y_1, Y_2$ are fibrant, if $X_1 \to Y_1$ is a homotopy equivalence, then so is $X_2 \to Y_2$.
\end{lemma}

\begin{proof}
Using \cref{span-anod-h-equiv}, we produce maps $X_1 \to X_2$ under $A$ and $Y_1 \to Y_2$ under $B$ that are homotopy equivalences.
By \cref{anod-equal-homotopic}, the square
\[
\xymatrix{
  X_1
  \ar[r]
  \ar[d]
&
  X_2
  \ar[d]
\\
  Y_1
  \ar[r]
&
  Y_2
}
\]
commutes up to homotopy.
The claim then follows from 2-out-of-3 and invariance under 2-cell isomorphism of equivalences in the groupoidal enrichment of $\widehat{\Delta_+}$.
\end{proof}

\begin{corollary} \label{weak-equiv-2-out-of-6}
Semisimplicial sets form a homotopical category, \ie weak equivalences satisfy 2-out-of-6.
\end{corollary}

\begin{proof}
Given a sequence of three composable arrows, we choose an objectwise fibrant replacement, yielding a sequence of three composable arrows between fibrant objects.
By \cref{def-weak-equiv-invariant}, 2-out-of-6 for weak equivalences of the original sequence reduces to 2-out-of-6 for homotopy equivalences between fibrant objects of the new sequence.
\end{proof}

\begin{lemma} \label{weak-equiv-to-fib-between-fibrant-factorization}
In the arrow category, any map from a weak equivalence to a fibration between fibrant objects factors via a homotopy equivalence between fibrant objects.
\end{lemma}

\begin{proof}
Consider a commuting square
\begin{equation} \label{weak-equiv-to-fib-between-fibrant-factorization:0}
\begin{gathered}
\xymatrix{
  A
  \ar[r]
  \ar[d]^{\sim}
&
  Y
  \ar@{->>}[d]
\\
  B
  \ar[r]
&
  X
}
\end{gathered}
\end{equation}
with $X$ fibrant.
Since $A \to B$ is a weak equivalence, we have a commuting square
\begin{equation} \label{weak-equiv-to-fib-between-fibrant-factorization:1}
\begin{gathered}
\xymatrix{
  A  
  \ar[r]^{\anod}
  \ar[d]^{\sim}
&
  S
  \ar[d]
\\
  B
  \ar[r]^{\anod}
&
  T
}
\end{gathered}
\end{equation}
with $S \to T$ a homotopy equivalence between fibrant objects.
The desired factoring of~\eqref{weak-equiv-to-fib-between-fibrant-factorization:0} through \eqref{weak-equiv-to-fib-between-fibrant-factorization:1} is then given by lifting first $B \to T$ against $X$ and then $A \to S$ against $Y \to X$.
\end{proof}

\begin{corollary} \label{weak-equiv-lift-hom-fibrant}
Weak equivalences lift against fibrant objects up to homotopy.
\qed
\end{corollary}

\begin{lemma} \label{cofibration-h-equiv-fibration-lift}
Consider a commuting diagram
\[
\xymatrix{
  A
  \ar[r]
  \ar@{>->}[d]
&
  S
  \ar[r]
  \ar[d]
&
  Y
  \ar@{->>}[d]
\\
  B
  \ar[r]
  \ar@{.>}[urr]
&
  T
  \ar[r]
&
  X
}
\]
where the middle vertical map is a homotopy equivalence between fibrant objects.
Then the composite square has a diagonal filler as indicated.
\end{lemma}

\begin{proof}
Factor $S \to T$ into an anodyne map $S \to M$ followed by a fibration $M \to T$.
By 2-out-of-3 for homotopy equivalences between fibrant objects and \cref{anodyne-criteria,triv-fib-criteria}, we have that $M \to T$ is a trivial fibration.
Then the desired lift is the composite of a lift of the cofibration $A \to B$ against the trivial fibration $M \to Y$ and a lift of the anodyne map $S \to M$ against the fibration $Y \to X$.
\end{proof}

The following statement is a relative version the ``generalized extension property of Kan $\Delta$-sets`` of \cite[Corollary~5.4]{rourke-sanderson:delta-sets}, but with our a priori weaker notion of weak equivalence.

\begin{corollary} \label{cof-weak-equiv-lift-fib-between-fibrant}
Cofibrations that are weak equivalences lift against fibrations between fibrant objects.
\end{corollary}

\begin{proof}
Combine \cref{weak-equiv-to-fib-between-fibrant-factorization,cofibration-h-equiv-fibration-lift}.
\end{proof}

\begin{remark} \label{dunce-hat}
Any anodyne map is a cofibration and a weak equivalence, but the reverse direction does not hold.
The smallest counterexample is the \emph{Dunce hat} inclusion $\Delta^0 \to D$ where $D$ is the semisimplicial set with one vertex, one edge, and one triangle.
Taking the pushout
\[
\xymatrix{
  \Delta^0 \star L
  \ar[r]^{d_0 \star L}
  \ar[d]
&
  \Delta^1 \star L
  \ar[d]
\\
  D
  \ar[r]
&
  \pullbackcorner{ul}
  Q
\rlap{,}}
\]
the horizontal maps are anodyne, as can be checked for the composite
\[
\xymatrix{
  \Delta^0
  \ar[r]
&
  D
  \ar[r]
&
  Q
\rlap{,}}
\]
showing by 2-out-of-3 that the cofibration $\Delta^0 \to D$ is a weak equivalence.
If $\Delta^0 \to D$ was anodyne, it would arise as a codomain retract of an $\omega$-composition $\Delta^0 \to S$ of pushouts of coproducts of horn inclusions, but $S$ cannot contain a triangle with identical sides as any horn filling comes with a ``fresh'' side.
\end{remark}

Given a category $\mathcal{D}$, the geometric monoidal structure and the interval lift to $[\mathcal{D}, \widehat{\Delta_+}]$, inducing a notion of homotopy there.

\begin{lemma} \label{homotopies-limit-colimit}
Homotopies are closed under limit and colimit, \ie the limit and colimit functors from $[\mathcal{D}, \widehat{\Delta_+}]$ to $\widehat{\Delta_+}$ preserve the homotopy relation.
\end{lemma}

\begin{proof}
For the limit functor this is direct, for the colimit functor use the path object representation (\ie a homotopy between maps $A \to B$ is a map $A \to [I, B]$).
\end{proof}

\begin{corollary} \label{homotopies-product-coproduct}
Homotopies are closed under product and coproduct (of arbitrary arity), \ie if $f_i \sim g_i$ for all $i \in I$, then $\prod_{i \in I} f_i \sim \prod_{i \in I} g_i$ and $\coprod_{i \in I} f_i \sim \coprod_{i \in I} g_i$.
\qed
\end{corollary}

\begin{corollary} \label{h-equiv-product-coproduct}
Homotopy equivalences are closed under product and coproduct (of arbitrary arity).
\qed
\end{corollary}

\begin{lemma} \label{hom-pres-homotopy}
The functors $- \otimes -$ and $\hom(-, -)$ preserve homotopies in each arguments.
\end{lemma}

\begin{proof}
By formal properties of the geometric symmetric monoidal structure.
\end{proof}

\begin{corollary} \label{hom-pres-h-equiv}
The functors $- \otimes -$ and $\hom(-, -)$ preserve homotopy equivalences in each arguments.
\qed
\end{corollary}

\begin{corollary} \label{hom-pres-weak-equiv-dom}
For $X$ fibrant, the functor $\hom(-, X)$ preserves weak equivalences.
\end{corollary}

\begin{proof}
Note first that $\hom(j, X)$ is a trivial fibration between fibrant objects for $j$ anodyne.
The claim then follows from this and \cref{hom-pres-h-equiv} by inspecting the definition of weak equivalence and 2-out-of-3.
\end{proof}

\begin{corollary} \label{cof-weak-equiv-hom-fib}
For $X$ fibrant, the functor $\hom(-, X)$ sends cofibrations that are weak equivalences to trivial fibrations.
\end{corollary}

\begin{lemma} \label{weak-equiv-joyal-characterization}
The following are equivalent for a map $f$:
\begin{enumerate}
\item $f$ is a weak equivalence,
\item $\hom(f, X)$ is a homotopy equivalence for all fibrant $X$,
\item $\tau_0(\hom(f, X))$ is a bijection all fibrant $X$, \ie precomposition with $f$ induces a bijection of homotopy classses when the target is fibrant.
\end{enumerate}
\end{lemma}

\begin{proof}
The direction from~(i) to~(ii) is \cref{hom-pres-weak-equiv-dom}.
If~(ii) holds, then clearly does~(iii).
For the direction from~(iii) to~(i), we may reduce to the case that $A$ and $B$ are fibrant by unfolding the definition of weak equivalence and using 2-out-of-3 for bijections.
Then $f$ is a weak equivalence exactly if it becomes an isomorphism in the homotopy category of $\widehat{\Delta_+}$.
Noting that a map in an isomorphism exactly if precomposition with it induces bijections, this corresponds exactly to the given condition.
\end{proof}

\begin{corollary} \label{cof-weak-equiv-triv-fib-characterization}
A cofibration $m$ is a weak equivalence exactly if the fibration $\hom(m, Z)$ is trivial for all fibrant $Z$.
\qed
\end{corollary}

\begin{corollary} \label{anodyne-is-weak-equiv}
Anodyne maps are weak equivalences.
\end{corollary}

The following statement is useful when reducing certain statements about weak equivalences to anodyne maps.

\begin{corollary} \label{weak-equivalences-as-closure-of-andoyne}
Weak equivalences are the closure of anodyne maps, more specifically just $\omega$-compositions of pushouts of coproducts of horn inclusions, under 2-out-of-6.
\end{corollary}

\begin{proof}
By \cref{anodyne-is-weak-equiv,weak-equiv-2-out-of-6}.
Note that closure under codomain retracts follows from closure under 2-out-of-6.
\end{proof}

The significance of the following (standard) corollary is in its applications to homotopy equivalences between non-fibrant objects.

\begin{corollary} \label{h-equiv-is-weak-equiv}
Homotopy equivalences are weak equivalences.
\qed
\end{corollary}

\begin{corollary} \label{weak-equiv-coproduct}
Weak equivalences are closed under coproduct (of arbitrary arity).
\end{corollary}

\begin{proof}
We use the characterization of weak equivalences given by part~(ii) of \cref{weak-equiv-joyal-characterization}.
Then the claim reduces to closure of homotopy equivalences between fibrant objects under product, given by \cref{h-equiv-product-coproduct}.
\end{proof}

\begin{corollary} \label{weak-equiv-retracts}
Weak equivalences are closed under retract.
\end{corollary}

\begin{proof}
We use the characterization of weak equivalences given by part~(ii) of \cref{weak-equiv-joyal-characterization}.
Then the claim reduces to closure of homotopy equivalences under retracts.
\end{proof}

%

\begin{lemma} \label{cof-weak-equiv-closure}
Weak equivalences that are cofibrations are closed under pushout as well as $\omega$-composition.
\end{lemma}

\begin{proof}
Recall that we have corresponding closure statements for cofibrations.
Under \cref{cof-weak-equiv-triv-fib-characterization}, the claims then reduce to the dual closure statements of trivial fibrations, namely pushouts in spans of fibrant objects and $\omega^\op$-compositions.
\end{proof}

\begin{theorem} \label{semisimplicial-sets-cofibration-category}
Semisimplicial sets form a cofibration category.
\end{theorem}

\begin{proof}
The only non-trivial conditions of a cofibration category here are items~\ref{cofib-cat:pushout-triv-cofib} and~\ref{cofib-cat:2-out-of-6} of the definition, which are given by \cref{cof-weak-equiv-closure} and \cref{weak-equiv-2-out-of-6}, respectively.
\end{proof}

\begin{question}
Are cofibrations that are weak equivalences the left class of a weak factorization system?
\end{question}

\begin{lemma} \label{geometric-product-pres-weak-equiv}
The geometric product preserves weak equivalences, \ie if $u$ and $v$ are weak equivalences, then so is $u \otimes v$.
\end{lemma}

\begin{proof}
By 2-out-of-3, we can reduce to the case where $v = \id_A$ for some $A$.
We use the characterization of weak equivalences given by part~(ii) of \cref{weak-equiv-joyal-characterization}.
Then the claim reduces to closure of homotopy equivalences between fibrant objects under $\hom(A, -)$, which holds by \cref{hom-pres-h-equiv}.
\end{proof}

Note that the cartesian product does not preserve weak equivalences: the inclusion $\delta^0 \co \Delta^0 \to \Delta^1$ is a weak equivalence, but $\top \times \delta^0 \co \Delta^0 \to \Delta^0 + \Delta^0$ is not.

\begin{corollary} \label{cof-w-equiv-tensor}
If $u$ is a weak equivalence and cofibration and $v$ is any map, then the Leibniz geometric product $u \hatotimes v$ is a weak equivalence.
In particular, cofibrations that are weak equivalences are closed under Leibniz geometric product with cofibrations.
\end{corollary}

\begin{proof}
The first part is \cref{geometric-product-pres-weak-equiv,cof-weak-equiv-closure} using 2-out-of-3.
The second part follows with \cref{cof-tensor}.
\end{proof}

\begin{lemma}
There is map of cylinders
\[
\xymatrix@C+0.5cm@R-0.5cm{
&
  I \otimes (A \star B)
  \ar@{.>}[dd]
\\
  A \star B
  \ar[ur]^(0.45){\delta_0 \otimes (A \star B)}
  \ar[dr]
  \ar@<-0.2em>@{{}{ }{}}[dr]_-(0.3){(\delta_0 \otimes A) \star (\delta_0 \otimes B)}
&&
  A \star B
  \ar[ul]_(0.45){\delta_1 \otimes (A \star B)}
  \ar[dl]
  \ar@<+0.2em>@{{}{ }{}}[dl]^-(0.3){(\delta_1 \otimes A) \star (\delta_1 \otimes B)}
\\&
  (I \otimes A) \star (I \otimes B)
}
\]
natural in semisimplicial sets $A$ and $B$.
\end{lemma}

\begin{proof}
Let us extend the endofunctor $I \otimes -$ from semisimplicial sets to augmented semisimplicial sets using left Kan extension (note that it restricts back to the original on semisimplicial sets by fully faithfulness of the embedding).
We will establish the map of cylinders in augmented semisimplicial sets.
Here, both cylinder functors are separately cocontinuous in $A$ and $B$.
It will thus suffice to produce the map on the restriction to representables.

Let $A = \Delta^a$ and $B = \Delta^b$ with $a, b \geq -1$.
An $n$-simplex of $I \otimes (\Delta^a \star \Delta^b)$ is maps $[n] \to [1]$ and $[n] \to [a] \star [b]$ in categories that are jointly monic.
The second map induces a decomposition $[n] = [a'] \star [b']$ and maps $[a'] \to [a]$ and $[b'] \to [b]$.
By cancellation of monomorphisms, we have that $[a'] \to [n] \to [1]$ and $[a'] \to [a]$ as well as $[b'] \to [b] \to [1]$ and $[b'] \to [b]$ are jointly monic.
This is an $a'$-simplex in $I \otimes \Delta^a$ and an $b'$-simplex in $I \otimes \Delta^b$, together an $n$-simplex in $(I \otimes \Delta^a) \star (I \otimes \Delta^b)$.
One may check this operation is natural in $[a], [b] \in \Delta_{+,\aug}$.
It also clearly commutes with the cylinder endpoint inclusions.
\end{proof}

\begin{question}
The decomposition of $n$ in the above proof essentially uses parametric right adjointness of the join.
One may wonder whether this can be turned into a more abstract argument.
\end{question}

\begin{corollary} \label{join-pres-homotopy-and-h-equiv}
The join monoidal structure preserves homotopies and homotopy equivalences, \ie given homotopies $f_0 \sim f_1$ and $g_0 \sim g_1$, we have $f_0 \star g_0 \sim f_1 \star g_1$, and given homotopy equivalences $u$ and $v$, then $u \star v$ is a homotopy equivalence.
\qed
\end{corollary}

Note that the nullary case is also included in the preceding statement, but is not interesting.

\begin{corollary} \label{join-pres-weak-equiv}
The join monoidal structure preserves weak equivalences, \ie given weak equivalences $u$ and $v$, then $u \star v$ is a weak equivalence.
\end{corollary}

\begin{proof}
Note that a map $A \to B$ is a weak equivalence precisely if there is a square
\[
\xymatrix{
  A
  \ar[r]^-{\anod}
  \ar[d]
&
  X
  \ar[d]
\\
  B
  \ar[r]^-{\anod}
&
  Y
}
\]
with $X \to Y$ a homotopy equivalence (not necessarily between fibrant objects); in one direction this is by definition, in the other by \cref{h-equiv-is-weak-equiv,anodyne-is-weak-equiv} and 2-out-of-3.
The join monoidal structure preserves anodyne maps by \cref{anodyne-join-cof} and homotopy equivalences by \cref{join-pres-homotopy-and-h-equiv}, thus also preserves weak equivalences.
\end{proof}

\begin{corollary} \label{cof-w-equiv-join-cof}
The Leibniz join of a cofibration that is a weak equivalence with any map is a weak equivalence.
Cofibrations that are weak equivalences are closed under Leibniz join with cofibrations.
\qed
\end{corollary}

The following statement and its proof are essentially~\cite[Proposition~2.6]{schwede:top-triang-categories}.
It is a finitary version of \cref{weak-equivalences-as-closure-of-andoyne}.

\begin{lemma} \label{finite-weak-equiv}
In the full subcategory of finite semisimplicial sets, the class of weak equivalences is the closure under 2-out-of-6 of pushouts of horn inclusions.
\end{lemma}

\begin{proof}
Let $W'$ denote the closure of the statement.
It will suffice to show that every weak equivalence $f \co A \to B$ admits a composable weak equivalence $g$ such that $gf \in W'$.
The claim then follows by applying this statement again to $g$ and using 2-out-of-6.

Let $A_0 \defeq A$ and $A_{n+1}$ be obtained by filling all (finitely many) possible horns in $A_n$ that have not already been filled in a previous step.
Then every inclusion in the sequence
\[
\xymatrix{
  A_0
  \ar[r]
&
  A_1
  \ar[r]
&
  \ldots
}
\]
is a finite composition of pushouts of horn inclusion and the colimit $X$ is fibrant.
Using \cref{weak-equiv-lift-hom-fibrant}, we may solve the lifting problem
\[
\xymatrix{
  A
  \ar[r]
  \ar@{>->}[d]^{\sim}_{f}
&
  X
\\
  B
  \ar@{.>}[ur]
}
\]
up to homotopy.
Since $B +_A I \otimes A$ is finite, there is $n$ such that this lift and the associated homotopy lift through $A_n \to X$, yielding a commuting diagram
\[
\xymatrix{
  A
  \ar[d]_{\delta_0 \otimes A}
  \ar[drr]
\\
  I \otimes A
  \ar@{.>}[rr]
&&
  A_n
\rlap{.}\\
  A
  \ar[u]^{\delta_1 \otimes A}
  \ar[r]^{f}
&
  B
  \ar@{.>}[ur]^(0.4){g}
}
\]
All solid maps here are weak equivalences, hence so is $g$ by 2-out-of-3.
Furthermore, the left vertical maps and the top right map lie in $W'$, hence so does $gf$ by 2-out-of-3.
\end{proof}

\begin{question}
One may wonder if the preceding statement can be strengthed.
\begin{itemize}
\item 
Is \cref{finite-weak-equiv} still true with 2-out-of-6 replaced by 2-out-of-3?
\item
Given a cofibration $A \to B$ between finite simplicial sets that is a weak equivalence, is there $B \to C$ such that both $B \to C$ and the composite $A \to C$ are finite composites of pushouts of horn inclusions?
\end{itemize}
\end{question}

\subsection{Comparison with simplicial sets}
\label{subsection:quasicategories}

Write $\Delta_-$ for the wide subcategory of $\Delta$ consisting of epimorphisms.
Then $(\Delta_-, \Delta_+)$ forms a Reedy structure on $\Delta$.
Let us write $U$ for the forgetful functor from simplicial sets to semisimplicial sets, induced by restriction along the inclusion $i \co \Delta_+ \to \Delta$.
Left and right Kan extensions along $i$ induce adjunctions $L \dashv U \dashv R$ where $L$ and $R$ denote the free and cofree simplicial set functor, respectively.
We will be particularly interested in the induced monad $(UL, \eta, \mu)$.

The \emph{interval} in simplicial sets is given by the image of the semisimplicial inverval under $L$ and will be denoted using the same symbols.
Note that it has a contraction and connections.
With the cartesian monoidal structure, this induces associated notions of cylinder, homotopy, homotopy equivalence.
Recall that $L$ sends the geometric to the cartesian symmetric monoidal structure, so that $L$ functorially preserves homotopies and homotopy equivalences.

By adjointness, we obtain the following statement.

\begin{lemma} \label{simplicial-vs-semisimplicial-cylinder}
Functorially in a simplicial set $A$, there is cylinder morphism from the semisimplicial cylinder on $UA$ to the image under $U$ of the simplicial cylinder on $A$.
\end{lemma}

\begin{proof}
By adjointness and monoidality of $L$, a cylinder map
\[
(I \otimes UA, \delta_0 \otimes UA, \delta_1 \otimes UA) \to U(I \times A, \delta_0 \times A, \delta_1 \times A)
\]
corresponds to a cylinder map
\[
(I \times LUA, \delta_0 \times LUA, \delta_1 \times LUA) \to (I \times A, \delta_0 \times A, \delta_1 \times A)
.\]
This is given by the counit of the adjunction $L \dashv U$.
\end{proof}

\begin{corollary} \label{simplicial-to-semisimplicial-homotopy}
The forgetful functor $U$ functorially preserves homotopies and homotopy equivalences.
\qed
\end{corollary}

Let $S$ be the subcategory of $[[1], \Delta]$ with objects restricted to maps in $\Delta_-$ and codomain part of morphisms restricted to maps in $\Delta_+$.
The evident fully faithful functor $S \to \Delta \downarrow i$ gives rise to a morphism of two-sided opfibrations (in particular the left legs are Grothendieck fibrations and the right legs are Grothendieck opfibrations) as below:
\begin{equation} \label{simplicial-vs-semisimplicial-spans}
\begin{gathered}
\xymatrix@R-0.4cm{
&
  S
  \ar[dl]_{\ev_0}
  \ar[dr]^{\ev_1}
  \ar[dd]
\\
  \Delta
&&
  \Delta_+
\rlap{.}\\&
  \Delta \downarrow i
  \ar[ul]^{\pi_0}
  \ar[ur]_{\pi_1}
}
\end{gathered}
\end{equation}
Note that $\ev_0$ is a discrete Grothendieck fibration.
The map $S \to \Delta \downarrow i$ is a coreflective embedding: its right adjoint sends a map in $\Delta$ to the left factor of its Reedy factorization.
This restricts to coreflective embeddings of fibers over $\Delta$.

The functor $L$ is given by restriction along $\pi_1^\op$ followed by left Kan extension along $\pi_0^\op$, which is computed as fiberwise colimits since $\pi_0^\op$ is a Grothendieck opfibration.
Since $S \to \Delta \downarrow i$ has fiberwise (over $\Delta$) right adjoints, its opposite is fiberwise final.
Thus, the left Kan extension along $\pi_0^\op$ is also given by restriction along (the opposite of) $S \to (\Delta \downarrow i)$ followed by fiberwise coproducts along $\ev_0^\op$ (since $\ev_0^\op$ is a discrete Grothendieck opfibration).
In summary, we can describe $L$ as restriction along $\ev_1^\op$ followed by fiberwise coproducts along $\ev_0^\op$.
Explicitly, given $X \in \widehat{\Delta_+}$ and $[n] \in \Delta$, the set $LX_n$ consists of pairs $(s, x)$ where $s \co [n] \to [k]$ is in $\Delta_-$ and $x \in X_k$.
The action of a map $f \co [m] \to [n]$ in $\Delta$ on $(s, x)$ is $(s', x[d'])$ where $(s', d')$ is the Reedy factorization of the map $sf$ in $\Delta$.

Let us think of presheaves over a category $\mathcal{C}$ as discrete Grothendieck fibrations over $\mathcal{C}$.
We generalize the notion of linear polynomial functors in tbe wide subcategory of $\Cat$ on discrete Grothendieck fibrations by allowing the right leg to be any functor, not just a discrete Grothendieck fibration.
Then $L$ is the linear polynomial functor specified by the top span in~\eqref{simplicial-vs-semisimplicial-spans}, \ie pullback along $\ev_1$ and postcomposition with $\ev_0$.

Given a category $A$, let us write $[A, \Delta]_-^+$ for the subcategory of $[A, \Delta]$ with objects restricted to functors $A \to \Delta_-$ and morphisms restricted to natural transformations with components in $\Delta_+$.
These categories will play an important role in the relationship between semisimplicial and simplicial sets.
Note that we have a pullback
\[
\xymatrix{
  [[1], \Delta]_-^+
  \ar[r]
  \ar[d]_{\ev_0}
  \pullbackcorner{dr}
&
  S
  \ar[d]^{\ev_0}
\\
  \Delta_+
  \ar[r]^-{i}
&
  \Delta
\rlap{.}}
\]
In particular, as for $S$, the evaluation $\ev_0 \co [[1], \Delta]_-^+ \to \Delta_+$ is a discrete Grothendieck fibration.
The composite linear polynomial functor $UL$ is specified by the span
\[
\xymatrix@C-0.5cm{
&
  [[1], \Delta]_-^+
  \ar[dl]_{\ev_0}
  \ar[dr]^{\ev_1}
\\
  \Delta_+
&&
  \Delta_+
\rlap{.}}
\]

Note that the $n$-th power of $UL$ is the linear polynomial functor specified by the span
\[
\xymatrix@C-0.5cm{
&
  [[n], \Delta]_-^+
  \ar[dl]_{\ev_0}
  \ar[dr]^{\ev_n}
\\
  \Delta_+
&&
  \Delta_+
\rlap{.}}
\]
The unit and multiplication of the monad $UL$ are given by maps of linear polynomial functors
\begin{align} \label{free-simplicial-set-unit-multiplication}
\begin{gathered}
\xymatrix@R-0.4cm{
&
  [[0], \Delta]_-^+
  \ar[dl]_{\ev_0}
  \ar[dr]^{\ev_0}
  \ar[dd]^{- \circ s_0}
\\
  \Delta_+
&&
  \Delta_+
\rlap{,}\\&
  [[1], \Delta]_-^+
  \ar[ul]^{\ev_0}
  \ar[ur]_{\ev_1}
}
\end{gathered}
&&
\begin{gathered}
\xymatrix@R-0.4cm{
&
  [[2], \Delta]_-^+
  \ar[dl]_{\ev_0}
  \ar[dr]^{\ev_0}
  \ar[dd]^{- \circ d_1}
\\
  \Delta_+
&&
  \Delta_+
\rlap{.}\\&
  [[1], \Delta]_-^+
  \ar[ul]^{\ev_0}
  \ar[ur]_{\ev_1}
}
\end{gathered}
\end{align}
So $UL$ is a linear polynomial monad (in our generalized sense), in particular a cartesian monad.

\begin{lemma} \label{free-simplicial-set-homotopically-idempotent}
The natural transformations $\eta UL$ and $UL\eta$ are homotopic in $[\widehat{\Delta_+}, \widehat{\Delta_+}]$.
\end{lemma}

\begin{proof}
Using the cylinder map of \cref{simplicial-vs-semisimplicial-cylinder} (with $L(-)$ for $A$), it will suffice to construct a homotopy $H$ with respect to the cylinder on $UL$ given by the image under $U$ of the simplicial cylinder on $L$:
\[
\xymatrix@C+0.5cm{
  UL
  \ar[d]_{UL \delta_0 \times UL}
  \ar[dr]^{UL\eta}
\\
  ULI \times UL
  \ar@{.>}[r]^-(0.3){H}
&
  ULUL
\rlap{.}\\
  UL
  \ar[u]^{UL \delta \times UL}
  \ar[ur]_{\eta UL}
}
\]
Note that the diagram lives over $UL$, with the indexing of $ULUL$ given by the monad multiplication.

Switching to the viewpoint of linear polynomial functors, the problem then becomes to find a map fitting as indicated into the commuting diagram
\begin{equation} \label{free-simplicial-set-homotopically-idempotent:0}
\begin{gathered}
\xymatrix{
  [[1], \Delta]_-^+
  \ar[d]_{0^* \times_\Delta \id}
  \ar[dr]^{- \circ s_0}
\\
  \Delta/[1] \times_\Delta [[1], \Delta]_-^+
  \ar@{.>}[r]^-(0.3){H}
&
  [[2], \Delta]_-^+
\\
  [[1], \Delta]_-^+
  \ar[u]^{1^* \times_\Delta \id}
  \ar[ur]_{- \circ s_1}
}
\end{gathered}
\end{equation}
of (summits of spans specifiying) linear polynomial functors over $UL$ (with indexing of $[[2], \Delta]_-^+$ given by $- \circ d_1$).
Here, the pullback on the left is taken with respect to $\ev_0$.

As maps between summits of (generalized) linear functors, \eqref{free-simplicial-set-homotopically-idempotent:0} is a diagram of discrete Grothendieck fibrations over $[[1], \Delta]_-^+$.
It will thus suffice to define the action of $H$ on objects and morphisms compatible with the indexing; preservation of identities and compositions will hold for free.
Similarly, it will suffice to check commutativity of the upper and lower triangles on objects; commutativity on morphisms will hold for free.

An object of $\Delta/[1] \times_\Delta [[1], \Delta]_-^+$ is a span in $\Delta$ as below:
\begin{equation} \label{free-simplicial-set-homotopically-idempotent:1}
\begin{gathered}
\xymatrix{
  [1]
&
  [n]
  \ar[l]
  \ar@{->>}[r]
&
  [k]
\rlap{.}}
\end{gathered}
\end{equation}
A morphism between two such spans
\begin{equation} \label{free-simplicial-set-homotopically-idempotent:2}
\begin{gathered}
\xymatrix@R-0.7cm{
&
  [n']
  \ar[dl]
  \ar@{->}[r]
  \ar@{>.>}[dd]
&
  [k']
  \ar@{>.>}[dd]
\\
  [1]
\\&
  [n]
  \ar[ul]
  \ar@{->>}[r]
&
  [k]
}
\end{gathered}
\end{equation}
is given by dotted maps commuting as indicated.
The behaviour of $H$ on an object~\eqref{free-simplicial-set-homotopically-idempotent:1} should be to functorially construct a factorization of $[n] \to [k]$ in $\Delta_-$ whose left and right factors are trivial if $[n] \to [1]$ is constantly $0$ and $1$, respectively.

Let us construct the action of $H$ on the object~\eqref{free-simplicial-set-homotopically-idempotent:1}.
Let $[n_1] \to [n]$ be the pullback in $\Delta_\aug$ of $1 \co [0] \to [1]$ along $[n] \to [1]$.
We construct the Reedy factorization (dashed arrows) and take the pushout drawn below:
\[
\xymatrix{
  [n_1]
  \ar@{>->}[r]
  \ar@{-->>}[d]
&
  [n]
  \ar@{->>}[d]
  \ar@/^1em/@{->>}[ddr]
\\
  [k_1]
  \ar@{>->}[r]
  \ar@/_1em/@{>-->}[drr]
&
  [m]
  \ar@{.>>}[dr]
  \pullbackcorner{ul}
\\&&
  [k]
\rlap{.}}
\]
We define $[n] \to [m] \to [k]$ as the result of $H$.
We check the boundary conditions:
\begin{itemize}
\item
If $[n] \to [1]$ is constantly $0$, then $[a]$ is initial.
Then $[a] \to [k]$ is mono, so $[a] \to [b]$ is iso, and hence is $[n] \to [m]$.
\item
If $[n] \to [1]$ is constantly $1$, then $[a] \to [n]$ is iso.
Then $[b] \to [k]$ and $[b] \to [m]$ are iso, and hence is $[m] \to [k]$.
\end{itemize}

Let us construct the action of $H$ on the morphism~\eqref{free-simplicial-set-homotopically-idempotent:2}.
Let $[n'] \to [m'] \to [k']$ be the value of $H$ on the codomain, constructed from a pullback $[n_1'] \to [n']$ of $1 \co [0] \to [1]$, a Reedy factorization $[n_1'] \to [k_1'] \to [k']$, and a pushout as before.
We obtain a commuting diagram
\[
\xymatrix@R-0.5cm@C-0.5cm{
  [n_1']
  \ar@{>->}[rr]
  \ar@{->>}[dd]
  \ar@{>->}[dr]
&&
  [n']
  \ar@{->>}[dd]|{\hole}
  \ar@{>->}[dr]
\\&
  [n_1]
  \ar@{>->}[rr]
  \ar@{->>}[dd]
&&
  [n]
  \ar@{->>}[dd]
\\
  [k_1']
  \ar@{>->}[rr]|{\hole}
  \ar@{>->}[dr]
&&
  [m']
  \ar@{.>}[dr]
\\&
  [k_1]
  \ar@{>->}[rr]
&&
  [m]
}
\]
where the front and back squares are pushouts and the top square is a pullback by pullback pasting.
The pushouts induce a map $[m'] \to [m]$, which is in $\Delta_+$ by \cref{delta-adhesivity-lemma} below.
The desired result of $H$ is depicted below:
\[
\xymatrix{
  [n']
  \ar@{->>}[r]
  \ar@{>->}[d]
&
  [m']
  \ar@{->>}[r]
  \ar@{>->}[d]
&
  [k']
  \ar@{>->}[d]
\\
  [n]
  \ar@{->>}[r]
&
  [m]
  \ar@{->>}[r]
&
  [k]
\rlap{;}}
\]
note that the right square commutes because $[n'] \to [m']$ is epi.
\end{proof}

\begin{lemma} \label{delta-adhesivity-lemma}
In $\Delta_\aug$, consider a diagram
\[
\xymatrix@R-0.5cm@C-0.5cm{
  A'
  \ar@{>->}[rr]
  \ar@{->>}[dd]
  \ar@{>->}[dr]
&&
  B'
  \ar@{->>}[dd]|{\hole}
  \ar@{>->}[dr]
\\&
  A
  \ar@{>->}[rr]
  \ar@{->>}[dd]
&&
  B
  \ar@{->>}[dd]
\\
  X'
  \ar@{>->}[rr]|{\hole}
  \ar@{>->}[dr]
&&
  Y'
  \ar[dr]
\\&
  X
  \ar@{>->}[rr]
&&
  Y
}
\]
where the vertical maps are in $\Delta_{\aug,-}$ and the horizontal maps except for $Y' \to Y$ are in $\Delta_{\aug,-}$.
Assume that the front square is a pushout and the top square is a pullback.
Then the back square is a pushout exactly if the bottom right map $Y' \to Y$ is in $\Delta_{\aug,+}$.
\end{lemma}

\begin{proof}
Assume that the forward direction holds.
Then the reverse direction can be derived as follows.
We introduce the pushout $Y'_0$ in the back square (recall that $\Delta_\aug$ has pushouts along maps in $\Delta_{\aug,+}$, that these pushouts preserve maps in $\Delta_{\aug,-}$, and that $\Delta_{\aug,+}$ is closed under pushout).
Thus, we obtain another cubical diagram with $Y'$ replaced by $Y'_0$.
By assumption, the induced map $Y'_0 \to Y$ is in $\Delta_{\aug,+}$.
By cancellation, so is the induced map $Y'_0 \to Y'$.
Again by cancellation, since $B' \to Y'$ is in $\Delta_-$, so is $Y'_0 \to Y'$.
Thus, $Y'_0 \to Y'$ is iso, making the back square a pushout.

It will thus suffice to show the forward direction.
Note that the forgetful functor to finite sets creates epis and monos and preserves the pullbacks and pushouts in the diagram.
Thus, it is enough to verify the corresponding assertion for finite sets.
Let us do this now.

Let $y'_0, y'_1 \in Y'$ have the same image $y \in Y$.
Note that $Y'$ can be descibed as the coproduct of $X'$ with the complement of the image of $A'$ in $B'$ (using that finite sets are Boolean), and similarly for the front pushout.
We perform case distinction on which component $y'_0, y'_1$ lie in.
\begin{itemize}
\item
Assume there are respective preimages $x'_0, x'_1 \in X'$.
Since $X' \to Y$ is mono, we have $x'_0 = x'_1$, hence $y'_0 = y'_1$.
\item
Assume they have respective preimages $b'_0, b'_1 \in B'$ that do not lie in the image of $A' \to B'$.
Because the top square is a pullback, the induced elements $b_0, b_1 \in B$ do not lie in the image of $A \to B$.
Since both map to $y \in Y$, we have $b_0 = b_1$.
Since $B' \to B$ is mono, we then have $b'_0 = b'_1$, hence $y'_0 = y'_1$.
\item
Assume $y'_0$ has a preimage $x'$ in $X'$ and $y'_1$ has a preimage $b'$ in $B'$ (the dual case is analogous).
As a van Kampen pushout, the front square is also a pullback.
By pullback pasting, the composite square from $A' \to B'$ to $X \to Y$ is a pullback.
Let $x \in X$ be the image of $x'$.
By assumption, $b'$ and $x$ map to the same element $y \in Y$.
Thus, they have a common preimage $a' \in A'$.
By construction, $a'$ and $x'$ have the same image in $X$.
Since $X' \to X$ is mono, $x'$ is the image of $a'$.
Since $x'$ and $b'$ have a common preimage in $A'$, they map to the same element in $Y'$, \ie $y'_0 = y'_1$.
\qedhere
\end{itemize}
\end{proof}

\begin{question}
It would be nice to have an abstract categorical argument replacing the last paragraph of the above proof that works in any suitably exact category without the assumption that it is Boolean.
\end{question}

\subsection{Comparison with Kan complexes}
\label{subsection:comparison-kan-complexes}

Let us define weak factorization systems (cofibration, trivial fibration) and (trivial cofibration, fibration) in simplicial sets generated by the images of boundary inclusions and horn inclusions under $L$, respectively.
These are exactly the usual boundary and horn inclusions in simplicial sets.
By definition, $U$ creates fibrations and creates trivial fibrations from maps lifting against cofibrations.
By cocontinuity, $L$ preserves cofibrations and sends anodyne maps to trivial cofibrations.

The cofibrations in simplicial sets are the monomorphisms $A \to B$ for which the relative latching object inclusions are decidable.
In particular, the cofibrant objects are those for which it is decidable if a simplex is degenerate.
By induction on dimension, one sees that a cofibration in simplicial sets is in particular levelwise decidable.
It follows that $U$ preserves cofibrations.

\begin{remark}
Let $X$ be fibrant and have a vertex $a$.
One might hope that $LX$ is fibrant.
However, this is not the case.
Consider the $\Lambda^2_2$-horn in $LX$ that on $d_1$ is $(s_0, a)$ and on $d_2$ is $(\id_{[1]}, l)$ for a loop $l$ on $a$ given by fibrancy.
Assume that this horn has a filler $(t, y)$.
Let $t d_1 = u t'$ for uniquely determined face map $u$ and degeneracy map $t'$; then $(s_0, a) = (t, y)[d_1] = (t', y[u])$.
Since $t'$ is constant and $d_1$ preserves terminal and initial objects, also $t$ is constant, \ie $t \co [2] \to [0]$.
This gives a contradiction $(\id_{[1]}, l) = (t, y)[d_2] = (s_0, y)$.

One might think that this sort of problem only appears for outer horns.
Unfortunately, it is also not generally the case that $LX$ lifts against inner horn inclusions if $X$ does.
Lifts for $LX$ against horn inclusions of dimension up to three exist.
However, there is another problem in dimension four.
Consider a $\Lambda^4_3$-horn in $LX$ that on $d_0$ is an $s_0$-degeneracy, on $d_4$ is an $s_1$-degeneracy, and on $d_1$ and $d_2$ is different tetrahedra ($X$ can be freely generated to admit this data).
Then any filler would be an $s_1$-degeneracy and have identical sides on $d_1$ and $d_2$, a contradiction.
\end{remark}

Constructively, the problem with simplicial sets over semisimplicial sets is that not every simplicial set is cofibrant, only those for which we can decide whether a given element is degenerate.
The property that every object is cofibrant permeates the theory of Kan complexes (\ie fibrant objects).
For example, by formal properties one has that $X^A$ is fibrant if $A$ is cofibrant and $X$ is fbrant.
One would like this to be valid for general $A$, but removing the cofibrancy assumption results in a statement that is not valid constructively~\cite{bezem:kripke-countermodel,simplicial-sets-non-constructive} (note that the arguably more primitive problem that there can be non-cofibrant objects, although not discussed in the cited articles, can be observed in an even simpler Kripke model).
Restricting attenting to cofibrant objects is also insufficient as for example the cotensor of a cofibrant object with a countably infinite set is not in general cofibrant.

Using classical logic, one may prove that the two weak factorization systems in simplicial sets as above form a model structure.
In fact, the only classical statement needed is precisely that every object is cofibrant; this is the development in~\cite{sattler:equivalence-extension}.
One may then ask how the forgetful functor and its adjoints relate the homotopy theories of simplicial and semisimplicial sets.
Previous developments of the homotopy theory of semisimplicial sets have usually defined weak equivalences as being created by the free simplicial set functor (or equivalently the realization to CW-complexes).
Since we took a more synthetic approach, letting weak equivalences essentially be generated by semisimplicial horn inclusions, we have an a priori weaker notion of weak equivalence.
It thus remains to show that $L$ reflects weak equivalences.
We will prove this classically in a way that enables us to extract a few useful constructive facts in semisimplicial sets along the way.
For differentiation, those statements that depend on classical logic will be marked as such.

Recall that $\Sk^n(1)$ is the unique (up to isomorphism) semisimplicial set consisting of just one simplex in each dimension up to $n \geq 0$.
We write $t^{k,n}$ for the unique map $1^{(k)} \to 1^{(n)}$ for $-1 \leq k \leq n$.
This is a cofibration: there is a unique cocartesian morphism (pushout square) from $t^{n-1,n}$ to $i^n$ for $n \geq 0$.

\begin{lemma} \label{quotient-boundary-inclusion-join-closure}
There is a unique and cocartesian morphism from the Leibniz join of $t^{a-1,a}$ and $t^{b-1,b}$ to $t^{a+b,a+b+1}$ for $a, b \geq 0$.
\end{lemma}

\begin{proof}
Uniqueness holds since source and target of $t^{a+b,a+b+1}$ are subterminal.

For existence, we have isomorphisms and cocartesian morphisms indicated by solid arrows below:
\[
\xymatrix@C+0.5cm{
  i^a \hatjoin i^n
  \ar[r]^-{\simeq}
  \ar[d]^{\cocart}
&
  i^{a+b+1}
  \ar[d]^{\cocart}
\\
  t^{a-1,a} \hatjoin t^{b-1,b}
  \ar@{.>}[r]^-{\cocart}
&
  t^{a+b,a+b+1}
\rlap{.}}
\]
By cocartesianness, to find a cocartesian arrow commuting as indicated is to find that arrow after applying the domain functor to the whole diagram.
Then the diagram lives in $n$-truncated semisimplicial sets and the bottom right object is terminal.
\end{proof}

\begin{corollary} \label{i-zero-join-quotient-boundary-inclusion}
There is a unique and cocartesian morphism $i^0 \hatjoin t^{n-1,n} \to t^{n,n+1}$ for $n \geq 0$.
\end{corollary}

\begin{proof}
We have $t^{-1,0} \simeq i^0$.
\end{proof}

The following statement is false for odd $n$.

\begin{lemma} \label{maximal-quotient-horn}
The map $\Sk^n(1) \to \Sk^{n+2}(1)$ is a weak equivalence for even $n$.
\end{lemma}

\begin{proof}
This is by induction on $n$.
The base case $n = 0$ is given by \cref{dunce-hat}, although it can also be seen as an instance of the induction step once $\Sk^{-2}(1)$ is treated as the initial augmented semisimplicial set.

Let now $n \geq 2$.
Just for this proof, we will abbreviate $1^{(k)} \defeq \Sk^k(1)$ and denote $d_1 \co \braces{a} \to \Delta^1$ and $d_0 \co \braces{b} \to \Delta^1$ so that some arrow annotations can be left implicit.

We construct pushouts
\[
\xymatrix@C+1cm{
  \braces{b} \star 1^{(n-2)}
  \ar[r]
  \ar[d]
&
  \Delta^1 \star 1^{(n-2)}
  \ar[d]
\\
  1^{(n)}
  \ar[r]_-{\anod}
&
  A
  \pullbackcorner{ul}
}
\]
and
\[
\xymatrix@C+1cm{
  \braces{b} \star 1^{(n+1)}
  \ar[r]
  \ar[d]
&
  \Delta^1 \star 1^{(n+1)}
  \ar[d]
\\
  1^{(n+2)}
  \ar[r]_-{\anod}
&
  B
  \pullbackcorner{ul}
}
\]
where the anodyne maps come from \cref{anodyne-join-cof}.
We have an evident morphism from the first pushout span to the second, so by functoriality of pushouts we obtain a map $A \to B$ commuting as below:
\[
\xymatrix@C+1cm{
  1^{(n)}
  \ar[r]_{\anod}
  \ar[d]
&
  A
  \ar@{.>}[d]
\\
  1^{(n+2)}
  \ar[r]_{\anod}
&
  B
\rlap{.}}
\]
By 2-out-of-3, it will suffice to show that $A \to B$ is a weak equivalence.

\begin{figure}
\begin{align*}
\scriptsize
\xymatrix@R-2em@C-2em{
  1^{(n-2)}
  \ar[rr]
  \ar[dd]
  \ar@{=}[dr]
&&
  \braces{a} \star 1^{(n-2)}
  \ar[dd]|{\hole}
  \ar@{=}[dr]
\\&
  1^{(n-2)}
  \ar[rr]
  \ar[dd]
&&
  \braces{a} \star 1^{(n-2)}
  \ar@{.>}[dd]
\\
  \braces{b} \star 1^{(n-2)}
  \ar[rr]|{\hole}
  \ar[dr]
&&
  \Delta^1 \star 1^{(n-2)}
  \ar@{.>}[dr]
\\&
  1^{(n)}
  \ar@{.>}[rr]
&&
  X_0
}
&&
\scriptsize
\xymatrix@R-2em@C-2em{
  1^{(n-2)}
  \ar[rr]
  \ar[dd]
  \ar[dr]
&&
  \braces{a} \star 1^{(n-2)}
  \ar[dd]|{\hole}
  \ar[dr]
\\&
  1^{(n)}
  \ar[rr]
  \ar@{=}[dd]
&&
  \braces{a} \star 1^{(n)}
  \ar@{.>}[dd]
\\
  \braces{b} \star 1^{(n-2)}
  \ar[rr]|{\hole}
  \ar[dr]
&&
  \Delta^1 \star 1^{(n-2)}
  \ar@{.>}[dr]
\\&
  1^{(n)}
  \ar@{.>}[rr]
&&
  X_1
}
\\
\scriptsize
\xymatrix@R-2em@C-2em{
  1^{(n-2)}
  \ar[rr]
  \ar[dd]
  \ar[dr]
&&
  \braces{a} \star 1^{(n-2)}
  \ar[dd]|{\hole}
  \ar[dr]
\\&
  1^{(n+1)}
  \ar[rr]
  \ar@{=}[dd]
&&
  \braces{a} \star 1^{(n+1)}
  \ar@{.>}[dd]
\\
  \braces{b} \star 1^{(n-2)}
  \ar[rr]|{\hole}
  \ar[dr]
&&
  \Delta^1 \star 1^{(n-2)}
  \ar@{.>}[dr]
\\&
  1^{(n+1)}
  \ar@{.>}[rr]
&&
  X_2
}
&&
\scriptsize
\xymatrix@R-2em@C-2em{
  1^{(n)}
  \ar[rr]
  \ar[dd]
  \ar[dr]
&&
  \braces{a} \star 1^{(n)}
  \ar[dd]|{\hole}
  \ar[dr]
\\&
  1^{(n+1)}
  \ar[rr]
  \ar@{=}[dd]
&&
  \braces{a} \star 1^{(n+1)}
  \ar@{.>}[dd]
\\
  \braces{b} \star 1^{(n)}
  \ar[rr]|{\hole}
  \ar[dr]
&&
  \Delta^1 \star 1^{(n)}
  \ar@{.>}[dr]
\\&
  1^{(n+1)}
  \ar@{.>}[rr]
&&
  X_3
}
\\
\scriptsize
\xymatrix@R-2em@C-2em{
  1^{(n+1)}
  \ar[rr]
  \ar[dd]
  \ar@{=}[dr]
&&
  \braces{a} \star 1^{(n+1)}
  \ar[dd]|{\hole}
  \ar@{=}[dr]
\\&
  1^{(n+1)}
  \ar[rr]
  \ar[dd]
&&
  \braces{a} \star 1^{(n+1)}
  \ar@{.>}[dd]
\\
  \braces{b} \star 1^{(n+1)}
  \ar[rr]|{\hole}
  \ar[dr]
&&
  \Delta^1 \star 1^{(n+1)}
  \ar@{.>}[dr]
\\&
  1^{(n+2)}
  \ar@{.>}[rr]
&&
  X_4
}
\end{align*}
\caption{Colimiting cubes for $X_0$ to $X_4$ used in the proof of \cref{maximal-quotient-horn}}
\label{cubes-figure}
\end{figure}
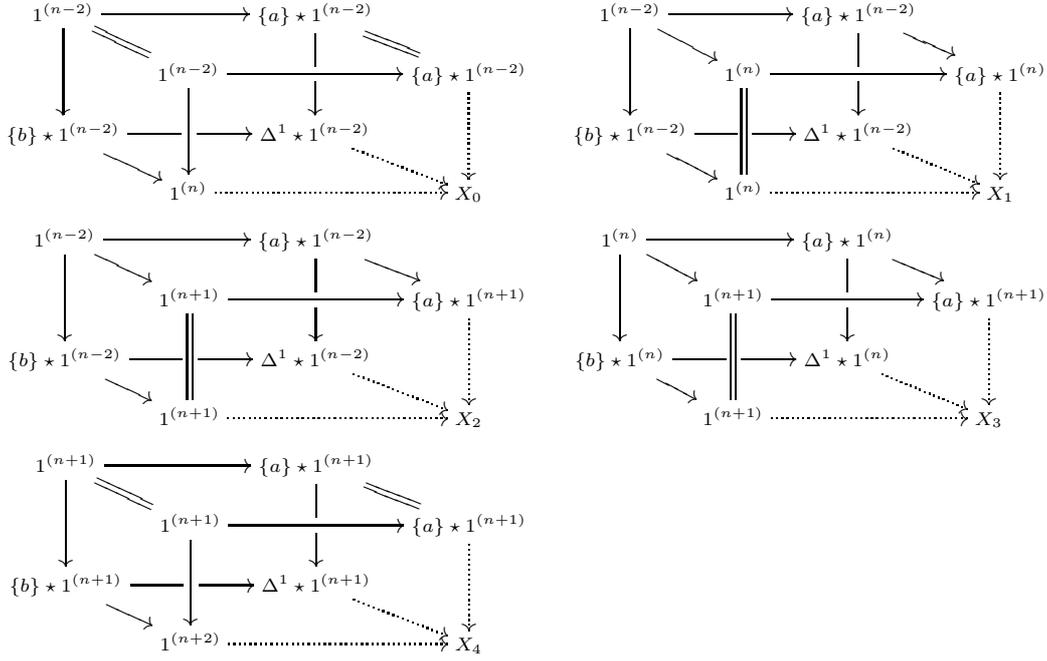

Now consider the colimits in \cref{cubes-figure}, each over a diagram of shape $S(\braces{0, 1, 2})^\op$ where $S(T)$ is the category of non-empty subobjects of a finite set $T$ (note that $S([n]) \simeq \Delta_+/[n]$).
Any map $T_1 \to T_2$ of finite sets induces a Grothendieck fibration $S(T_1) \to S(T_2)$ given by postcomposition followed by taking the image; this operation forms a functor $\FinSet \to \Cat$.
Colimits over $S(\braces{0, 1, 2})^\op$, visualized as a cube, reduce to pushouts in three different ways, each corresponding to the choice of an axis $i \in \braces{0, 1, 2}$, by first taking pushouts in the squares orthogonal to this axis and then taking a final pushout of the induced map between pushouts with an induced pushout corner map; this reduction is itself functorial and justified by pasting of left Kan extensions along
\[
\xymatrix{
  S(\braces{0, 1, 2})^\op
  \ar[r]
&
  S(\braces{0, 1, 2}/{\sim})^\op
  \ar[r]
&
  S(\braces{i})^\op
}
\]
where $i \sim i$ and $j \sim k$ if $j, k \neq i$ since left Kan extensions along Grothendieck opfibrations compute as fiberwise colimits.
This reduction to pushouts will be made use of frequently in the following.
In particular, a cocone under a diagram over $S(\braces{0, 1, 2})^\op$ is colimiting exactly if the induced square between pushout corner maps in the faces parallel to the axis $i$ is a pushout.

There are canonical maps between the diagrams in \cref{cubes-figure} that by functoriality of colimits give rise to a sequence
\begin{equation} \label{maximal-quotient-horn:0}
\begin{gathered}
\xymatrix{
  X_0
  \ar[r]
&
  X_1
  \ar[r]
&
  X_2
  \ar[r]
&
  X_3
  \ar[r]
&
  X_4
\rlap{.}}
\end{gathered}
\end{equation}
Reducing the colimit diagrams for $X_0$ and $X_4$ to pushouts with respect to the axis from top to bottom, the pushout corner maps in the top faces are isomorphisms and the bottom faces coincide with the pushout squares defining $A$ and $B$.
So we see that the map $X_0 \to X_4$ above coincides (up to isomorphism) with the map $A \to B$ constructed earlier.
It will thus suffice to show that each step $X_i \to X_{i+1}$ for $i \in \braces{0, 1, 2, 3}$ is a weak equivalence.

For the first step, we apply the reduction to pushouts with respect to the axis from back to front.
Since the back faces in the colimit diagrams for $X_0$ and $X_1$ are identical, by pushout pasting the map $X_0 \to X_1$ writes as a pushout of the induced map between the pushouts in the front faces:
\[
\xymatrix{
  1^{(n)} +_{1^{(n-2)}} \Delta^0 \star 1^{(n-2)}
  \ar[r]
  \ar[d]_{i^0 \hatjoin t^{n-2,n}}
&
  X_0
  \ar[d]
\\
  \Delta^0 \star 1^{(n)}
  \ar[r]
&
  X_1
  \pullbackcorner{ul}
\rlap{.}}
\]
Thus, $X_0 \to X_1$ is a weak equivalence by induction hypothesis, \cref{cof-w-equiv-join-cof}, and \cref{cof-weak-equiv-closure}.

For the second step, we also use the reduction to pushouts with respect to the axis from back to front.
Again the back faces in the colimit diagrams for $X_1$ and $X_2$ are identical, and the left maps of the front faces are identities.
By pushout pasting, it follows that $X_1 \to X_2$ writes as a pushout of the map between the remaining objects:
\[
\xymatrix{
  \Delta^0 \star 1^{(n)}
  \ar[r]
  \ar[d]_{\Delta^0 \star t^{n,n+1}}
&
  X_1
  \ar[d]
\\
  \Delta^0 \star 1^{(n+1)}
  \ar[r]
&
  X_2
  \pullbackcorner{ul}
\rlap{.}}
\]
Hence, $X_1 \to X_2$ is a weak equivalence by \cref{point-join-cofib,cof-weak-equiv-closure} (in fact, the left map in the above diagram is just a pushout of $h^{n+2}_0$).

For the third step, we yet again use the reduction to pushouts with respect to the axis from back to front.
This time, the spans in the front faces in the colimit diagrams for $X_2$ and $X_3$ are identical.
By pushout pasting, it follows that $X_2 \to X_3$ writes as a pushout of the pushout corner map in the induced square between the pushout corner maps in the back faces.
Since the join with fixed second argument preserves pushouts, the pushout corner maps in the back faces are just $i^1 \star 1^{(n-2)}$ and $i^1 \star 1^{(n)}$, respectively.
Hence, $X_2 \to X_3$ writes as a pushout of a Leibniz join:
\[
\xymatrix{
  \partial\Delta^1 \star 1^{(n)} +_{\partial\Delta^1 \star 1^{(n-2)}} \Delta^1 \star 1^{(n-2)}
  \ar[r]
  \ar[d]_{i^1 \hatjoin t^{n-2,n}}
&
  X_2
  \ar[d]
\\
  \Delta_1 \star 1^{(n)}
  \ar[r]
&
  X_3
  \pullbackcorner{ul}
\rlap{.}}
\]
Thus, $X_2 \to X_3$ is a weak equivalence by induction hypothesis, \cref{cof-w-equiv-join-cof}, and \cref{cof-weak-equiv-closure}.

For the fourth step, we use the reduction to pushouts with respect to the axis from left to right.
We first claim that the induced square between pushout corner maps in the left faces in the colimit diagrams for $X_3$ and $X_4$ is a pushout.
This means that the cube between the left faces, depicted below, is a colimit diagram:
\[
\xymatrix@R-0.5cm@C-0.5cm{
  1^{(n)}
  \ar[rr]
  \ar[dd]
  \ar[dr]
&&
  1^{(n+1)}
  \ar[dd]|!{[dl];[dr]}{\hole}
  \ar[dr]
\\&
  1^{(n+1)}
  \ar[rr]
  \ar[dd]
&&
  1^{(n+1)}
  \ar[dd]
\\
  \braces{b} \star 1^{(n)}
  \ar[rr]|!{[ur];[dr]}{\hole}
  \ar[dr]
&&
  \braces{b} \star 1^{(n+1)}
  \ar[dr]
\\&
  1^{(n+1)}
  \ar[rr]
&&
  1^{(n+2)}
\rlap{.}}
\]
Using that same equivalence in a different direction, this means that the induced square between pushout corner maps in the back and front faces of the above cube is a pushout.
But that is precisely the cocartesian morphism of \cref{i-zero-join-quotient-boundary-inclusion}.
Having dealt with that claim, it follows by pushout pasting that $X_3 \to X_4$ writes as a pushout of the induced map between the pushouts in the right faces:
\[
\xymatrix{
  1^{(n+1)} +_{1^{(n)}} \Delta^0 \star 1^{(n)}
  \ar[r]
  \ar[d]_{h^1_0 \hatjoin t^{n,n+1}}
&
  X_3
  \ar[d]
\\
  \Delta^1 \star 1^{(n+1)}
  \ar[r]
&
  X_4
  \pullbackcorner{ul}
\rlap{.}}
\]
Therefore, $X_3 \to X_4$ is a weak equivalence by \cref{anodyne-join-cof,cof-weak-equiv-closure} (in fact, the left map in the above diagram is just a pushout of $h^{n+3}_0$).
\end{proof}

Note that the preceding proof, while staying at an abstract level of chasing pushouts through diagrams, is completely formal in the following sense: every detail of the proof is specified, there is no hand-waving that obscures hidden laborous details such as manually having to check that two abstractly specified arrows indeed coincide or separately needed manual verification of functoriality or naturality of certain constructions.
One would like every combinatorial proof here to have this level of formality without needing to go down to the level of set-based reasoning.

\begin{remark}
All colimits considered in the proof of \cref{maximal-quotient-horn} are actually homotopy colimits as they decompose into pushouts along cofibrations.
As such, instances of the glueing lemma (weakly equivalent spans with one leg a cofibration have weakly equivalent pushouts)  in the cofibration category of semisimplicial sets might be used to substitute certain parts of the argument.
It would be nice if it could be used to actually simplify the proof.
\end{remark}

\begin{corollary} \label{monoidal-to-cartesian-unit-w-equiv}
The map $\top \to 1$ is a cofibration and a weak equivalence.
\end{corollary}

\begin{proof}
Note that $\top \to 1$ arises as the $\omega$-composite
\[
\xymatrix{
  \Sk^0(1)
  \ar[r]
&
  \Sk^2(1)
  \ar[r]
&
  \ldots
\rlap{.}}
\]
Thus, the claim follows from \cref{maximal-quotient-horn}.
\end{proof}

\begin{lemma} \label{simplicial-unit-w-equiv}
The unit $\eta \co \Id \to UL$ is valued in cofibrations that are weak equivalences.
\end{lemma}

\begin{proof}
Viewing $\eta$ as a map of linear polynomial functors, note that the map between the span summits in~\eqref{free-simplicial-set-unit-multiplication} is a decidable inclusion.
It follows that $\eta$ is valued in cofibrations.
Since $UL$ preserves cofibrations, it follows that Leibniz application of $\eta$ preserves cofibrations.

It remains to prove that $\eta$ is valued in weak equivalences.
We generalize to Leibniz application of $\eta$ sending cofibrations to weak equivalences.
By \cref{semisimplicial-sets-cofibration-category}, we know that semisimplicial sets form a cofibration category, so we may apply the glueing lemma in the standard way to reduce the claim to Leibniz applications of $\eta$ to boundary inclusions.
By induction on dimension and the usual uses of 2-out-of-3 and closure of cofibrations that are weak equivalences under pushout, the claim reduces further to $\eta_{\Delta^n} \co \Delta^n \to UL\Delta^n$ being a weak equivalence for $[n] \in \Delta_+$.
Recall that both $L$ and $U$ are monoidal with respect to the join and that weak equivalences are closed under join (\cref{join-pres-weak-equiv}),
Since $\Delta^n$ is the $(n+1)$-fold join of $\top$, it thus suffices to check that $\eta_\top \co \top \to UL\top$ is a weak equivalence.
This is the statement of \cref{monoidal-to-cartesian-unit-w-equiv}.
\end{proof}

\begin{remark}
In a classical setting, \cref{simplicial-unit-w-equiv} is usually proved by first showing that $\epsilon$ is valued in weak equivalences in simplicial sets and then appealing to a triangle law of the adjunction $L \dashv U$ and the fact that $L$ reflects weak equivalences (which is often taken to be the definition of weak equivalences in semisimplicial sets).
However, in order to be able to use this strategy, we would need to know that $L$ reflects weak equivalences, which will be proven (for a certain definition of weak equivalences in simplicial sets) precisely using \cref{simplicial-unit-w-equiv}.
\end{remark}

\begin{remark}
Let a \emph{horn extension} be a pushout of a horn inclusion and a \emph{simultaneous horn extension} be a pushout of a coproduct of horn inclusions.
It is possible to continue differently in the second half of the proof of \cref{simplicial-unit-w-equiv} and directly exhibit the Leibniz application of $\eta$ to boundary inclusions as $\omega$-composites of pushout joins of maps of the form in \cref{maximal-quotient-horn} with boundary inclusions (allowing suitably augmented the case of dimension $-1$) on both sides.
After unfolding various pushout joins and the induction of \cref{maximal-quotient-horn}, this shows that for any semisimplicial set $A$ there are $\omega$-composites of simultaneous horn extensions $f$ and $g$ such that $g = f \eta_A$, \ie $ULA$ can be obtained from $A$ by first countably infinitely many times filling systems of independent horns and then countably infinitely many times removing fillers for systems of independent horns.
\end{remark}

\begin{corollary} \label{UL-create-w-equiv}
The functor $UL$ creates weak equivalences.
\end{corollary}

\begin{proof}
By \cref{simplicial-unit-w-equiv} and 2-out-of-3.
\end{proof}

\begin{corollary} \label{U-triv-cofib-to-w-equiv}
$U$ sends trivial cofibrations to cofibrations that are weak equivalences.
\end{corollary}

\begin{proof}
By closure properties \cref{weak-equiv-coproduct,cof-weak-equiv-closure} of cofibrations that are weak equivalences, it suffices to show that $U$ sends horn inclusions to weak equivalences, \ie that $UL$ sends generating anodyne maps to weak equivalences.
This holds by \cref{UL-create-w-equiv}.
\end{proof}

\begin{lemma} \label{U-triv-fib-to-w-equiv}
The functor $U$ sends trivial fibrations to weak equivalences.
\end{lemma}

\begin{proof}
Let $Y \to X$ be a trivial fibration in simplicial sets.
By adjointness, $UY \to UX$ is a map lifting against cofibrations in semisimplicial sets.
If $UX$ was fibrant, we could use \cref{triv-fib-criteria} to show that $UY \to UX$ is a homotopy equivalence.
Note that the only use of fibrancy of $UX$ in the proof of that statement is to produce a constant loop $I \otimes UX \to UX$ on $\overline{\id_{UX}} \in \hom(UX, UX)_0$.
However, we can use the simplicial structure to produce such a constant loop via
\[
\xymatrix@C+0.2cm{
  I \otimes UX
  \ar[r]^-{\eta_{I \otimes UX}}
&
  UL(I \otimes UX)
  \ar[r]^-{\simeq}
&
  U(I \times LUX)
  \ar[r]^-{U \pi_1}
&
  ULUX
  \ar[r]^-{U \epsilon_X}
&
  UX
\rlap{.}}
\qedhere
\]
\end{proof}

Note that the above proof does not make use of cofibrancy of $UX$ in simplicial sets, which we cannot have constructively.

Define now the \emph{weak equivalences} in simplicial sets as the closure of trivial cofibrations and trivial fibrations under 2-out-of-6.
Classically, note that (cofibration, weak equivalence, fibration) is the usual Kan model structure: the cofibrations and trivial cofibrations are defined as usual and that is enough to determine the model structure.

\begin{corollary} \label{U-pres-w-equiv}
The functor $U$ preserves weak equivalences.
\end{corollary}

\begin{proof}
By \cref{U-triv-cofib-to-w-equiv,U-triv-fib-to-w-equiv,weak-equiv-2-out-of-6}.
\end{proof}

\begin{corollary} \label{L-create-w-equiv}
The functor $L$ creates weak equivalences.
\end{corollary}

\begin{proof}
For preservation, recall weak equivalences in semisimplicial sets by definition lie in the closure of anodyne maps under 2-out-of-6 and that $L$ sends anodyne maps to trivial cofibrations.
Reflection is \cref{U-pres-w-equiv} and reflection in \cref{UL-create-w-equiv}.
\end{proof}

Finally, we have related our definition of weak equivalence in semisimplicial sets with the one of the classical development.
This allows us to conclude as follows.

\begin{theorem}[Classical logic] \label{summary-equiv}
In the triple of functors
\[
\xymatrix@C+1.5cm{
  \widehat{\Delta_+}
  \ar@/^2em/[r]^{L}
  \ar@/^-2em/[r]_{R}
  \ar@{{}{ }{}}@/^1em/[r]|{\bot}
  \ar@{{}{ }{}}@/^-1em/[r]|{\bot}
&
  \widehat{\Delta}
  \ar[l]|{U}
}
\]
relating semisimplicial sets and simplicial sets where $U$ is the forgetful functor:
\begin{itemize}
\item
the maps $L$ and $U$ are weak equivalences of cofibration categories,
\item
the maps $U$ and $R$, restricted to fibrant objects, are weak equivalences of fibration categories,
\item
both adjunctions descend to equivalences of homotopy categories.
\end{itemize}
\end{theorem}

\begin{proof}
As left adjoints, $L$ and $U$ preserve all colimits.
We have already noted that $L$ and $U$ preserve cofibration; they preserve weak equivalences by \cref{U-pres-w-equiv,L-create-w-equiv}.
Thus, they are exact functors of cofibration categories.
By adjointness, $U$ and $R$ become exact functors of fibration categories.

As a consequence, both adjunctions descend to adjunctions of homotopy categories (which for the case of the right adjoints are constructed equivalently using only the fibrant objects).
Recall that weak equivalences between fibration categories can be tested at the level of homotopy categories.
Note that $U$ as an exact functor between cofibration categories is a weak equivalence exactly if it is a weak equivalence as an exact functor between fibration categories, as the weak equivalence coincide.
As claims will thus follow if we can verify that $U$ descends to an equivalence on homotopy categories.

By \cref{simplicial-unit-w-equiv}, the unit $\eta \co \Id \to UL$ is valued in weak equivalences.
Using classically logic, we can also show that the counit $\epsilon \co LU \to \Id$ is valued in weak equivalences: since every simplicial set is cofibrant, the claim reduces as usual to $\epsilon_{\Delta^n}$ being weak equivalence for $n \geq 0$.
It can be seen directly that its source and target are contractible (see for instance \cite[Lemma~6.2]{kapulkin-szumilo:internal-language}), but we can also deduce it in a way that will generalize to the marked case using 2-out-of-3 in the triangle identity $\epsilon L \circ L\eta = \id$ since $L$ preserves weak equivalences by \cref{L-create-w-equiv-marked}, 
From the above facts about $\eta$ and $\epsilon$, it follows that the adjunction $L \dashv U$ descends to an equivalence between homotopy categories.
\end{proof}

\bigskip

One might wonder when a semisimplicial set can be endowed with degeneracy operators, \ie lies in the essential image of $U$.
As proven in~\cite[Lemma~5.6]{rourke-sanderson:delta-sets}, the free simplicial set adjunction is monadic.
(As the author learned from Paolo Capriotti, using classical logic, it is also comonadic.)
So to endow a semisimplicial set with degeneracy operators is to find an algebra structure for it with respect to the monad $UL$.

As shown geometrically in \cite{rourke-sanderson:delta-sets}, every fibrant semisimplicial set admits degeneracy operators.
A purely combinatorial proof of this statement is~\cite{mcclure:semisimplicial-degeneracies}; however, the proof method is non-constructive, namely by choosing degeneracies at every level only for those simplices which do not lie in the image of certain degeneracy operators already constructed at lower levels (the images of the constructed degeneracy operators are not decidable subsets).
It seems unlikely that the statement holds constructively.

Given a fibrant semisimplicial set $A$, note that we can find an algebra structure with respect to $UL$ as a pointed endofunctor, \ie a retraction $r \co ULX \to X$ to $\eta_X$, as in the proof of the following lemma.
This corresponds to a system of degeneracy operators $s$ that are not required to satisfy the simplicial identities of the form $ss = s$.
For a direct constructive of such pseudo-degeneracy operators, see \cite[Proposition~7.10]{coquand:semisimplicial}.

\begin{lemma} \label{free-simplicial-set-unit-h-equiv-on-fibrant}
For fibrant $X$, the map $\eta_X \co X \to ULX$ is a homotopy equivalence.
\end{lemma}

\begin{proof}
Construct a retraction $r$ to $\eta_X$ by taking a lift in the diagram
\[
\xymatrix{
  X
  \ar[r]^{\id}
  \ar[d]_{\eta_A}
&
  X
\\
  ULX
  \ar@{.>}[ur]_{r}
}
\]
using \cref{simplicial-unit-w-equiv,cof-weak-equiv-lift-fib-between-fibrant}.
Then note that
\begin{align*}
\eta_X \circ r
&=
ULr \circ \eta_{ULX}
\\&\sim
ULr \circ UL \eta_X
\\&=
\id_{ULX}
\end{align*}
where the homotopy is due to \cref{free-simplicial-set-homotopically-idempotent}.
\end{proof}

We have a unique map $\mathcal{C} \times \mathcal{D} \to \mathcal{C} \otimes \mathcal{D}$ natural in semicategories $\mathcal{C}$ and $\mathcal{D}$ and compatible with associativity and symmetry of the two symmetric monoidal structures.
This induces also a map $A \times B \to A \otimes B$ natural in semisimplicial sets $A$ and $B$ with corresponding properties.

The following statement is usually proved by choosing degeneracies as explained earlier.
As this does not work in a constructive setting, we are forced to come up with a different proof.

\begin{theorem} \label{cartesian-vs-geometric}
Given fibrant semisimplicial sets $X$ and $Y$, the canonical comparison map
\[
X \times Y \to X \otimes Y
\]
between the cartesian and the geometric product is a weak equivalence.
\end{theorem}

\begin{proof}
Note that $\eta_{X \otimes Y} \co X \otimes Y \to UL(X \otimes Y)$ is a weak equivalence by \cref{simplicial-unit-w-equiv}.
Postcomposing with it, it will suffice by 2-out-of-3 to show that $X \times Y \to UL(X \otimes Y)$ is a weak equivalence.
Note that $UL(X \otimes Y) \simeq U(LX \times LY) \simeq ULX \times ULY$.
Under this isomorphism, the map in question corresponds to the product $\eta_X \times \eta_Y$.
Since $\eta_X$ and $\eta_Y$ are homotopy equivalences by \cref{free-simplicial-set-unit-h-equiv-on-fibrant}, so is their product by \cref{h-equiv-product-coproduct}.
\end{proof}

Note that closure under product fails to hold for weak equivalence.
This explains why the combinatorial complexity of \cref{free-simplicial-set-homotopically-idempotent} is necessary in the above proof.

\subsection{Subdivision}

Let us write $L \dashv U$ for both the free category on a semicategory and the free simplicial set on a semisimplicial set adjunction.

\subsubsection{Simplicial subdivision}

We recall the definition of the subdivision functor $\Sd$ on simplicial sets.
Let $T$ be the full subcategory of $[[1], \Delta]$ restricted to maps in $\Delta_+$.
Note that $\ev_1 \co T \to \Delta$ is a Grothendieck opfibration: the cobase change of $[a] \to [m]$ along $[m] \to [n]$ returns the right factor of the Reedy factorization of the composite $[a] \to [n]$.
Under the inverse of the Grothendieck construction, this induces a functor $P \co \Delta \to \Cat$.
Now $\Sd$ is defined as the composition of the cocontinuous extension of $P$ with $N$ (observe that this coincides with the cocontinuous extension of $NP$).
Note that the subdivision functor extends canonically to augmented simplicial sets.

\subsubsection{Semisimplicial subdivision}

We also have a subdivision functor $\Sd$ on semisimplicial sets.
Recall that $\Delta_+$ as a category is freely generated by a semicategory.
Seeing semisimplicial sets as semicategorical presheaves over that semicategory, $\Sd$ is defined as the composite of the semicategory of elements functor $\scel$ followed by the semisimplicial nerve $N$.
It is cocontinuous, and its restriction $\sd$ to representables is described by the top row in the following diagram commuting up to natural isomorphism:
\[
\xymatrix@C+0.5cm{
  \Delta_+
  \ar[r]^-{(\Delta_+)_{/-}}
  \ar[d]
&
  \SemiCat
  \ar[r]^-{N}
  \ar[d]
&
  \widehat{\Delta_+}
  \ar[d]
\\
  \Delta
  \ar[r]^{P}
&
  \Cat
  \ar[r]^-{N}
&
  \widehat{\Delta}
\rlap{;}}
\]
here, the top left arrow denotes the categorical slice of the semicategory generating $\Delta_+$ and the middle and right vertical arrows are the free category and free simplicial set functors, respectively.
This shows that the free simplicial set functor commutes with subdivision.
Note that the subdivision functor extends canonically to augmented semisimplicial sets.

Semisimplicial subdivision is simpler than simplicial subdivision.
Given a simplicial set $A$ with decidable degeneracies, we can compute
\[
(\Sd A)_n \simeq \coprod_{k \geq 0, x \in A^\nd_k} \set{\sigma \in N((\Delta_+)_{/[k]})_n}{\sigma(n) = ([k], \id_{[k]})}
\]
where $A^\nd_k$ denotes the set of non-degenerate $k$-simplices of $A$.
For a general simplicial set $A$, we do not have a description of $(\Sd A)_n$ without quotients.
In contrast, the set of $n$-simplices of the subdivision of a semisimplicial set is describable without quotients in general.

\subsubsection{Last and first vertex maps}

Recall the inclusion $i \co \Delta \to \Cat$.
We have a natural transformation $\Delta_{/-} \to i$ that on $[n]$ is the functor $\Delta_{/[n]} \to [n]$ sending $f \co [k] \to [n]$ to $f(k)$.
Extending cocontinuously, we get a map $\int \to Q$ where $Q$ is the fundamental category functor, left adjoint to the simplicial nerve $N$.
Equivalently, this is a map $\int N \to \Id$ between endofunctors on categories, known as the \emph{categorical last vertex map}.
Dually, we have the \emph{categorical first vertex map} $\int N \to (-)^\op$.

Confusingly, we also have a map $N \int \to \Id$, given on representables by $N$ applied to the map $\Delta_{/-} \to i$.
Note that we have a map $P \to \Delta_{/-}$, inducing a map $\Sd \to N \int$.
We thus obtain the \emph{simplicial last vertex map} $\Sd \to \Id$.
Dually, we also have a \emph{simplicial first vertex map} $\Sd \to (-)^\op$ where $(-)^\op$ denotes the opposite simplicial set functor, precomposition with the opposite functor restricted to $\Delta$.

These maps do not have semicategorical or semisimplicial analogues.
Recall that the forgetful functors commute with the nerves, \ie $N U \simeq U N$, and that the induced Beck-Chevalley map $L N \to N L$ is an isomorphism.
We have a canonical map $\int_+ U \to U \int$, by adjointness giving rise to a map $L \int_+ \to \int L$.
Combining these with the categorical last vertex map, we obtain the \emph{approximation to a semicategorical last vertex map} 
\[
\xymatrix{
  \cel N
  \ar[r]^-{\simeq}
&
  L \scel N
  \ar[r]
&
  \cel L N
  \ar[r]^-{\simeq}
&
  \cel N L
  \ar[r]
&
  L
\rlap{.}}
\]
Dually, we have the \emph{approximation to a semicategorical first vertex map} $\cel N \to (-)^\op L$.

\subsubsection{Indexed join}

Let $(\mathcal{C}, \top, \otimes)$ be a monoidal category.
Recall the Day convolution monoidal structure induced on presheaves over $\mathcal{C}$.
Given presheaves $A, B$ over $\mathcal{C}$ and a presheaf $F$ over $\int A \times \int B$, the \emph{indexed Day convolution} $A \otimes_F B$ is the left Kan extension of $F$ along
\[
\xymatrix{
  \int A \times \int B
  \ar[r]
&
  \mathcal{C} \times \mathcal{C}
  \ar[r]^-{\otimes}
&
  \mathcal{C}
\rlap{.}}
\]
Let $A \boxtimes B$ denote the external product of $A$ and $B$.
Note that we have an equivalence
\[
\widehat{\cel A \times \cel B} \simeq \widehat{\mathcal{C} \times \mathcal{C}}/(A \boxtimes B)
.\]
With respect to this equivalence, left Kan extension along the first arrow in the previous diagram is the forgetful functor.
Thus we recover the usual Day convolution $A \otimes B$ if $F$ is terminal.


As a special case of the above construction, we obtain the \emph{indexed join} on augmented simplicial sets.
This transfers to the indexed join on simplicial sets via terminal augmentation (right Kan extension).
Explicitly, given simplicial sets $A, B$ and a presheaf $F$ over $\int A \times \int B$, an $n$-simplex of $A \star_F B$ is either an $n$-simplex of $A$ or $B$ or a decomposition $[n] = [a] \star [b]$ with $x \in A_a, y \in B_b$ and $t \in F(([a], x), ([b], y))$.
Note that we recover $A \star B$ if $F$ is terminal and $A + B$ if $F$ is initial.
Mirroring the notation for the join, we write
\[
\xymatrix@C-0.5cm{
  A
  \ar[dr]_{\iota_0}
&&
  B
  \ar[dl]^{\iota_1}
\\&
  A \star_F B
}
\]
for the induced Reedy cofibrant span of inclusions.

Write $X_\termaug$ for the terminal augmentation of a semisimplicial set $X$.
Fixing simplicial sets $A$ and $B$, note that right Kan extension along $\int A \times \int B \to \int A_\termaug \times \int B_\termaug$ is parametric left adjoint (recall that a functor $F \co \mathcal{C} \to \mathcal{D}$ where $\mathcal{C}$ has an initial object is called parametric left adjoint if the induced functor $\mathcal{C} \to \mathcal{D}_{\backslash F0}$ is left adjoint).
Therefore, the functor sending a presheaf $F$ over $\int A \times \int B$ to $A \star_F B$ is also parametric left adjoint, \ie the induced functor
\[
\xymatrix{
  \widehat{\int A \times \int B}
  \ar[r]
&
  \widehat{\Delta}_{\backslash A + B}
}
\]
is left adjoint.
Its right adjoint sends a cospan of $m_0 \co A \to Z$ and $m_1 \co B \to Z$ to the presheaf over $\int A \times \int B$ sending $x \in A_a$ and $y \in B_b$ to the set of $z \in Z_{a \star b}$ such that $z[\iota_0] = m_0(x)$ and $z[\iota_1] = m_1(y)$.

Fix now only the simplicial set $A$ and consider the functor sending a simplicial set $B$ and a presheaf $F$ over $\int A \times \int B$ to the inclusion $\iota_0 \co A \to A \star_F B$.
Note that the input data can equivalently be described as presheaves over $\int A_\termaug \times \Delta$, \ie presheaves over $\int A_\termaug$ valued in simplicial sets.
Similar to the above, this functor has a right.
Explicitly, it sends $m_0 \co A \to Z$ to the presheaf over $\int A_\termaug$ valued in simplicial sets sending $[-1]$ to $Z$ and $x \in A_a$ to $Z_{\backslash x}$.
A dual discussion applies when fixing the simplicial set $B$.

We can also consider the indexed join as a functor to simplicial sets from the category of triples $(A, B, F)$ of simplicial sets $A, B$ and a presheaf $F$ over $\int A \times \int B$.
A morphism $(A_0, B_0, F_0) \to (A_1, B_1, F_1)$ in the latter category consists of maps $f \co A_0 \to A_1$ and $g \co B_0 \to B_1$ with a natural transformation $F_0 \to F_1 \circ (\int f \times \int g)$.
Then we have a right adjoint to the indexed join sending a simplicial set $Z$ to the triple $(Z, Z, F)$ where $F((a, x), (b, y))$ is the set of maps $[a] \star [b] \to Z$ that restrict to $x$ and $y$ on $\iota_0$ and $\iota_1$, respectively.

The above discussion of indexed joins applies just as well to semisimplicial sets.
We use the same notation and terminology there.

\subsubsection{Relating $\Sd A$ and $A$}

Given a simplicial set $A$, we have a comparison map between $\Sd A$ and $A$ given by the last vertex map above.
Using classical logic, one may show that this map is a weak equivalence.
Given a semisimplicial set $A$, we do not have such a comparison map between $\Sd A$ and $A$.
However, we can still relate them via a cospan
\[
\xymatrix@!C@C-0.6cm{
  A
  \ar@{>->}[dr]^{\sim}_{\eta_A}
&&
  \Sd A 
  \ar[dl]_{\sim}
\\&
  U L A
}
\]
of weak equivalences.
Here, the left leg is a cofibration and weak equivalence by \cref{simplicial-unit-w-equiv}.
The right leg is the composite
\[
\xymatrix{
  \Sd A
  \ar@{>->}[r]^-{\eta_{\Sd A}}_-{\sim}
&
  U L \Sd A
  \ar[r]^-{m}
&
  U L A
}
\]
where $m$ is $U$ applied to the simplicial last vertex map on $L A$.
Note that $m$ is a homotopy equivalence on representables: as is well-known, the last vertex map $\Delta_+/[n] \to [n]$ is a homotopy equivalence in categories with respect to the interval $[1]$, and the simplicial nerve as well as $U$ (by \cref{simplicial-to-semisimplicial-homotopy}) preserve homotopy equivalences.
By cocontinuity and standard uses of the glueing lemma, it then follows that it is a weak equivalence for all $A$.

The above cospan has the drawback that $ULA$ is infinite even if $A$ is finite.
The goal of this subsubsection is to construct an alternative cospan of weak equivalences relating $A$ and $\Sd A$ that will be Reedy cofibrant, have anodyne legs, and have a finite summit if $A$ is finite.

\bigskip

Fix a semisimplicial set $A$.
Although we lack a map $\Sd A \to A$, we still have this map after applying categories of elements:
\[
\xymatrix{
  \cel \Sd A
  \ar[r]^-{=}
&
  \cel N \scel A
  \ar[r]
&
  L \scel A
  \ar[r]^-{\simeq}
&
  \cel A
\rlap{.}}
\]
Here, the middle map is the semicategorical approximation to the last vertex map.
Dually, we obtain a map $\cel \Sd A \to (\cel A)^\op$ using the semicategorical approximation to the first vertex map.
Taking the composite
\[
\xymatrix{
  F(A) \co (\cel A \times \cel \Sd A)^\op
  \ar[r]
&
  (\cel A)^\op \times \cel A
  \ar[r]^-{\hom}
&
  \Set
}
\]
gives a Reedy cofibrant span
\begin{equation} \label{subdivision-mapping-cone}
\begin{gathered}
\xymatrix@!C@C-1.3cm{
  A
  \ar[dr]_{\iota_0}
&&
  \Sd A
  \ar[dl]^{\iota_1}
\\&
  A \star_{F(A)} \Sd A
}
\end{gathered}
\end{equation}
functorial and cocontinuous in $A$.
For a monomorphism $f \co A \to B$, note that the induced morphism from $F(A)$ to the restriction of $F(B)$ along the sieves $\int A \to \int B$ and $\int \Sd A \to \int \Sd B$ is an isomorphism.

Consider now the legs of~\eqref{subdivision-mapping-cone} as natural transformations in $A$.

\begin{lemma} \label{subdivision-mapping-cone-left-leg}
Pushout application of the left leg of~\eqref{subdivision-mapping-cone} sends cofibrations to anodyne maps.
\end{lemma}

\begin{proof}
By cocontinuity, it will suffice to show that the pushout application to $\partial\Delta^n \to \Delta^n$ is anodyne.
This is the pushout corner map in the following square:
\begin{equation} \label{subdivision-mapping-cone-left-leg:0}
\begin{gathered}
\xymatrix{
  \partial\Delta^n
  \ar[r]
  \ar[d]
&
  \partial\Delta^n \star_{F(\partial\Delta^n)} \Sd \partial\Delta^n
  \ar[d]
\\
  \Delta^n
  \ar[r]
&
  \Delta^n \star_{F(\Delta^n)} \Sd \Delta^n
\rlap{.}}
\end{gathered}
\end{equation}
Given subobjects $u \co U \to \Delta^n$ and $v \co V \to \Sd \Delta^n$, let us abbreviate $U \star' V$ the indexed join of $U$ and $V$ with respect to the restriction of $F$ along $\int u$ and $\int \Sd v$.
The rest of the argument proceeds by working with subobjects of $\Delta^n \star' \Sd \Delta^n$.
The square~\eqref{subdivision-mapping-cone-left-leg:0} isomorphically writes as
\[
\xymatrix{
  \partial\Delta^n
  \ar[r]
  \ar[d]
&
  \partial\Delta^n \star' \Sd \partial\Delta^n
  \ar[d]
\\
  \Delta^n
  \ar[r]
&
  \Delta^n \star' \Sd \Delta^n
\rlap{.}}
\]
Note that the pushout in this square is $\Delta^n \star' \Sd \partial\Delta^n$, so the pushout corner map in~\eqref{subdivision-mapping-cone-left-leg:0} is
\[
\xymatrix{
  \Delta^n \star' \Sd \partial\Delta^n
  \ar[r]
&
  \Delta^n \star' \Sd \Delta^n
\rlap{.}}
\]
We decompose the cofibration $\Sd \partial\Delta^n  \to \Sd \Delta^n$ into a finite composite of pushouts of boundary inclusions.
Note that all attaching maps $f \co \Delta^m \to \Sd \Delta^n$ of this relative cell complex are monomorphisms.
By cocontinuity of the indexed join with fixed first argument, it will suffice to show that
\begin{equation} \label{subdivision-mapping-cone-left-leg:1}
\begin{gathered}
\xymatrix{
  \Delta^n \star' \partial\Delta^m
  \ar[r]
&
  \Delta^n \star' \Delta^m
}
\end{gathered}
\end{equation}
is anodyne for all such attaching maps $f$.
Recall that $f \co \Delta^m \to \Sd \Delta^n$ means a map $f \co [m] \to (\Delta_+)_{/[n]}$.
Denote the first vertex as $f(0) \co [k] \to [n]$.
The map~\eqref{subdivision-mapping-cone-left-leg:1} is a pushout of
\begin{equation} \label{subdivision-mapping-cone-left-leg:2}
\begin{gathered}
\xymatrix{
  \Delta^k \star \partial\Delta^m
  \ar[r]
&
  \Delta^k \star \Delta^m
\rlap{,}}
\end{gathered}
\end{equation}
which is anodyne by iterated applications of \cref{point-join-cofib}.
\end{proof}

\begin{lemma} \label{subdivision-mapping-cone-right-leg}
Pushout application of the right leg of~\eqref{subdivision-mapping-cone} sends cofibrations to anodyne maps.
\end{lemma}

\begin{proof}
We follow the strategy and notation of the proof of \cref{subdivision-mapping-cone-right-leg}.
For $n \geq 0$, we have to show that the pushout corner map in
\begin{equation} \label{subdivision-mapping-cone-right-leg:0}
\begin{gathered}
\xymatrix{
  \Sd \partial\Delta^n
  \ar[r]
  \ar[d]
&
  \partial\Delta^n \star' \Sd \partial\Delta^n
  \ar[d]
\\
  \Sd \Delta^n
  \ar[r]
&
  \Delta^n \star' \Sd \Delta^n
}
\end{gathered}
\end{equation}
is anodyne.
We decompose this into the two separate squares
\begin{equation} \label{subdivision-mapping-cone-right-leg:1}
\begin{gathered}
\xymatrix{
  \Sd \partial\Delta^n
  \ar[r]
  \ar[d]
&
  \partial\Delta^n \star' \Sd \partial\Delta^n
  \ar[d]
\\
  \Sd \Delta^n
  \ar[r]
  \ar[d]
&
  \partial\Delta^n \star' \Sd \Delta^n
  \ar[d]
\\
  \Sd \Delta^n
  \ar[r]
&
  \Delta^n \star' \Sd \Delta^n
}
\end{gathered}
\end{equation}
and show separately that each square has an anodyne pushout corner map.
The bottom pushout corner map is identified as a pushout of
\begin{equation} \label{subdivision-mapping-cone-right-leg:2}
\begin{gathered}
\xymatrix{
  \partial\Delta^n \star \Delta^0
  \ar[r]
&
  \Delta^n \star \Delta^0
}
\end{gathered}
\end{equation}
and is thus a horn inclusion.
For the top square, we take a cellular presentation of $\partial\Delta^n$ and use cocontinuity of the indexed join with fixed right argument, reducing the problem to showing that the pushout corner map in 
\[
\xymatrix{
  \partial\Delta^k \star' \Sd \partial\Delta^n
  \ar[r]
  \ar[d]
&
  \Delta^k \star' \Sd \partial\Delta^n
  \ar[d]
\\
  \partial\Delta^k \star' \Sd \Delta^n
  \ar[r]
&
  \Delta^k \star' \Sd \Delta^n
}
\]
is anodyne for all attaching maps $f \co \Delta^k \to \partial\Delta^n$ (which are mono).
That pushout corner map is a pushout of the one in
\begin{equation} \label{subdivision-mapping-cone-right-leg:3}
\begin{gathered}
\xymatrix{
  \partial\Delta^k \star N ((\int \partial\Delta^n)_{/f})
  \ar[r]
  \ar[d]
&
  \Delta^k \star N ((\int \partial\Delta^n)_{/f})
  \ar[d]
\\
  \partial\Delta^k \star N ((\int \Delta^n)_{/f})
  \ar[r]
&
  \Delta^k \star N ((\int \Delta^n)_{/f})
}
\end{gathered}
\end{equation}
(where we have used the categorical slice of semicategories) as both squares are related by an evident morphism that has pushouts between the top and bottom arrows.
Note that all indexed joins in the latter diagram have reduced to ordinary ones.
Using \cref{anodyne-join-cof}, it will thus suffice to show that the nerve of
\[
\xymatrix{
  (\int \partial\Delta^n)_{/f}
  \ar[r]
&
  (\int \Delta^n)_{/f}
}
\]
is anodyne.
Note that this functor of semicategories is equivalently
\begin{equation} \label{subdivision-mapping-cone-right-leg:4}
\begin{gathered}
\xymatrix{
  [1]^{[n] \smallsetminus \im f} - (1, \ldots, 1)
  \ar[r]
&
  [1]^{[n] \smallsetminus \im f}
\rlap{,}}
\end{gathered}
\end{equation}
so its nerve is the $([n] \smallsetminus \im f)$-fold Leibniz geometric product of $\delta_0$ and hence anodyne by \cref{anodyne-join-cof} since $f$ is not surjective.
\end{proof}

\begin{corollary} \label{subdivision-mapping-cone-legs-anodyne}
The legs of~\eqref{subdivision-mapping-cone} are anodyne for all semisimplicial sets $A$.
\end{corollary}

\begin{proof}
By \cref{subdivision-mapping-cone-left-leg,subdivision-mapping-cone-right-leg}.
\end{proof}

\section{Marked semisimplicial sets}

The development of the homotopy theory of marked semisimplicial sets will be largely analogous to the unmarked case.
As such, we will be less verbose and stress the differences to the unmarked case.

\subsection{Basic notions}

\subsubsection{Marked semicategories}

We have a category $\SemiCat_\m$ of \emph{marked semicategories}, \ie semiscategories equipped with a subset of edges called marked.
We have a forgetful functor $\SemiCat_\m \to \SemiCat$ part of an adjoint quadruple that will be denoted as for marked semisimplicial sets.
The symmetric monoidal structure of $\SemiCat$ extends to $\SemiCat_\m$.
All functors in the adjoint quadruple between $\SemiCat$ and $\SemiCat_\m$ lift canonically to symmetric monoidal functors.

We have a canonical functor $\Cat \to \SemiCat_\m$ that sends a category to the underlying semicategory with those morphisms marked that were isomorphisms in the original category.
This functor is fully faithful and its image contains exactly those marked semicategories $(\mathcal{C}, W)$ such that pre- and postcomposition with marked morphisms is a bijection and every object appears as the domain of a marked morphism.
It has a symmetric monoidal left adjoint $(\SemiCat_\m, \top, \otimes) \to (\Cat, 1, \times)$ sending a semicategory to the free category localized at the marked morphisms.

\subsubsection{Marked semisimplicial sets}

A \emph{marking} $W$ of a semisimplicial set $A$ is given by a subset of its edges.
A morphism $(A, W_A) \to (B, W_B)$ of marked semisimplicial sets is a morphism $f \co A \to B$ of semisimplicial maps such that $f(W_A) \subseteq W_B$.
We obtain a category $\MS$ of marked semisimplicial sets.
It can equivalently be described as the full category of separated presheaves on the category $\Delta_+'$ that under $\Delta_+$ has an object $[1]_\m$ and a morphism $[1] \to [1]_\m$ freely added and the only non-trivial cover given by the singleton cover at that morphism.
Like $\Delta_+$, note that $\Delta_+'$ is a direct category.
We have a retraction $\Delta_+' \to \Delta_+$ that sends the morphism $[1] \to [1]_\m$ to the identity on $[1]$.

\begin{remark}
We could just as well work with presheaves over $\Delta_+'$ as our notion of marked semisimplicial set, with the added benefits of a presheaf category.
However, this is still not enough to define classifiers for fibrations.
For this, we would want that codomains of generating anodyne maps are representable.
This can be achieved by working with presheaves over $\Delta_\m$, the wide subcategory of marked non-empty finite total linear orders restricted to monomorphisms.
The theory works just as well there, although with generating anodyne maps suitably enlarged: marked semisimplicial sets are a full subcategory of sheaves on $\Delta_\m$; one now turns the sheaf condition into weak sheaf condition for fibrant objects.
In fact, one can do the same for weak semicomplicial sets: here, the nicest setting would be presheaves over the full subcategory of semisimplicial sets on cofibrations with representable codomain.
\end{remark}

We write $\M$ for the category of marked simplicial sets, defined similarly to marked semisimplicial sets but with degenerate edges always marked.
Note that it can be equivalently described as the full subcategory of separated presheaves on the category $\Delta'$ that under $\Delta$ has a factorization $[1] \to [1]_\m \to [0]$ of $s_0$ freely added and the only non-trivial cover given by the singleton cover $[1] \to [1]_\m$.
Like $\Delta$, note that $\Delta'$ is an elegant Reedy category; its direct fragment is $\Delta_+'$ nad its inverse fragment consists of $\Delta_-$ together with the map $[1]_\m \to [0]$.
We have a retraction $\Delta' \to \Delta$ that sends the morphism $[1] \to [1]_\m$ to the identity on $[1]$.
This section-retraction pair is compatible with the Reedy structures.

We have a forgetful functor from marked simplicial sets to marked semisimplicial sets given by precomposition with the evident inclusion $\Delta_+' \to \Delta'$.
This has a left adjoint, the \emph{free marked simplicial set} functor, and a right adjoint, the \emph{cofree marked simplicial set} functor.

We denote $(-)^\unmarked$ the forgetful functor from marked semisimplicial sets to semisimplicial sets.
It has as adjoints on the left the \emph{minimal marking} $(-)^\flat$ and on the right the \emph{maximal marking} $(-)^\sharp$.
We will usually leave the minimal marking functor implict, regarding a semisimplicial set as minimally marked by default.
The maximal marking functor has a further right adjoint $\core$ that returns the restriction of the underlying semisimplicial set to those simplices all of whose edges are marked.
We have a corresponding quadruple of adjoint functors in the simplicial case and will use the same notation for it.

Note that forgetting the marking commutes (strictly) with forgetting degeneracy operators.
This commuting square of functors satisfies a number of Beck-Chevalley conditions:
\begin{itemize}
\item
The free (marked) simplicial set functor commutes with minimal marking, forgetting the marking, maximal marking, and taking the core.
\item
The forgetful functor from (marked) simplicial to (marked) semisimplicial sets does not commute with minimal marking, but does commute with maximal marking, fogetting the marking, and taking the core.
\item
The cofree (marked) simplicial set functor does not commute with forgetting the marking, but does commute with maximal marking and taking the core.
\end{itemize}

\medskip

We extend the closed geometric symmetric monoidal structure from semisimplicial sets to marked semisimplicial sets by letting those edges in $(A, W_A) \otimes (B, W_B)$ be marked whose edge components in $A$ or $B$ are marked.
We use the same notation for it as in semisimplicial sets.
All functors in the adjoint quadruple between semisimplicial sets and marked semisimplicial sets lift canonically to symmetric monoidal functors with respect to the geometric symmetric monoidal structures.
As in the unmarked case, the free marked simplicial set functor is symmetric monoidal with respect to the geometric symmetric monoidal structure on marked semisimplicial sets and the cartesian monoidal structure on marked simplicial sets.

\medskip

Note that the fully faithful functor $\Delta_+ \to \SemiCat$ extends to a fully faithful functor $\Delta_+' \to \SemiCat_\m$ that sends $[1]_+$ to the semicategory $[1]$ with maximal marking.
We have a strictly commuting diagram 
\[
\xymatrix{
  \Delta_+
  \ar[r]
  \ar[d]
&
  \SemiCat
  \ar[d]
\\
  (\Delta_+)'
  \ar[r]
&
  \SemiCat_\m
}
\]
where the vertical arrows are the evident inclusion and the minimal marking functor.
The functor $\Delta_+' \to \SemiCat_\m$ gives rise to a nerve $\SemiCat_\m \to \widehat{\Delta_+'}$ valued in separated presheaves, thus lifting to a fully faithful functor $\SemiCat_\m \to \MS$ we call the \emph{nerve} $N$ for marked semisimplicial sets.
Both the nerve and its realization left adjoint commute up to canonical natural isomorphism with the adjoint quadruples for marked semicategories and marked semisimplicial sets.

The geometric symmetric monoidal structure of $\SemiCat$ extends to $\SemiCat_\m$ by letting those morphisms in $(\mathcal{C}, W_\mathcal{C}) \otimes (\mathcal{D}, W_\mathcal{D})$ be marked whose morphisms components in $\mathcal{C}$ or $\mathcal{D}$ are marked.
All functors in the adjoint quadruple between $\SemiCat$ and $\SemiCat_\m$ lift canonically to symmetric monoidal functors with respect to the geometric symmetric monoidal structures.
The marked semisimplicial nerve and its realization are symmetric monoidal functors with respect to the geometric symmetric monoidal structures.

As in the unmarked case, we can understand the geometric symmetric monoidal structure as being induced via argumentwise cocontinuous extension (a generalized Day convolution process) by the geometric symmetric monoidal structure on marked semicategories.

\medskip

\subsubsection{Join and slice}

We extend the join monoidal structure from semisimplicial sets to marked simplicial sets, also reusing notation, by marking only those edges in a join that were already marked in its arguments.
The functor $(-)^\flat, (-)^\unmarked, \core$ lift to monoidal functors between join monoidal structures.
Occasionally, we will have use also for the \emph{marking join monoidal structure} $(0, \star_\m)$, which additionally marks the newly introduced edges going between vertices of different arguments.
The functors $(-)^\unmarked, (-)^\sharp, \core$ lift to monoidal functors between the join monoidal structure on simplicial sets and the marking join monoidal structure.
Given a marked simplicial set $A$, the functors
\[
- \star A, A \star -, - \star_\m A, A \star_\m - \co \M \to (\M)_{\backslash A}
\]
admit right adjoints.
Their respective action on an object $p \co A \to X$ is written
\[
X_{/p}, X_{\backslash p}, X_{\m /p}, X_{\m \backslash p}
.\]
For a cofibration $i \co A \to B$, we have a functor
\[
\M \to (\M)_{\backslash B}
\]
sending $X$ to the pushout
\[
\xymatrix{
  A \star X
  \ar[r]^-{i \star X}
  \ar[d]
&
  A \star_\m X
  \ar[d]
\\
  B \star X
  \ar[r]
&
  \bullet
  \pullbackcorner{ul}
\rlap{.}}
\]
This functor has a right adjoint.
Its action on an object $p \co B \to X$ is written $X_{/(p, pi)}$.
It is equivalently obtained as the pullback
\[
\xymatrix{
  X_{/(p, pi)}
  \ar[r]
  \ar[d]
  \pullbackcorner{dr}
&
  X_{/p}
  \ar[d]
\\
  X_{\m/pi}
  \ar[r]
&
  X_{/pi}
\rlap{.}}
\]
Dually, we define $X_{\backslash (p, pi)}$.

Everything we have said in this paragraph extends to marked semisimplicial sets as well.
We use the same notation in that setting.
The forgetful functor marked simplicial to marked semisimplicial sets is monoidal with respect to join as well as marking join.
The induced oplax monoidal structures of the free marked simplicial set functor are monoidal structures.

\subsection{Weak factorization systems}

The \emph{cofibrations} of marked semisimplicial sets are those maps that are cofibrations (\ie levelwise decidable monomorphisms) on underlying semisimplicial sets and additionally consist of a decidable inclusion of markings.
They are generated under weak saturation by boundary inclusions and the \emph{edge marking inclusion} $(\Delta^1)^\flat \to (\Delta^1)^\sharp$, concretely write as $\omega$-compositions of pushouts of coproducts of boundary inclusions and the edge marking inclusion.
Note that every object is cofibrant.

\begin{lemma} \label{marked-cof-tensor}
Cofibrations are closed under the Leibniz geometric symmetric monoidal structure.
\end{lemma}

\begin{proof}
Analogous to \cref{cof-tensor}.
\end{proof}

\begin{lemma} \label{marked-cof-join}
Cofibrations are closed under the Leibniz join monoidal structure.
\end{lemma}

\begin{proof}
Analogous to \cref{cof-join}.
The Leibniz join of an edge marking inclusion with a boundary inclusion is a relative cell complex of edge marking inclusions.
\end{proof}

\begin{lemma} \label{marked-cof-marking-join}
Cofibrations are closed under the Leibniz marking join monoidal structure.
\end{lemma}

\begin{proof}
Analogous to \cref{marked-cof-join}.
\end{proof}

The forgetful functor to semisimplicial sets and the minimal and maximal marking functors preserve cofibrations.

\medskip

Let us recall some terminology in semisimplicial sets.
For $n \geq 1$ and $0 \leq k \leq n$, the horn $h_k^n \co \Lambda_k^n \to \Delta^n$ is called \emph{inner} if $0 < k < n$, else outer (left outer if $k = 0$ and right outer if $k = n$).
The \emph{critical edge} of an outer horn $h_k^n \co \Lambda_k^n \to \Delta^n$ is given by the distance preserving map $[1] \to [n]$ preserving the initial object if $k = 0$ and the terminal object if $k = n$; we regard it as part of $\Lambda^n_k$ for $n \geq 2$.
Inner horns do not have a critical edge.

The \emph{marked horn inclusions} are given by minimally marked inner horn inclusions together with outer horn inclusions with only the critical edge marked.
The \emph{marking saturation inclusion} is given by $(\Delta^3, \braces{02, 13}) \to (\Delta^3)^\sharp$.
It will correspond to a saturation condition, specifically the 2-out-of-6 property.
The marked horn inclusions and marking saturation inclusion generate the weak factorization system (anodyne, fibration).
The forgetful functor to semisimplicial sets and the maximal marking functor presserve anodyne maps.

\begin{remark} \label{fibrations-arbitrary-marked}
Following up on \cref{fibrations-arbitrary}, the choice of these generators is even more arbitrary than in the unmarked case.
For example, we could have added generators such as $(\Delta^2, \braces{01, 02}) \to (\Delta^2)^\sharp$ corresponding to a 2-out-of-3 property.
As per \cref{fibrant-lift-2-out-of-3,fib-between-fibrant-lift-2-out-of-3} (and \cref{cof-weak-equiv-lift-fib-between-fibrant-marked} for other generators), this would not change the fibrant objects and fibrations between fibrant objects.
Ultimately, for the purpose of developing the homotopy theory of marked semisimplicial sets, we will only care about fibrations that have fibrant codomain (recall that in the theory of quasicategories, fibrations are also only characterized explicitly when they have a quasicategory as codomain).
Thus, we have chosen the least amount of generators needed to characterize these.
\end{remark}

\begin{lemma} \label{marked-anodyne-tensor-cof}
Anodyne maps are closed under Leibniz geometric product with cofibrations.
\end{lemma}

\begin{proof}
This is analogous to \cref{anodyne-tensor-cof}.

Let us make sure that the anydone map~\eqref{anodyne-join-cof:1} from the unmarked case remains anodyne in the current marked setting.
If $a_0, a_1 \geq 0$, this is a (generalized) inner horn, which is anodyne also in the marked case.
Recall that $a_0 = a_1 = -1$ is impossible.
Let now treat the case $a_0 = -1$ and $a_1 \geq 0$ (the case $a_0 \geq 0$ and $a_1 = -1$ is dual).
Since $u_0$ is surjective, we must have $m_0 = -1$, \ie $k = 0$.
Thus, we are in the case of the pushout product of the horn inclusion $\Lambda_0^m \to \Delta^m$ with edge from $0$ to $1$ marked with the boundary inclusion $\partial\Delta^n \to \Delta^n$.
In particular, the relevant marking on $\Delta^m \otimes \Delta^n$ contains the edge from $(0, i)$ to $(1, i)$ for $i \in [n]$.

Observe that $E$ is the singleton subposet of $[a]$ just consisting of the greatest element $a$.
If $a \geq 1$, then~\eqref{anodyne-join-cof:1} is an inner horn, so asume $a = 0$.
Now~\eqref{anodyne-join-cof:1} is the outer horn $\Lambda^{a_1+1}_0 \to \Delta^{a_1+1}$, so we have to make sure that the edge from $0$ to $1$ is mapped to an edge of $\Delta^m \otimes \Delta^n$ that is marked.
Indeed, it is sent to the edge from $(0, v_1(0))$ to $(1, v_1(0))$.
\end{proof}

\begin{lemma} \label{marked-anodyne-join-cof}
Anodyne maps are closed under Leibniz join with cofibrations.
Anodyne maps are closed under join with marked semisimplicial sets.
\end{lemma}

\begin{proof}
Analogous to \cref{anodyne-join-cof}.
\end{proof}

\begin{lemma} \label{marked-anodyne-marking-join-mono}
Anodyne maps are closed under Leibniz marking join with cofibrations.
\end{lemma}

\begin{proof}
Analogous to \cref{anodyne-join-cof}.
\end{proof}

We have corresponding adjoint closure statements for fibrations.

A map between fibrant objects that lifts against cofibrations is called a \emph{trivial fibration}.

\subsection{Fundamental category}

The explicit construction of the fundamental groupoid of a fibrant semisimplicial sets extends more or less mechanically to an explicit construction of the fundamental category of a fibrant marked semisimplicial set.
We note only the differences.
For this subsection, fix a fibrant marked semisimplicial set $X$.

A \emph{constant loop} in $X$ is a constant loop in $\core(X)$.
Recall that $\core(X)$ is a fibrant semisimplicial set, so the discussion of constant loops from the unmarked case applies.

Recall the maximal marking monad $((-)^\unmarked)^\sharp$.
The following lemma shows that markings in fibrant marked semisimplicial sets also have a 2-out-of-3 property.

\begin{lemma} \label{fibrant-lift-2-out-of-3}
Fibrant objects lift against the Leibniz application of the unit of the maximal marking monad to 2-dimensional horn inclusions.
Concretely, given fibrant $X$ and a triangle $\sigma \in X_2(f, g, h)$ with $f \in X_1(a, b), g \in X_1(b, c), h \in X_1(a, c)$, if two out of $f, g, h$ are marked, then so is the third.
\end{lemma}

\begin{proof}
We will show that $h$ is marked if $f$ and $g$ are.
The other cases are analogous.

Recall the anodyne map $j \co \Delta^0 \to L \to \Delta^0 \star L$ in semisimplicial sets.
Using \cref{marked-anodyne-join-cof}, we have that
\[
\xymatrix@C+1cm{
  \Delta^2
  \ar[r]^-{\Delta^0 \star j^\sharp \star \Delta^0}
&
  \Delta^0 \star (\Delta^0 \star L)^\sharp \star \Delta^0
}
\]
is anodyne.
Extending $\sigma$ along this map and then restricting along the map induced by $\Delta^1 \to L \to \Delta^0 \star L$, we obtain a tetrahedron $\Delta^3 \to X$ that has $f$ as $02$-edge, $g$ as $13$-edge, and $h$ as $03$-edge.
Lifting against the marking saturation inclusion, we deduce that $h$ is marked.
\end{proof}

\begin{remark} \label{fib-between-fibrant-lift-2-out-of-3}
Fibrations between fibrant objects also have the lifting property of \cref{fibrant-lift-2-out-of-3}.
This follows abstractly since the maps lifted against are epi.
\end{remark}

\newcommand{\va}{{\mathsf{a}}}
\newcommand{\vb}{{\mathsf{b}}}

Let $\va, \vb$ denote the vertices of $P$, in this order.
We say that parallel edges $f, g \in X_1(a, b)$ forming a map $[f, g] \co P \to X$ are \emph{related}, written $f \approx g$, if $P \to X$ lifts to a map
\[
(\Delta^0 \star P, \braces{\Delta^0 \star a}) \to X
.\]
Equivalently, the marked semisimplicial set $X_{/([f, g], a)}$ contains a vertex.
Note that $X_{/(f, a)}$ and $X_{/(g, a)}$ are trivially fibrant and that $X_{/([f, g], a)} \to X_{/(f, a)}$ and $X_{/([f, g], a)} \to X_{/(g, a)}$ are fibrations.
Thus, if $f \approx g$, then the latter two maps are surjective on vertices.

Note that we have anodyne maps as indicated:
\[
\xymatrix@R-0.5cm@C-2cm{
&
  P
  \ar@{>->}[dr]
  \ar@{>->}[dl]
\\
  (\Delta^0 \star P, \braces{\Delta^0 \star a})
  \ar@{>->}[dr]_(0.4){\anod}
&&
  (P \star \Delta^0, \braces{b \star \Delta^0})
  \ar@{>->}[dl]^(0.4){\anod}
\\&
  (\Delta^0 \star P \star \Delta^0, \braces{\Delta^0 \star a, b \star \Delta^0})
\rlap{.}}
\]
As before, this implies that the notion of relatedness is invariant under taking opposites.

Just as before, one sees that relatedness is an equivalence relation.
Using \cref{fibrant-lift-2-out-of-3}, one sees that markings are invariant under relatedness.

Recall the anodyne map $P \to Q$ from the unmarked case.
It was constructed as the functorial action of taking cokernels on the map of arrows from $\partial\Delta^1 \to \Delta^1$ to $\Delta^0 + \Delta^1 \to \Delta^2$.
This map is Reedy left anodyne, where left anodyne means the weak saturation of inner and left horn inclusions.
It follows that $P \to Q$ is left anodyne, hence that $\Delta^0 \star P \to \Delta^0 \star Q$ is inner anodyne.
Given triangles $\sigma_0 \in X_2(f_0, g, h_0)$ and $\sigma_1 \in X_2(f_1, g, h_1)$ with first vertex $a$, we form a map $[\sigma_0, \sigma_1] \co Q \to X$ and then have a span
\[
\xymatrix@C-1cm{
&
  X_{/([\sigma_0, \sigma_1], a)}
  \ar@{->>}[dl]_{\triv}
  \ar@{->>}[dr]
\\
  X_{/([f_0, f_1], a)}
&&
  X_{/([h_0, h_1], a)}
\rlap{.}}
\]
It follows that $f_0 \approx f_1$ implies $h_0 \approx h_1$.
We obtain a dual statement for triangles that share the $d_2$-indexed edge.
Combining these using transitivity, given triangles $\sigma_0 \in X_2(f_0, g_0, h_0)$ and $\sigma_1 \in X_2(f_1, g_1, h_1)$, if $f_0 \approx f_1$ and $g_0 \approx g_1$, then $h_0 \approx h_1$.

The \emph{fundamental category} $\tau_1 X$ is defined analogously to the groupoidal case.
The objects are given by vertices $X_0$ and the morphisms from $a$ to $b$ are given by the quotient $X_1(a, b)/{\approx}$ of edges under relatedness.
Composition is given by fillings of $h^2_1$-horns.
The identity on $a \in X_0$ is given by a constant loop on $a$.
This makes $\tau_1 X$ into a category.

Note that marked edges represent isomorphisms by fillings of marked outer 2-dimensional horns.
Furthermore, any edge that represents an isomorphism is marked by lifting against the marking saturation inclusion and invariance of markings under relatedness.
Thus, an edge is marked precisely if it represents an isomorphism.

\subsection{Enrichment in categories}

We have a nerve
\[
\xymatrix{
  \Cat
  \ar[r]
&
  \SemiCat_\m
  \ar[r]^-{N}
&
  \MS
}
\]
from categories.
Its left adjoint is the \emph{fundamental category} functor $\tau_1$.
Given a marked semisimplicial set $A$, one may describe $\tau_1 A$ as the free category on the graph $A_1 \rightrightarrows A_0$ quotiented under the relation on morphisms given by triangles $A_2$ and further localized at morphisms coming from marked edges.
It thus depends only on the truncation of $A$ to three levels (plus markings).

If $A$ is fibrant, we see that $\tau_1 A$ as defined here coincides up to isomorphism with $\tau_1 A$ as defined in the previous subsection.
Note for non-fibrant $A$, an isomorphism in $\tau_1 A$ does not need to come from a zig-zag of marked edges in $A$.

Since
\[
\tau_1 \co (\MS, \top, \otimes) \to (\Cat, 1, \times)
\]
is symmetric monoidal, we can regard marked semisimplicial sets as enriched in categories with the hom-category from $A$ to $B$ given by $\tau_1(\hom(A, B))$ (note that $\hom(A, B)$ is fibrant if $B$ is).
When restricted to to fibrant objects, we will also call this the \emph{homotopy 2-category} of marked semisimplicial sets.
It is analogous to the 2-category of quasicategories of~\cite{riehl-verity:2-cat-of-quasicat}, but has less strict structure (for example, it is not cartesian closed, but only has a weak notion of exponential).
In order to avoid excessive overloading of notation, identities, composition, and inverses in hom-categories will be written $\id_2$, $- \circ_2 -$, and $\inv_2$.

\medskip

In the category of categories, let us call cofibrations the functors whose action on objects is a decidable monomorphism, fibrations the isofibrations, and weak equivalences the equivalences.
Recall that this forms a model category.

We have that $\tau_1$ maps cofibrations and anodyne maps of marked semisimplicial sets to cofibrations and trivial cofibrations of categories, respectively: by cocontinuity, it suffices to check this for the respective generators; for these, it is clear.

We also have that $\tau_1$ maps fibrations between fibrant objects of marked semisimplicial sets to fibration of categories: given a fibration $Y \to X$ of marked semisimplicial sets, an isomorphism $a \to b$ in $\tau_1(X)$ and a lift $y$ of $b$ to $\tau_1(Y)$, recall that $a \to b$ is represented by a marked edge from $a$ to $b$; using a marked $h^1_1$-horn lifting, we can it to $Y$.
In contrast to the groupoidal case, this observation fails for fibrations with non-fibrant codomain.

Given $X \in \widehat{\Delta_+}$ and $f, g \in X_1(a, b)$, we call $f$ and $g$ \emph{related}, written $f \approx g$, if $f, g \co \Delta^1 \to X$ become equal after applying $\tau_1$.
If $X$ is fibrant, we see that this notion of relatedness coincides with the one defined in the previous subsection.

\begin{lemma} \label{lifting-relatedness-marked}
Let $p \co Y \to X$ be a fibration in marked semisimplicial sets.
Given $f \in Y_1(a, b)$ and $\overline{g} \in X_1(pa, pb)$ such that $pf \approx \overline{g}$, there is $g \in Y_1(a, b)$ such that $pg = \overline{g}$ and $f \approx g$.
\end{lemma}

\begin{proof}
After taking a fibrant replacement of $X$, we see that the claim reduces to the case where $X$ is fibrant.
In that case, we way work with the explicit description of $\tau_1 Y$ and $\tau_1 X$ from the previous subsection.
The claim then follows by solving a lifting problem
\[
\xymatrix{
  \Delta^1
  \ar[r]^{f}
  \ar@{>->}[d]_{\anod}
&
  Y
  \ar@{->>}[d]
\\
  (\Delta^0 \star P, \Delta^0 \star \braces{a})
  \ar[r]
  \ar@{.>}[ur]
&
  X
}
\]
where the bottom map is the witness of $pf \approx pg$ and the left map comes from one of the inclusions $\Delta^1 \to P$ and is observed to be anodyne.
\end{proof}

\subsection{Homotopies}

The \emph{interval} in marked semisimplicial sets is imported from semisimplicial sets via maximal marking.
As before, the interval together with the geometric monoidal structure induces a notion of \emph{homotopy} $h \co f_0 \sim f_1$ between maps $f_0, f_1 \co X \to Y$ of marked semisimplicial sets that is in bijection with marked edges between the transposes $\overline{f_0}$ and $\overline{f_1}$ in $\hom(X, Y)$ (\ie an edges in $\core(\hom(X, Y))$ between the imgea of $\overline{f_0}$ and $\overline{f_1}$).

If $Y$ is fibrant, such a marked edge exists precisely if there is an isomorphism between the images of $\overline{f_0}$ and $\overline{f_1}$ in $\tau_1(\hom(X, Y))$.
Thus, the existence of homotopies with fibrant target can be tested in the categorical enrichment of marked semisimplicial sets.

Following Riehl-Verity, when working with homotopies targetting fibrant objets, we will freely exploit the relation to the categorical enrichment and regard homotopies as idomorphisms in hom-categories with relatedness of homotopies given by equality in the hom-category, yielding notions of horizontal composition and vertical composition and inversion for homotopies that is closed under relatedness.
Given an isomorphism in $\hom(B, C)$ represented by homotopy $H \co I \otimes B \to C$, note that the right whiskering with an object of $\hom(A, B)$ given by a map $f \co A \to B$ is represented by $H \circ (I \otimes f)$ and that the left whiskering with an object of $\hom(C, D)$ given by a map $g \co C \to D$ is represented by $g \circ H$.

As before, we have a notion of \emph{homotopy equivalence}.
Homotopy equivalences become actual equivalences in the categorical enrichment of marked semisimplicial sets (and come from those if $A$ and $B$ are fibrant).
As before, we say this homotopy equivalence is \emph{coherent} if $f \circ H \approx K \circ (I \otimes f)$.

\begin{lemma}[Graduate Lemma] \label{graduate-lemma-marked}
For any homotopy equivalence $(f, g, H, K)$ as above, there is a replacement $K'$ of $K$ such that $(f, g, H, K')$ forms a coherent homotopy equivalence.
\end{lemma}

\begin{proof}
Same proof as \cref{graduate-lemma}.
\end{proof}

As before, we have the notion of \emph{strong homotopy equivalence}.

\subsection{Homotopy theory}

We obtain statements as in \cref{homotopy-theory}.
The omitted proofs are adapted verbatim.

\begin{lemma} \label{h-equiv-coherent-to-strong-marked}
Let $(f, g, H, K)$ form a coherent homotopy equivalence between objects $U$ and $V$ with $V$ fibrant.
\begin{enumerate}
\item
If $f$ is a fibration, there is $H' \approx H$ such that $(f, g, H', K)$ forms a strong homotopy equivalence,
\item
If $f$ is a cofibration, there is $K' \approx K$ such that $(f, g, H, K')$ forms a strong homotopy equivalence.
\qed
\end{enumerate}
\end{lemma}

\begin{corollary} \label{anodyne-criteria-marked}
The following are equivalent for a cofibration $m$ between fibrant objects:
\begin{enumerate}
\item $m$ is a homotopy equivalence,
\item $m$ is a strong homotopy equivalence,
\item $m$ is anodyne.
\qed
\end{enumerate}
\end{corollary}

\begin{corollary} \label{triv-fib-criteria-marked}
The following are equivalent for a fibration $p$ between fibrant objects:
\begin{enumerate}
\item $p$ is a homotopy equivalence,
\item $p$ is a strong homotopy equivalence,
\item $p$ is a trivial fibration.
\qed
\end{enumerate}
\end{corollary}

Since equivalences in a 2-category satisfy 2-out-of-6, so do homotopy equivalences between fibrant objects.
Thus, \cref{anodyne-criteria-marked,triv-fib-criteria-marked} show that the full subcategory $\MSfib$ of marked semisimplicial sets on fibrant objects form a model structure (cofibrations, homotopy equivalences, fibrations).

\begin{theorem} \label{semisimplicial-sets-fib-cat-marked}
Fibrant marked semisimplicial sets form a fibration category.
\end{theorem}

\begin{proof}
Analogous to \cref{semisimplicial-sets-fib-cat}.
\end{proof}

The notion of homotopy equivalence between fibrant objects extends to a notion of \emph{weak equivalence} between arbitrary marked semisimplicial sets as in \cref{homotopy-theory}.

\begin{lemma} \label{anod-equal-homotopic-marked}
Consider a diagram
\[
\xymatrix{
  A
  \ar[r]_{\anod}^{j}
&
  X
  \ar@<+0.5em>[r]^{f}
  \ar@<-0.5em>[r]_{g}
&
  Y
}
\]
with $X$ and $Y$ fibrant, $j$ anodyne, and $jf = jg$.
Then $f$ and $g$ are homotopic.
\qed
\end{lemma}

\begin{corollary} \label{span-anod-h-equiv-marked}
Given a span
\[
\xymatrix@C-0.5cm{
&
  A
  \ar[dl]_{\anod}
  \ar[dr]^{\anod}
\\
  X_1
&&
  X_2
}
\]
with $X_1$ and $X_2$ fibrant, there is a map $X_1 \to X_2$ under $A$ that is a homotopy equivalence.
\qed
\end{corollary}

\begin{lemma} \label{def-weak-equiv-invariant-marked}
Given a commuting diagram
\[
\xymatrix{
  X_1
  \ar[d]
&
  A  
  \ar[r]^-{\anod}
  \ar[l]_-{\anod}
  \ar[d]
&
  X_2
  \ar[d]
\\
  Y_1
&
  B
  \ar[r]^-{\anod}
  \ar[l]_-{\anod}
&
  Y_2
}
\]
with anodyne maps and fibration as indicated such that $X_1, X_2, Y_1, Y_2$ are fibrant, if $X_1 \to Y_1$ is a homotopy equivalence, then so is $X_2 \to Y_2$.
\qed
\end{lemma}

\begin{corollary} \label{weak-equiv-2-out-of-6-marked}
Marked semisimplicial sets form a homotopical category, \ie weak equivalences satisfy 2-out-of-6.
\qed
\end{corollary}

\begin{lemma} \label{weak-equiv-lift-hom-fibrant-marked}
Weak equivalences lift against fibrant objects up to homotopy.
\qed
\end{lemma}

\begin{corollary} \label{cof-weak-equiv-lift-fib-between-fibrant-marked}
Cofibrations that are weak equivalences lift against fibrations between fibrant objects.
\qed
\end{corollary}

A marked semisimplicial set is called \emph{finite} if its underlying semisimplicial set is finite.

\begin{lemma} \label{finite-weak-equiv-marked}
In the full subcategory of finite marked semisimplicial sets, the class of weak equivalences is the closure under 2-out-of-6 of pushouts of marked horn inclusions.
\qed
\end{lemma}

\begin{lemma} \label{homotopies-limit-colimit-marked}
Homotopies are closed under limit and colimit, \ie the limit and colimit functors from $[\mathcal{D}, \MS]$ to $\MS$ preserve the homotopy relation.
\qed
\end{lemma}

\begin{corollary} \label{homotopies-product-coproduct-marked}
Homotopies are closed under product and coproduct (of arbitrary arity), \ie if $f_i \sim g_i$ for all $i \in I$, then $\prod_{i \in I} f_i \sim \prod_{i \in I} g_i$ and $\coprod_{i \in I} f_i \sim \coprod_{i \in I} g_i$.
\qed
\end{corollary}

\begin{corollary} \label{h-equiv-product-coproduct-marked}
Homotopy equivalences are closed under product and coproduct (of arbitrary arity).
\qed
\end{corollary}

\begin{lemma} \label{hom-pres-homotopy-marked}
The functors $- \otimes -$ and $\hom(-, -)$ preserve homotopies in each arguments.
\qed
\end{lemma}

\begin{corollary} \label{hom-pres-h-equiv-marked}
The functors $- \otimes -$ and $\hom(-, -)$ preserve homotopy equivalences in each arguments.
\qed
\end{corollary}

\begin{corollary} \label{hom-pres-weak-equiv-dom-marked}
For $X$ fibrant, the functor $\hom(-, X)$ preserves weak equivalences.
\qed
\end{corollary}

\begin{corollary} \label{cof-weak-equiv-hom-fib-marked}
For $X$ fibrant, the functor $\hom(-, X)$ sends cofibrations that are weak equivalences to trivial fibrations.
\qed
\end{corollary}

\begin{lemma} \label{weak-equiv-joyal-characterization-marked}
The following are equivalent for a map $f$:
\begin{enumerate}
\item $f$ is a weak equivalence,
\item $\hom(f, X)$ is a homotopy equivalence for all fibrant $X$,
\item $\tau_0(\hom(f, X))$ is a bijection all fibrant $X$, \ie precomposition with $f$ induces a bijection of homotopy classses when the target is fibrant.
\end{enumerate}
\qed
\end{lemma}

\begin{corollary} \label{cof-weak-equiv-triv-fib-characterization-marked}
A cofibration $m$ is a weak equivalence exactly if the fibration $\hom(m, Z)$ is trivial for all fibrant $Z$.
\qed
\end{corollary}

\begin{corollary} \label{anodyne-is-weak-equiv-marked}
Anodyne maps are weak equivalences.
\qed
\end{corollary}

\begin{corollary} \label{weak-equivalences-as-closure-of-andoyne-marked}
Weak equivalences are the closure of anodyne maps, more specifically just $\omega$-compositions of pushouts of coproducts of generating anodyne maps, under 2-out-of-6.
\qed
\end{corollary}

\begin{corollary} \label{h-equiv-is-weak-equiv-marked}
Homotopy equivalences are weak equivalences.
\qed
\end{corollary}

\begin{corollary} \label{weak-equiv-coproduct-marked}
Weak equivalences are closed under coproduct (of arbitrary arity).
\qed
\end{corollary}

\begin{corollary} \label{weak-equiv-retracts-marked}
Weak equivalences are closed under retract.
\qed
\end{corollary}

\begin{lemma} \label{cof-weak-equiv-closure-marked}
Weak equivalences that are cofibrations are closed under pushout as well as $\omega$-composition.
\qed
\end{lemma}

\begin{theorem} \label{semisimplicial-sets-cofibration-category-marked}
Marked semisimplicial sets form a cofibration category.
\end{theorem}

\begin{proof}
Analogous to \cref{semisimplicial-sets-cofibration-category}.
\end{proof}

\begin{lemma} \label{geometric-product-pres-weak-equiv-marked}
The geometric product preserves weak equivalences, \ie if $u$ and $v$ are weak equivalences, then so is $u \otimes v$.
\qed
\end{lemma}

\begin{corollary} \label{cof-w-equiv-tensor-marked}
If $u$ is a weak equivalence and cofibration and $v$ is any map, then the Leibniz geometric product $u \hatotimes v$ is a weak equivalence.
In particular, cofibrations that are weak equivalences are closed under Leibniz geometric product with cofibrations.
\qed
\end{corollary}

\begin{lemma}
There is map of cylinders
\[
\xymatrix@C+0.5cm@R-0.5cm{
&
  I \otimes (A \star B)
  \ar@{.>}[dd]
\\
  A \star B
  \ar[ur]^(0.45){\delta_0 \otimes (A \star B)}
  \ar[dr]
  \ar@<-0.2em>@{{}{ }{}}[dr]_-(0.3){(\delta_0 \otimes A) \star (\delta_0 \otimes B)}
&&
  A \star B
  \ar[ul]_(0.45){\delta_1 \otimes (A \star B)}
  \ar[dl]
  \ar@<+0.2em>@{{}{ }{}}[dl]^-(0.3){(\delta_1 \otimes A) \star (\delta_1 \otimes B)}
\\&
  (I \otimes A) \star (I \otimes B)
}
\]
natural in marked semisimplicial sets $A$ and $B$.
\qed
\end{lemma}

\begin{corollary} \label{join-pres-homotopy-and-h-equiv-marked}
The join monoidal structure preserves homotopies and homotopy equivalences, \ie given homotopies $f_0 \sim f_1$ and $g_0 \sim g_1$, we have $f_0 \star g_0 \sim f_1 \star g_1$, and given homotopy equivalences $u$ and $v$, then $u \star v$ is a homotopy equivalence.
\qed
\end{corollary}

\begin{corollary} \label{join-pres-weak-equiv-marked}
The join monoidal structure preserves weak equivalences, \ie given weak equivalences $u$ and $v$, then $u \star v$ is a weak equivalence.
\qed
\end{corollary}

\begin{corollary} \label{cof-w-equiv-join-cof-marked}
The Leibniz join of a cofibration that is a weak equivalence with any map is a weak equivalence.
Cofibrations that are weak equivalences are closed under Leibniz join with cofibrations.
\qed
\end{corollary}

\subsection{Comparison with unmarked semisimplicial sets}

Let us relate the homotopical aspects of semisimplicial and marked semisimplicial sets.

The minimal marking functor $(-)^\flat$ preserves cofibrations, but not weak equivalences (for example, the one-dimensional horn inclusions are not weak equivalences in marked semisimplicial sets).
The marking forgetting functor $(-)^\unmarked$ preserves cofibrations and anodyne maps.
By adjointness, the maximal marking functor $(-)^\sharp$ preserves fibrations.

\begin{lemma} \label{forget-marking-pres-w-equiv}
The marking forgetting functor preserves weak equivalences.
\end{lemma}

\begin{proof}
Given a map $f \co A \to B$ of marked semisimplicial sets and a semisimplicial set $X$, note by adjointness that $\hom(f^\unmarked, X)$ is a homotopy equivalence exactly if $\hom(f, X^\sharp)$ is a homotopy equivalence.
Then claim then follows Using the characterizations \cref{weak-equiv-joyal-characterization,weak-equiv-joyal-characterization-marked} of weak equivalences.
\end{proof}

The maximal marking functor preserves cofibrations and anodyne maps.

\begin{lemma} \label{maximal-marking-pres-w-equiv}
The maximal marking functor preserves weak equivalences.
\end{lemma}

\begin{proof}
This follows directly from our definition of weak equivalencess.
All the ingredients are preserved: anodyne maps, fibrant objects, homotopy equivalences.
\end{proof}

\subsection{Comparison with marked simplicial sets}

Let us write $(\Delta_-', \Delta_+')$ for the Reedy structure on $\Delta'$.
We write $L \dashv U \dashv R$ for the forgetful functor from marked simplicial to marked semisimplicial sets and left and right adjoint.
We denote the induced monad on marked semisimplicial sets by $(UL, \eta, \mu)$.
Recall that $L$ sends the geometric to the cartesian symmetric monoidal structure.

The \emph{interval} in marked simplicial sets is given by the image of the semisimplicial inverval under $L$ and will be denoted using the same symbols.
Note that it has a contraction and connections.
It is equivalently the image of the interval in simplicial sets under maximal marking.
With the cartesian monoidal structure, this induces associated notions of cylinder, homotopy, homotopy equivalence.
It follows that $L$ preserves homotopies and homotopy equivalences.
As in the unmarked case, we also have the corresponding statement for $U$.

\begin{lemma} \label{simplicial-vs-semisimplicial-cylinder-marked}
Functorially in a simplicial set $A$, there is cylinder morphism from the semisimplicial cylinder on $UA$ to the image under $U$ of the simplicial cylinder on $A$.
\qed
\end{lemma}

\begin{corollary} \label{simplicial-to-semisimplicial-homotopy-marked}
The forgetful functor $U$ functorially preserves homotopies and homotopy equivalences.
\qed
\end{corollary}

The following is analogous to the unmarked case.

Let $S'$ be the subcategory of $[[1], \Delta']$ with objects restricted to maps in $\Delta_-'$ and codomain part of morphisms restricted to maps in $\Delta_+'$.
We put a coverage it generated by the singleton cover of the identity on $[1]_\m$ by the identity on $[1]$.
We have a two-sided opfibration
\begin{equation} \label{simplicial-vs-semisimplicial-span-marked}
\begin{gathered}
\xymatrix@C-0.3cm{
&
  S'
  \ar[dl]_{\ev_0}
  \ar[dr]^{\ev_1}
\\
  \Delta'
&&
  \Delta_+'
}
\end{gathered}
\end{equation}
with $\ev_0$ a discrete Grothendieck fibration.
Thinking of marked simplicial and marked semisimplicial sets as full subcategories of discrete Grothendieck fibrations over $\Delta'$ and $\Delta_+'$ on separated objects, the functor $L$ is given by the linear polynomial functor (in our generalized sense) above; note that pullback along $\ev_1$ and postcomposition with $\ev_0$ preserves separatedness.
Explicitly, given a marked semisimplicial set $A$, recall that the underlying simplicial set $LA^\unmarked$ of $LA$ has as $n$-simplices pairs $(s, a)$ where $s \co [n] \to [k]$ is epi and $a \in A_k$.
Thus, the edges in $LA^\unmarked$ are of the form $(\id_{[1]}, f)$ with $f \in A_1$ or $(s_0, a)$ with $a \in A_0$.
The marking of $LA$ consists of those edges $(\id_{[1]}, f)$ for which $f$ is marked in $A$ and all edges $(s_0, a)$ (as they are degenerate).

Given a category $A$, let us write $[A, \Delta']_-^+$ for the subcategory of $[A, \Delta']$ with objects restricted to functors $A \to \Delta_-'$ and morphisms restricted to natural transformations with components in $\Delta_+'$.
We have a pullback
\[
\xymatrix{
  [[1], \Delta']_-^+
  \ar[r]
  \ar[d]_{\ev_0}
  \pullbackcorner{dr}
&
  S
  \ar[d]^{\ev_0}
\\
  \Delta_+'
  \ar[r]^-{i}
&
  \Delta'
\rlap{.}}
\]
The composite linear polynomial functor $UL$ (on presheaves over $\Delta_+'$) is specified by the span
\[
\xymatrix@C-0.5cm{
&
  [[1], \Delta']_-^+
  \ar[dl]_{\ev_0}
  \ar[dr]^{\ev_1}
\\
  \Delta_+'
&&
  \Delta_+'
}
\]
and has $n$-th power
\[
\xymatrix@C-0.5cm{
&
  [[n], \Delta']_-^+
  \ar[dl]_{\ev_0}
  \ar[dr]^{\ev_n}
\\
  \Delta_+'
&&
  \Delta_+'
\rlap{.}}
\]
The unit and multiplication of the monad $UL$ are given by maps of linear polynomial functors
\begin{align} \label{free-simplicial-set-unit-multiplication-marked}
\begin{gathered}
\xymatrix@R-0.4cm{
&
  [[0], \Delta']_-^+
  \ar[dl]_{\ev_0}
  \ar[dr]^{\ev_0}
  \ar[dd]^{- \circ s_0}
\\
  \Delta_+'
&&
  \Delta_+'
\rlap{,}\\&
  [[1], \Delta']_-^+
  \ar[ul]^{\ev_0}
  \ar[ur]_{\ev_1}
}
\end{gathered}
&&
\begin{gathered}
\xymatrix@R-0.4cm{
&
  [[2], \Delta']_-^+
  \ar[dl]_{\ev_0}
  \ar[dr]^{\ev_0}
  \ar[dd]^{- \circ d_1}
\\
  \Delta_+'
&&
  \Delta_+'
\rlap{.}\\&
  [[1], \Delta']_-^+
  \ar[ul]^{\ev_0}
  \ar[ur]_{\ev_1}
}
\end{gathered}
\end{align}
So $UL$ is a linear polynomial monad (in our generalized sense), in particular a cartesian monad.

\begin{lemma} \label{free-simplicial-set-homotopically-idempotent-marked}
The natural transformations $\eta UL$ and $UL\eta$ are homotopic in $[\MS, \MS]$.
\end{lemma}

\begin{proof}
Analogous to \cref{free-simplicial-set-homotopically-idempotent}.
In fact, the underlying homotopy in semisimplicial sets will be the one constructed there.
So it is just about preservation of marking.

As before, it suffices to construct a map $H'$ fitting into the diagram
\begin{equation} \label{free-simplicial-set-homotopically-idempotent-marked:0}
\begin{gathered}
\xymatrix{
  [[1], \Delta']_-^+
  \ar[d]_{0^* \times_{\Delta'} \id}
  \ar[dr]^{- \circ s_0}
\\
  \Delta'/[1]_\m \times_{\Delta'} [[1], \Delta']_-^+
  \ar@{.>}[r]^-(0.3){H'}
&
  [[2], \Delta']_-^+
\\
  [[1], \Delta']_-^+
  \ar[u]^{1^* \times_{\Delta'} \id}
  \ar[ur]_{- \circ s_1}
}
\end{gathered}
\end{equation}
of (summits of spans specifiying) linear polynomial functors over $UL$ (with indexing of $[[2], \Delta']_-^+$ given by $- \circ d_1$ and pullback on the left taken with respect to $\ev_0$).

We will construct $H'$ as an extension of $H$ as in~\eqref{free-simplicial-set-homotopically-idempotent:0} through the evident inclusions:
\begin{equation} \label{free-simplicial-set-homotopically-idempotent-marked:1}
\begin{gathered}
\xymatrix{
  \Delta/[1] \times_{\Delta} [[1], \Delta]_-^+  
  \ar@{.>}[r]^-{H}
  \ar[d]
&
  [[2], \Delta]_-^+  
  \ar[d]
\\
  \Delta'/[1]_\m \times_{\Delta'} [[1], \Delta']_-^+
  \ar[r]^-{H'}
&
  [[2], \Delta']_-^+
\rlap{.}}
\end{gathered}
\end{equation}
An object of the domain of $H'$ not in the image of the left map is an edge of the form
\[
\xymatrix{
  [1]_\m
&
  [1]_\m
  \ar[l]
  \ar@{->>}[r]
&
  [k]
}
\]
(note that there only six objects of this form).
The corresponding unmarked edge is given by restricting along $[1] \to [1]_\m$:
\[
\xymatrix{
  [1]
  \ar@{>->}[d]
&
  [1]
  \ar@{.>}[l]
  \ar@{.>>}[r]
  \ar@{>->}[d]
&
  [k']
  \ar@{>.>}[d]
\\
  [1]_\m
&
  [1]_\m
  \ar[l]
  \ar@{->>}[r]
&
  [k]
\rlap{.}}
\]
We take the image of the top row object under $H$ and construct the image of the bottom row object under $H'$ using a pushout:
\[
\xymatrix{
  [1]
  \ar@{->>}[r]
  \ar@{>->}[d]
&
  [m']
  \ar@{->>}[r]
  \ar@{>.>}[d]
&
  [k']
  \ar@{>->}[d]
\\
  [1]_\m
  \ar@{.>>}[r]
&
  [m]
  \ar@{.>>}[r]
  \pullbackcorner{ul}
&
  [k]
\rlap{.}}
\]
This is easily seen to extend $H$ to $H'$ as in~\eqref{free-simplicial-set-homotopically-idempotent-marked:1} and to cohere with the endpoint inclusions as in~\eqref{free-simplicial-set-homotopically-idempotent-marked:0}.
Here, recall again that $H'$ is a map between discrete Grothendieck fibrations over $[[1], \Delta']$, hence it is enough to define $H'$ on objects and check a coherence condition on morphisms.c
\end{proof}

\subsection{Comparison with quasicategories}
\label{subsection:comparison-quasicategories}

Consider weak factorization systems (cofibration, trivial fibration) and (anodyne map, naive fibration) in marked simplicial sets induced by the corresponding ones in marked semisimplicial sets via the adjunction $L \dashv U$.
Note that a map is a cofibration exactly if it is a cofibration of underlying simplicial sets (\ie a monomorphism with decidable latching object inclusions) and the inclusion of markings is decidable.
Using classical logic, there is a Cisinski model structure on marked simplicial sets where the fibrant objects and fibrations between fibrant objects coincide with the naive ones, so will be drop the adjective `naive' in these cases.
Except for an unimportant detail (marking only edges rather than simplices in all dimensions), this is the model structure for weak 1-complicial sets of Verity~\cite{verity:complicial}.
Using the established classical theory, it is easily seen that $(-)^\flat \dashv (-)^\unmarked$ forms a Quillen equivalence from the Joyal model structure for quasicategories to the model structure for weak 1-complicial sets.
Recall that the action of the left derived functor of $(-)^\flat$ on fibrant objects is the \emph{natural marking}~\cite{lurie:htt} functor $(-)^\natural$ that marks the equivalences, providing explicit (1-categorical) equivalences in both directions between fibration categories of fibrant objects.
Instead of directly comparing fibrant marked semisimplicial sets with quasicategories, it will thus suffice to relate the homotopy theories of marked semisimplicial and marked simplicial sets.

The following development is analogous to the unmarked case.
The omitted proofs are adapted verbatim.

\begin{lemma} \label{maximal-quotient-horn-marked}
The map $\Sk^n(1) \to \Sk^{n+2}(1)$ is a weak equivalence for even $n$.
\end{lemma}

\begin{proof}
By \cref{maximal-marking-pres-w-equiv}.
\end{proof}

\begin{corollary} \label{monoidal-to-cartesian-unit-w-equiv-marked}
The map $\top \to 1$ is a cofibration and a weak equivalence.
\end{corollary}

\begin{proof}
By \cref{maximal-marking-pres-w-equiv}.
\end{proof}

\begin{lemma} \label{simplicial-unit-w-equiv-marked}
The unit $\eta \co \Id \to UL$ is valued in cofibrations that are weak equivalences.
\qed
\end{lemma}

\begin{proof}
As in the proof of \cref{simplicial-unit-w-equiv}, it will suffice to show that the components of $\eta$ on $\Delta^n$ for $n \geq 0$ and $(\Delta^1)^\sharp$ are weak equivalences.
The cases of $\Delta^0$ and $(\Delta^1)^\sharp$ follow from \cref{simplicial-unit-w-equiv} by maximal marking.
The case of $\Delta^n$ for $n \geq 1$ reduces to the one for $n = 0$ by \cref{join-pres-weak-equiv} and monoidality of $L$ and $U$ with respect to the join.
\end{proof}

\begin{corollary} \label{UL-create-w-equiv-marked}
The functor $UL$ creates weak equivalences.
\qed
\end{corollary}

\begin{corollary} \label{U-triv-cofib-to-w-equiv-marked}
$U$ sends anodyne maps to cofibrations that are weak equivalences.
\qed
\end{corollary}

\begin{lemma} \label{U-triv-fib-to-w-equiv-marked}
The functor $U$ sends trivial fibrations to weak equivalences.
\qed
\end{lemma}

The following statements presuppose the model structure on marked simplicial sets, which requires classical logic.

\begin{corollary}[Classical logic] \label{U-pres-w-equiv-marked}
The functor $U$ preserves weak equivalences.
\end{corollary}

\begin{proof}
Using fibrant replacement in the model structure on marked simplicial sets, every weak equivalences admits a morphism in the arrow category consisting of anodyne maps to a trivial fibration.
The claim then follows by \cref{U-triv-cofib-to-w-equiv-marked,U-triv-fib-to-w-equiv-marked,weak-equiv-2-out-of-6-marked}.
\end{proof}

\begin{corollary}[Classical logic] \label{L-create-w-equiv-marked}
The functor $L$ creates weak equivalences.
\qed
\end{corollary}

We conclude as in the unmarked case.

\begin{theorem}[Classical logic] \label{summary-equiv-marked}
In the triple of functors
\[
\xymatrix@C+1.5cm{
  \MS
  \ar@/^2em/[r]^{L}
  \ar@/^-2em/[r]_{R}
  \ar@{{}{ }{}}@/^1em/[r]|{\bot}
  \ar@{{}{ }{}}@/^-1em/[r]|{\bot}
&
  \M
  \ar[l]|{U}
}
\]
relating marked semisimplicial sets and marked simplicial sets where $U$ is the forgetful functor:
\begin{itemize}
\item
the maps $L$ and $U$ are weak equivalences of cofibration categories,
\item
the maps $U$ and $R$, restricted to fibrant objects, are weak equivalences of fibration categories,
\item
both adjunctions descend to equivalences of homotopy categories.
\end{itemize}
\end{theorem}

\begin{proof}
Analogous to~\cref{summary-equiv}.
\end{proof}

Composing with the Quillen equivalence between the model structures on marked simplicial sets and simplicial sets for higher categories, we obtain the following corollary.

\begin{corollary}[Classical logic] \label{summary-equiv-marked-composite}
In the composite of adjunctions
\[
\xymatrix@C+1cm{
  \MSfib
  \ar@/_1em/[r]_-{R}
  \ar@{}[r]|-{\top}
&
  \Mfib
  \ar@/_1em/[l]_-{U}
  \ar@/_1em/[r]_-{(-)^\unmarked}
  \ar@{}[r]|-{\top}
&
  \widehat{\Delta}_{\fib}
  \ar@/_1em/[l]_-{(-)^\natural}
}
\]
relating marked semisimplicial sets and simplicial sets where $U$ and $(-)^\unmarked$ are the forgetful functors:
\begin{itemize}
\item
the left adjoints are weak equivalences of cofibration categories,
\item
the right adjoints, restricted to fibrant objects, are weak equivalences of fibration categories,
\item
the adjunctions descend to equivalences of homotopy categories,
\item
restricted to fibrant objects, the functor $U$ is and the functor $(-)^\unmarked$ has a weakly equivalent replacement $(-)^\natural$ that is a weak equivalence of fibration categories.
\qed
\end{itemize}
\end{corollary}

\bigskip

As in the unmarked case, we have a map $A \times B \to A \otimes B$ natural in marked semisimplicial sets $A$ and $B$.

\begin{lemma} \label{free-simplicial-set-unit-h-equiv-on-fibrant-marked}
For fibrant $X$, the map $\eta_X \co X \to ULX$ is a homotopy equivalence.
\qed
\end{lemma}

\begin{theorem} \label{cartesian-vs-geometric-marked}
Given fibrant marked semisimplicial sets $X$ and $Y$, the canonical comparison map
\[
X \times Y \to X \otimes Y
\]
is a weak equivalence.
\end{theorem}

\begin{proof}
Analogous to the proof of \cref{cartesian-vs-geometric}.
\end{proof}

\subsection{Subdivision}

Everything in this section is an evident extension of the unmarked case.
Let us write $L \dashv U$ for both the free (marked) category on a (marked) semicategory and the free (marked) simplicial set on a (marked) semisimplicial set adjunction.

\subsubsection{Marked simplicial subdivision}

The category of elements construction on simplicial sets extends to a marked category of elements construction
\[
\mcel \co \widehat{\Delta}_\m \to \Cat_\m
\]
by marking those morphisms that $\mcel A$ that correspond to a marked edge in $A$ under evaluation at the last vertex.

The simplicial subdivision functor extends to marked simplicial sets by letting the markings of $\Sd A$ be created via the last vertex map $\Sd A \to A$.
By definition, one thus also has a last vertex map in the marked setting.
Using the same argument as in the unmarked case, this map can classically shown to be a weak equivalence in the model structure on marked simplicial sets we have considered in~\cref{subsection:quasicategories}.

\subsubsection{Marked semisimplicial subdivision}

Recall the semicategory of elements construction
\[
\scel \co \widehat{\Delta_+} \to \SemiCat
.\]
It extends to a \emph{marked semicategory of elements} construction
\[
\smcel \co \MS \to \SemiCat_\m
\]
by marking those morphisms $f \co ([m], x) \to ([n], y)$ in $\smcel(A, W)$ that preserve initial objects under the forgetful map to $\Delta_+$, \ie $f(0) = 0$, or otherwise have $A(g)(y) \in W$ where $g \co [1] \to [n]$ sends 0 to $0$ and 1 to $f(0)$.

Say for a moment that a marked simplicial or semisimplicial set $(A, W)$ is closed under 2-out-of-3 if it lifts against the Leibniz applications of the unit of the maximal marking monad to $h^2_i$ for $0 \leq i \leq 2$; say it is closed under 2-out-of-6 if it is closed under 2-out-of-3 and lifts against the marking saturation inclusion.
Then $\smcel$ preserves taking the closure of markings under 2-out-of-3 and 2-out-of-6, respectively.
In particular, the marked semicategory $\smcel(A, W)$ is closed under 2-out-of-3 and 2-out-of-6, respectively, exactly if $(A, W)$ is.

Analogous to the unmarked case, the marked semisimplicial subdivision functor $\Sd$ is defined as the composite of the marked semicategory of elements functor $\smcel$ followed by the marked semisimplicial nerve.
On underlying semisimplicial sets, it is given by the semisimplicial subdivision functor, so the only new thing is the action on markings.
Recall that colimits in marked semisimplicial sets are computed as colimits in underlying semisimplicial sets with an edge marked if it comes from a marked edge via a colimit coprojection.
We see that marked semisimplicial subdivision is cocontinuous and its action $\sd$ on simplices (including the marked edge) is described by the composite
\[
\xymatrix{
  \Delta_+'
  \ar[r]^-{(\Delta_+)_{/-}}
&
  \SemiCat_\m
  \ar[r]^-{N}
&
  \MS
}
\]
where the first arrow sends $[n] \in \Delta_+'$ to the categorical slice $(\Delta_+)_{/[n]}$ with markings created by the last vertex functor $(\Delta_+)_{/[n]} \to [n]$ of categories.
As in the unmarked case, from this one sees immediately that the free marked simplicial set functor commutes with subdivision.

\subsubsection{Relating $A$ and $\Sd A$}

Everything we have said in the unmarked setting about relating a semisimplicial set with its subdivision via a cospan of weak equivalences lifts to the marked setting.
Given a marked semisimplicial set $A$, we have a cospan
\[
\xymatrix@!C@C-0.6cm{
  A
  \ar@{>->}[dr]^{\sim}_{\eta_A}
&&
  \Sd A 
  \ar[dl]_{\sim}
\\&
  U L A
}
\]
of weak equivalences.
Here, the left leg is a cofibration and weak equivalence by \cref{simplicial-unit-w-equiv-marked}.
The right leg is the composite
\[
\xymatrix{
  \Sd A
  \ar@{>->}[r]^-{\eta_{\Sd A}}_-{\sim}
&
  U L \Sd A
  \ar[r]^-{m}
&
  U L A
}
\]
where $m$ is $U$ applied to the marked simplicial last vertex map on $L A$.
Again we argue that $m$ is a homotopy equivalence when $A$ comes from $\Delta_+'$ and use cocontinuity and the glueing lemma in the standard way to show that $m$ is a weak equivalence for all $A$.

We will now lift the finiteness preserving cospan construction~\eqref{subdivision-mapping-cone} to the marked setting.
For this, we first define a marked version of the indexed join that also combines the join and the marking join in one operation.
The input consists of marked semisimplicial sets $A$ and $B$ with functors
\[
F, F_\m \co \cel A \times \cel B
\]
with a monomorphism $F_\m \to F$ that we write as an inclusion.
The output $A *_{F, F_\m} B$ has as underlying semisimplicial set the indexed join $A^\unmarked *_F B^\unmarked$ with those edges marked that come from a marked edge in $A$ or $B$ or an element $t \in F_\m(([0], x), ([0], y))$ with $x \in A_0$ and $y \in B_0$.

Given a marked semisimplicial set $A$, note that we have subfunctor $\hom_\m$ of the hom-functor of $\cel A$ that returns only the marked morphisms in $\cel_\m A$.
This induces a subfunctor $F_m(A)$ of $F(A^\unmarked)$.
We then obtain a Reedy cofibrant span
\begin{equation} \label{subdivision-mapping-cone-marked}
\begin{gathered}
\xymatrix@!C@C-2.3cm{
  A
  \ar[dr]_{\iota_0}
&&
  \Sd A
  \ar[dl]^{\iota_1}
\\&
  A \star_{F(A^\unmarked), F_\m(A)} \Sd A
}
\end{gathered}
\end{equation}
functorial and cocontinuous in the marked semisimplicial set $A$.

We now prove analogues of \cref{subdivision-mapping-cone-left-leg,subdivision-mapping-cone-right-leg}.
We will only discuss the changes to the unmarked setting.

\begin{lemma} \label{subdivision-mapping-cone-left-leg-marked}
Pushout application of the left leg of~\eqref{subdivision-mapping-cone-marked} sends cofibrations to anodyne maps.
\end{lemma}

\begin{proof}
By cocontinuity, one once again reduces to boundary inclusions and the edge marking inclusion.
The second case is easy, using that the marking satural inclusion is anodyne.

In the first case, we have to make sure that the marked analogue of the map~\eqref{subdivision-mapping-cone-left-leg:1} (with $\star'$ defined in terms of $\star_{F(A^\unmarked), F_\m(A)}$ instead of $\star_F$) is anodyne.
In the presentation of~\eqref{subdivision-mapping-cone-left-leg:1} as a pushout of~\eqref{subdivision-mapping-cone-left-leg:2}, the edge from $\tau_0(k)$ to $\tau_1(0)$ of the simplex $\Delta^k \star \Delta^m$ is mapped to a marked edge in $\Delta^n \star' \Delta^m$ as the last vertex of $f(0)$ is exactly $k$.
But~\eqref{subdivision-mapping-cone-left-leg:2} with that edge marked is the Leibniz join of the boundary inclusion of dimension $k-1$ (omitted if $k = 0$) with the marked left outer inclusion of dimension $m+1$, hence anodyne by \cref{marked-anodyne-join-cof}.
\end{proof}

\begin{lemma} \label{subdivision-mapping-cone-right-leg-marked}
Pushout application of the right leg of~\eqref{subdivision-mapping-cone-marked} sends cofibrations to anodyne maps.
\end{lemma}

\begin{proof}
We once again only discuss the case of a boundary inclusion.

The bottom square of~\eqref{subdivision-mapping-cone-right-leg:1}, with $\star'$ redefined as in the proof of~\eqref{subdivision-mapping-cone-left-leg-marked}, is a pushout of~\eqref{subdivision-mapping-cone-right-leg:2} with edge from $\tau_0(n)$ to $\tau_1(0$) marked (the attaching map $\Delta^0 \to \Sd \Delta^n$ corresponds to the object $([n], \id)$ of $(\Delta_+)_{/[n]}$ and has last vertex $n$) and hence anodyne.

For the top square of~\eqref{subdivision-mapping-cone-right-leg:1}, it will suffice to verify that the pushout corner map in~\eqref{subdivision-mapping-cone-right-leg:3} remains anodyne in marked semisimplicial sets; we claim that it lies in the saturation of inner horn inclusions.
This was the Leibniz join of $\partial\Delta^k \to \Delta^k$ with the nerve of~\eqref{subdivision-mapping-cone-right-leg:3}.
Let us call a semisimplicial map left anodyne if it lies in the weak saturation of inner and left outer horn inclusions.
The Leibniz geometric product of a left anodyne map with a cofibration is again left anodyne, by the analysis of~\cref{anodyne-tensor-cof} given in the proof of~\eqref{marked-anodyne-tensor-cof}.
In particular, the iterated Leibniz geometric product of $\delta_0$ that is the nerve of~\eqref{subdivision-mapping-cone-right-leg:3} is left anodyne.
The claim then follows since Leibniz join with $\partial\Delta^k \to \Delta^k$ from the left sends left outer horn inclusions to inner horn inclusions and preserves inner horn inclusions.
\end{proof}

\begin{corollary} \label{subdivision-mapping-cone-legs-anodyne-marked}
The legs of~\eqref{subdivision-mapping-cone-marked} are anodyne for all semisimplicial sets $A$.
\end{corollary}

\begin{proof}
By \cref{subdivision-mapping-cone-left-leg-marked,subdivision-mapping-cone-right-leg-marked}.
\end{proof}

\bibliographystyle{alpha}

\input{mss.bbl}

\end{document}

%% file: mss.bbl
\newcommand{\noopsort}[1]{}